\documentclass[11pt,oneside,reqno]{amsart}

\usepackage{fullpage}
\usepackage{amsthm,amsfonts,amssymb,amsmath,amsxtra,mathtools}
\usepackage[dvipsnames]{xcolor}
\usepackage[all]{xy}
\SelectTips{cm}{}
\usepackage{tikz}
\usepackage{verbatim}

\usepackage{mathrsfs}
\usepackage{booktabs,makecell}
\usepackage{diagbox}

\RequirePackage{xspace}
\RequirePackage{etoolbox}
\RequirePackage{varwidth}
\RequirePackage{enumitem}
\RequirePackage{tensor}
\RequirePackage{mathtools}
\RequirePackage{longtable}
\RequirePackage{multirow}
\RequirePackage{makecell}
\RequirePackage{bigints}

\usepackage[bookmarks = true]{hyperref}
\hypersetup{  
	colorlinks=true,
	linkcolor=blue,
	citecolor=green,
	urlcolor=magenta,
}
\usepackage[alphabetic]{amsrefs}

\makeatletter

\def\@buildmath#1{%
  \expandafter\def\csname mb#1\endcsname{\mathbb{#1}}%
  \expandafter\def\csname mc#1\endcsname{\mathcal{#1}}%
  \expandafter\def\csname mf#1\endcsname{\mathfrak{#1}}%
  \expandafter\def\csname rm#1\endcsname{\mathrm{#1}}%
  \expandafter\def\csname bf#1\endcsname{\mathbf{#1}}%
}

\def\@buildmathletters#1{%
  \ifx#1\relax
    \expandafter\@gobble
  \else
    \@buildmath{#1}%
    \expandafter\@buildmathletters
  \fi
}

\@buildmathletters ABCDEFGHIJKLMNOPQRSTUVWXYZabcdefghijklmnopqrstuvwxyz\relax

\makeatother

\newcommand{\sF}{\ensuremath{\mathscr{F}}\xspace}

\DeclareMathOperator{\Gal}{Gal}

\DeclareMathOperator{\End}{End}

\DeclareMathOperator{\Lie}{Lie}

\DeclareMathOperator{\tr}{Tr}

\newcommand{\U}{\mathrm{U}}

\newcommand{\ep}{\varepsilon}

\newenvironment{altenumerate}
   {\begin{list}
      {(\theenumi) }
      {\usecounter{enumi}
       \setlength{\labelwidth}{0pt}
       \setlength{\labelsep}{0pt}
       \setlength{\leftmargin}{0pt}
       \setlength{\itemsep}{\the\smallskipamount}
       \renewcommand{\theenumi}{\arabic{enumi}}
      }}
   {\end{list}}
\newenvironment{altenumerate2}
   {\begin{list}
      {\textup{(\theenumii)} }
      {\usecounter{enumii}
       \setlength{\labelwidth}{0pt}
       \setlength{\labelsep}{0pt}
       \setlength{\leftmargin}{2em}
       \setlength{\itemsep}{\the\smallskipamount}
       \renewcommand{\theenumii}{\roman{enumii}}
      }}
   {\end{list}}

\newenvironment{altitemize}
   {\begin{list}
      {$\bullet$}
      {\setlength{\labelwidth}{0pt}
	   \setlength{\itemindent}{5pt}
       \setlength{\labelsep}{5pt}
       \setlength{\leftmargin}{0pt}
       \setlength{\itemsep}{\the\smallskipamount}
      }}
   {\end{list}}

\newcommand{\isoarrow}{%
   \ifbool{@display}{\overset{\sim}{\longrightarrow}}{\xrightarrow\sim}%
   }

\newcommand{\loc}{\text{loc}}
\newcommand{\ad}{\text{ad}}

\DeclareFontFamily{U}{matha}{\hyphenchar\font45}
\DeclareFontShape{U}{matha}{m}{n}{
      <5> <6> <7> <8> <9> <10> gen * matha
      <10.95> matha10 <12> <14.4> <17.28> <20.74> <24.88> matha12
      }{}
\DeclareSymbolFont{matha}{U}{matha}{m}{n}
\DeclareFontFamily{U}{mathx}{\hyphenchar\font45}
\DeclareFontShape{U}{mathx}{m}{n}{
      <5> <6> <7> <8> <9> <10>
      <10.95> <12> <14.4> <17.28> <20.74> <24.88>
      mathx10
      }{}
\DeclareSymbolFont{mathx}{U}{mathx}{m}{n}

\DeclareMathSymbol{\obot}         {2}{matha}{"6B}

\newcommand{\bX}{\mathbb{X}}

\newcommand{\bD}{\mathbb{D}}
\newcommand{\bV}{\mathbb{V}}

\newcommand{\bF}{\mathbf{F}}

\newcommand{\Gr}{\mathrm{Gr}}

\newcommand{\Fil}{\mathrm{Fil}}

\newcommand{\bL}{\mathbb{L}}

\newcommand{\Q}{\mathbb{Q}}
\newcommand{\Z}{\mathbb{Z}}

\newcommand{\F}{\mathbb{F}}

\newcommand{\Oo}{\mathcal{O}}
\newcommand{\cV}{\mathcal{V}}
\newcommand{\cZ}{\mathcal{Z}}

\newcommand{\cY}{\mathcal{Y}}

\newcommand{\cL}{\mathcal{L}}

\newcommand{\cG}{\mathcal{G}}
\newcommand{\cU}{\mathcal{U}}
\newcommand{\cN}{\mathcal{N}}

\newcommand{\cP}{\mathcal{P}}

\newcommand{\cM}{\mathcal{M}}

\newcommand{\cF}{\mathcal{F}}

\newcommand{\id}{\mathrm{id}}

\newcommand{\GL}{\mathrm{GL}}
\newcommand{\Sp}{\mathrm{Sp}}

\newcommand{\Nilp}{\mathrm{Nilp}\, }
\newcommand{\Ros}{\mathrm{Ros}}

\newcommand{\Spec}{\mathrm{Spec}\, }
\newcommand{\Spf}{\mathrm{Spf}\, }

\newcommand{\red}{\mathrm{red}}

\newcommand{\p}{q}

\newcommand{\sfe}{e}
\newcommand{\sff}{f}
\newcommand{\sfs}{{s}}

\newcommand{\sfw}{{w}}
\newcommand{\sfg}{{g}}

\newcommand{\sfF}{\mathcal{F}}

\DeclareMathAlphabet{\mathdist} {U}{BOONDOX-calo}{b}{n}
\newcommand{\tth}{\mathdist h}
\newcommand{\ttr}{\mathdist r}
\newcommand{\ttt}{\mathdist t}
\newcommand{\tts}{\mathdist s}
\newcommand{\ttn}{\mathdist n}

\newcommand{\dF}{\dot{\bF}}

\newcommand{\Fl}{\cF\ell}
\newcommand{\uF}{{\mathrm{F}}}
\newcommand{\uV}{{\mathrm{V}}}
\newcommand{\ov}{\overline}

\newcommand{\pr}{\mathrm{pr}}

\newtheorem{theorem}{Theorem}[section]
\newtheorem{corollary}[theorem]{Corollary}
\newtheorem{proposition}[theorem]{Proposition}
\newtheorem*{proposition*}{Proposition}
\newtheorem{lemma}[theorem]{Lemma}

\theoremstyle{definition}
\newtheorem{definition}[theorem]{Definition}
\newtheorem{example}[theorem]{Example}
\newtheorem{remark}[theorem]{Remark}
\numberwithin{equation}{section}

\usepackage{tikz-cd}

\title{
The basic locus of ramified unitary Shimura varieties of signature $(n-1,1)$ at maximal vertex level}
 
\author{Qiao He}
\address{Department of Mathematics, Columbia University, 2990 Broadway, New York, NY 10027, USA}
\email{qh2275@columbia.edu } 
\author{Yu LUO}
\address{Massachusetts Institute of Technology, Department of Mathematics, 77 Massachusetts Avenue, Cambridge, MA 02139, USA}
\email{yuluo@mit.edu}
\author{Yousheng Shi}
\address{School of Mathematical Sciences, Zhejiang University, 866 Yuhangtang Rd, Hangzhou, 310058, P.R. China}
\email{0023140@zju.edu.cn}

\date{\today}

\begin{document}
\begin{abstract}
We construct the Bruhat--Tits stratification of the reduced locus of the ramified unitary Rapoport--Zink space of signature $(n-1,1)$, with the level being the stabilizer of a vertex lattice. We develop the local model theory for Bruhat--Tits strata, proving their normality and Cohen--Macaulayness, and provide precise dimension formulas. Additionally, we establish an explicit scheme-theoretical isomorphism between Bruhat--Tits strata and Deligne--Lusztig varieties.

\end{abstract}
\maketitle
\tableofcontents

\section{Introduction}
In this paper, we study the reduced locus of basic unitary Rapoport--Zink spaces of signature $(n-1,1)$.
This study contributes to the broader theory of reduction of integral models of Shimura varieties.
For historical context and background on this subject, we refer the reader to \cite{Vollaard}.

The reduced locus of basic unitary Rapoport--Zink spaces was first studied by Vollaard and Wedhorn \cites{Vollaard,Vollaard-Wedhorn} for the unramified unitary group with signature $(1,n-1)$ (which is identical to the case of signature $(n-1,1)$ by conjugation) with hyperspecial level. 
Their work uncovered a Bruhat--Tits stratification of the reduced locus, where the indexing set relates to the Bruhat--Tits building of the inner form of the unitary group. 
Furthermore, they showed that each stratum admits an explicit geometric description as a disjoint union of Deligne--Lusztig varieties. This elegant structure has been extended to more general settings in recent years.

When the quadratic extension is ramified, which is the case we are interested in, the reduced locus was studied by Rapoport--Terstiege--Wilson \cite{RTW} for self-dual levels and by Wu \cite{Wu16} for exotic smooth levels. In this paper, we extend the study of the Bruhat--Tits stratification to all maximal vertex levels, defined as levels constructed as the stabilizer of a vertex lattice. 

Basic unitary Rapoport--Zink spaces and their Bruhat--Tits stratifications have important applications in arithmetic geometry. For instance, they relate to Kudla--Rapoport cycles and CM cycles, and play an important role in the proof of the Kudla--Rapoport conjecture \cites{LZ,LL2}, as well as some cases of the arithmetic fundamental lemma \cites{RTZ,Li-Zhu}. Additionally, they can also be used to determine the geometric structure of the supersingular locus of unitary Shimura varieties \cites{Vollaard-Wedhorn,RTW}. 
In those applications, it is important to have a moduli-theoretic and scheme-theoretic treatment of the Bruhat--Tits stratification, see Remark \ref{rmk:comp-to-group-way}. In this sense, our paper is new even when restricted to the self-dual level or exotic smooth level.

Let us also mention some works regarding the Bruhat--Tits stratification in different settings. These include studies on the orthogonal group \cites{HP1,HP2,HZBasic}, the quaternionic unitary group \cite{Wang-quaternionic}, higher signatures \cites{FHI24, Stefania24}, and general Coxeter-type group data \cite{GHN2024}. 
Since the first version of this draft appeared, Muller \cite{Muller} has extended the Bruhat--Tits stratification to the unramified unitary group of signature $(n-1,1)$ for arbitrary parahoric level structures.

\vspace{3mm}

Let $F/F_0$ be a ramified quadratic extension of $p$-adic fields ($p\geq 3)$ with uniformizers $\pi$ and $\pi_0$ respectively, satisfying $\pi^2=\pi_0$. Let $V$ be a hermitian space over $F$ of dimension $n$ with Hasse invariant $\ep$. The \emph{basic Rapoport--Zink space} $\cN_{n,\ep}^{[h]}=\cN_{n,\ep}^{[h]}(1,n-1)$ is a formal scheme that relates to a vertex lattice $L\subseteq V$ of type $h$ (along with a conjugacy class of geometric cocharacter $\mu_{(1,n-1)}$ and the minimal element in $B(\U(V),\mu_{(1,n-1)}^{-1})$\footnote{In \cite{RZ}, the moduli functor is related to the group $\mathrm{GU}(V)$. We are working on the height $0$ part of the Rapoport--Zink space. Such a part is expected to relate to the unitary group.}, see \S \ref{sec:RZ-def}.
The main result of this paper is a description of the reduced locus $\cN^{[h]}_{n,\ep,\red}$ in terms of vertex lattices in the nearby hermitian space $\bV$, which is a hermitian space over $F$ of dimension $n$ and Hasse invariant $-\ep$.

We denote by $\cL_{\cZ}$ (resp. $\cL_{\cY}$) the set of vertex lattices $\Lambda\subset \bV$ of type $\geq h$ (resp. $\leq h$).
For each $\Lambda_1\in\cL_{\cZ}$ and $\Lambda_2\in\cL_{\cY}$, we define closed subschemes $\cZ(\Lambda_1)$ and $\cY(\Lambda_2^\sharp)$ of the special fiber of the Rapoport--Zink space $\ov{\cN}_{n,\ep}^{[h]}$ (see Definition \ref{def:YZ-cycles}). The main result of the paper is the decomposition of the underlying reduced scheme into these subschemes:
$$\cN_{n,\ep,\red}^{[h]}=\Bigl(\bigcup_{\Lambda_1\in\cL_\cZ}\cZ(\Lambda_1)\Bigr)\cup \Bigl(\bigcup_{\Lambda_2\in\cL_\cY}\cY(\Lambda_2^\sharp)\Bigr).$$
We also show that those subschemes satisfy nice inclusion relations corresponding to lattice inclusion (Theorem \ref{thm:main}). 
These relations enable us to construct the \emph{Bruhat--Tits stratification}. The proof relies on Dieudonn\'e theory to transform the problem into a semi-linear algebra problem, then it boils down to the crucial lemma, which is originally established in \cites{Vollaard-Wedhorn,RTW} and later generalized in \cites{KR-BG,HZBasic}.
\begin{figure}
    \centering
\begin{tikzpicture}[x=1cm, y=1cm, >=stealth, scale=0.6, transform shape]

    \def\curvecolor{teal}
    \def\linecolor{blue!80!black}
    \draw[thick, \curvecolor] (-5, 3.5) .. controls (-3, 4.5) .. (-1, 4.5);
    \draw[thick, \curvecolor] (-4.8, 1.5) .. controls (-3, 2.5) .. (-1, 2.5);
    \draw[thick, \linecolor] (-4.5, 1.7) -- (-4.2, 3.7);
    \draw[thick, \linecolor] (-3.8, 2.0) -- (-3.5, 4.1);
    \draw[thick, \linecolor] (-3.1, 2.2) -- (-2.8, 4.3);
    \draw[thick, \linecolor] (-2.4, 2.4) -- (-2.1, 4.4);
    \draw[thick, \linecolor] (-1.7, 2.5) -- (-1.4, 4.5);
    \node at (0.2, 3) {$\coprod$};
    \node at (0.2, 1.5) {\Large $\mathcal{N}_{4, \mathrm{red}}^{[4,2]}$};
    \draw[thick, \curvecolor] (1.5, 3.5) .. controls (3.5, 4.5) .. (5.5, 4.5);
    \draw[thick, \curvecolor] (1.3, 1.5) .. controls (3.3, 2.5) .. (5.3, 2.5);
    \draw[thick, \linecolor] (1.6, 1.7) -- (1.9, 3.7);
    \draw[thick, \linecolor] (2.3, 2.0) -- (2.6, 4.1);
    \draw[thick, \linecolor] (3.0, 2.2) -- (3.3, 4.3);
    \draw[thick, \linecolor] (3.7, 2.4) -- (4.0, 4.4);
    \draw[thick, \linecolor] (4.4, 2.5) -- (4.7, 4.5);
    \draw[thick, \curvecolor] (-8, -1) .. controls (-6, 0) .. (-4, 0);
    \draw[thick, \curvecolor] (-7.8, -3) .. controls (-5.8, -2) .. (-3.8, -2);
    \draw[thick, \linecolor] (-7.5, -2.8) -- (-7.2, -0.8);
    \draw[thick, \linecolor] (-6.8, -2.5) -- (-6.5, -0.4);
    \draw[thick, \linecolor] (-6.1, -2.3) -- (-5.8, -0.2);
    \draw[thick, \linecolor] (-5.4, -2.1) -- (-5.1, -0.1);
    \draw[thick, \linecolor] (-4.7, -2.0) -- (-4.4, 0.0);
    \node at (-5.5, -3.7) {\Large $\mathcal{N}_{4, \mathrm{red}}^{[2]}$};
    \draw[thick, \curvecolor] (3.5, -2.5) .. controls (4.5, -2.0) .. (5.5, -1.5);
    \node at (6.5, -2.0) {$\coprod$};
    \draw[thick, \linecolor] (7.5, -2.8) -- (8.3, -1.0);
    \node at (6.5, -3.7) {\Large $\mathcal{N}_{4, \mathrm{red}}^{[4]}$};
    \draw[->, thick] (-2, 1.3) -- (-3.5, 0.2);
    \draw[->, thick] (4.5, 1) -- (5.3, 0);
    \draw[->, thick] (5.1, 1) -- (5.9, 0);

\end{tikzpicture}
    \caption{\it\small When $V$ is split of $\dim V=n=4$ and $t=2,4$. In this case, only $\mcY$-strata appear. The maximal closed $\mcY$-strata are both indexed by self-dual lattice $\Lambda$. In     $\mcN_{4,\red}^{[4]}$, each strata is isomorphic to disjoint of two $\mbP^1$. 
    In
    $\mcN^{[2]}_{4,\red}$, each $\mcY(\Lambda^\sharp)$ is isomorphic to a ruled surface $\mbP^1\times \mbP^1$ inside $\mbP(\Lambda/\pi\Lambda^\vee)$. $\mcN_{4,\red}^{[4,2]}$ is a trivial double cover of $\mcN_{4,\red}^{[2]}$ and the maps $\mcN_{4,\red}^{[4,2]}\to \mcN_{4,\red}^{[4]}$ are given by the map to the two parameterizing projective lines.
    }
    \label{fig:n=4-split}
\end{figure}
\begin{remark}
We exclude the case $p=2$ since the moduli spaces in the ramified setting differ significantly from those in the odd prime case, see \cite{Yang25}. 
But our methods, such as the use of relative displays, are applicable to $p=2$, since all $p$-divisible groups considered in the basic locus are biformal (see \S \ref{sec:strict-module} for definition). 
\end{remark}

\vspace{2mm}

We note that by definition, $\cZ(\Lambda_1)$ and $\cY(\Lambda_2^\sharp)$ are closed formal subschemes of the special fiber. We prove that these formal schemes are representable by reduced schemes. 
The proof proceeds by constructing a local model diagram
\begin{equation*}
\begin{tikzcd}
    &\widetilde{\cZ}(\Lambda_1)\arrow[dl,"\varphi"']\arrow[dr,"\pi"]&\\
    \cZ(\Lambda_1)&&\ov{\cM}^{[h]}_{n}(t),
\end{tikzcd}
\end{equation*}
where $\ov{\cM}^{[h]}_{n}(t)$ denotes the \emph{strata local model} of type $t=t(\Lambda_1)$ (Definition \ref{def:local-model}). The morphisms $\varphi$ and $\pi$ are smooth of equal dimension.
In particular, the $\cZ$-stratum $\cZ(\Lambda_1)$ shares many common geometric properties with its strata local model.  Similar local model diagrams exist for both the $\cY$-stratum $\cY(\Lambda_2^\sharp)$ and their intersection $\cZ(\Lambda_1)\cap\cY(\Lambda_2^\sharp)$.

Similar to the local models associated to Shimura varieties and Rapoport--Zink spaces, strata local models are defined by purely linear algebraic data, enabling simpler analysis than that of Bruhat--Tits strata. By embedding strata local models into partial affine flag varieties and relating them to Schubert varieties, we obtain the reducedness of strata local models and, consequently, of the corresponding Bruhat--Tits strata. By explicit computations, we establish numerous geometric properties of the Bruhat--Tits strata, including normality, Cohen--Macaulayness, and precise dimension formulas. A complete list of these properties is presented in Theorem \ref{thm:LM-results}. It remains an interesting question to characterize strata local models and deduce more geometric properties using group-theoretical methods.

\begin{remark}
The strata local model diagram is not compatible with the local model diagrams for Shimura varieties and Rapoport--Zink spaces. Hence it cannot be derived directly from standard local model theory.
See Remark \ref{rmk:not compatible with other lm diagram} for a discussion.
\end{remark}

The reducedness of the Bruhat--Tits strata is useful for arithmetic intersection, see Remark \ref{rmk:comp-to-group-way}.
Prior to our work, the reducedness of Bruhat--Tits strata was established only in specific cases with smooth structure: see Vollaard--Wedhorn \cite[Thm. 3.10]{Vollaard-Wedhorn}
and Li--Zhu \cite[Cor. 3.2.3]{Li-Zhu} for unramified hyperspecial level and He--Li--Shi--Yang \cite[\S 3]{HLSY} for Kr\"amer models. 
It is expected that local model diagrams can be constructed in a more general setting. 
\begin{figure}
    \centering
\begin{tikzpicture}[x=1cm, y=1cm, >=stealth, scale=0.6, transform shape]
    \draw (-5.0, -1.8) .. controls (0, -0.8) .. (5.0, -1.8);
    \node[anchor=north] at (4.0, -1.8) {\Large $\mathbb{P}^1 \times \mathbb{P}^1$};
    \node[anchor=south] at (-2.5, 1.2) {\Large $\mathbb{P}^3$};
    \draw (-2.5, -0.1) ellipse (0.4 and 1.2);
    \node at (0, 1.8) {\Large $\mathbb{P}^3$};
    \draw (0, 0.15) ellipse (0.4 and 1.2);
    \node[anchor=south] at (2.5, 1.2) {\Large $\mathbb{P}^3$};
    \draw (2.5, -0.1) ellipse (0.4 and 1.2);
\end{tikzpicture}
    \caption{\it \small When $V$ is nonsplit of $\dim V=n=4$, and $t=2$. In this case, the biggest $\mcY$-strata are isomorphic to ruled surfaces, and $\mcZ$-strata are isomorphic to $\mbP^3$, they intersect at rational points in the ruled surfaces, like balloons and grounds.}
    \label{fig:n=4-nonsplit}
\end{figure}

\vspace{2mm}

Finally, we establish an explicit isomorphism between Bruhat--Tits strata and Deligne--Lusztig varieties, extending similar results from \cites{Vollaard-Wedhorn,RTW,Wu16}.
The relationship between the reduced locus of Bruhat--Tits strata and Deligne--Lusztig varieties has been studied group-theoretically in \cite{GHN19}. To be more precise, by passing to perfect schemes, the special fiber of Rapoport--Zink spaces are identical to affine Deligne--Lusztig varieties (ADLV), which are analogues of the Deligne--Lusztig varieties over local fields, and can be described in purely group-theoretical language. For a unitary group with signature $(n-1,1)$, the ADLVs are fully Newton-Hodge decomposable (\cite[\S 1.3]{GHN19}), and the EKOR strata decompose into weak Bruhat--Tits strata (\cite[\S 2.4]{GHN2024}) where each stratum is isomorphic to the perfection of a Deligne--Lusztig variety. When the level structure is of the Coxeter type, which is the case in \cites{Vollaard-Wedhorn,RTW,Wu16}, the paper \cite{GHN2024} shows that the weak Bruhat--Tits stratification coincides with the classical Bruhat--Tits stratification.

\begin{remark}\label{rmk:comp-to-group-way}
While the group-theoretical approach provides a comprehensive framework, it has some limitations: it is topological, and may lose geometric and arithmetic information. 
For instance, intersection multiplicities can change when passing to the reduced scheme or its perfection.
Furthermore, it turns out that the lattice descriptions of the Bruhat--Tits stratification are more useful in some arithmetic applications, such as in the description of special cycles, see for instance, \cite[\S 4]{KR1}, \cite[\S 2.7]{LZ}, \cite[Cor. 2.30]{LL2} and \cite[Prop. 3.20]{HLSY}.
\end{remark}
\begin{remark}
It is known that the original morphism constructed by Vollaard and Wedhorn \cite{Vollaard-Wedhorn} is not an isomorphism, as it differs by a Frobenius twist. We resolve this discrepancy and explain this appearance of the Frobenius twist, see Remark \ref{rmk:correct-VW} for details.
\end{remark}

We also study the geometry of Deligne--Lusztig varieties in our setting. In contrast to cases of the Coxeter type, several new phenomena appear. 
Given the technical complexity and extensive notation involved, we refer the reader to \S \ref{sec:relate-DLV} for the complete result.
We wish to highlight a new phenomenon specific to our case.
In \cite[Prop. 5.3]{RTW}, a key step involved is studying the space
$$
S_\Lambda=\{
\cV\subset(\Lambda^\sharp/\Lambda)\otimes_{\F_q}\F\mid \cV\text{ is Lagrangian, and}\, \cV\cap\Phi(\cV)\stackrel{\leq 1}{\subseteq}\cV
\}.
$$
Here $\Lambda^\sharp/\Lambda$ is equipped with a natural symplectic structure (see \S \ref{sec:notation}) and extends to $\F$, and $\Phi$ is the Frobenius action on this space.
In loc. cit, it is shown that the space $S_\Lambda$ admits a first-step decomposition:
\begin{equation*}
S_\Lambda=X_P(\id)\amalg X_P(w),
\end{equation*}
where $X_P(\id)$ corresponds to $\Phi$-stable subspaces and $X_P(w)$ to the non-$\Phi$-stable ones.
We refer the reader to \cite[Prop. 5.3]{RTW} for detailed definitions of the notation.

In the non-Coxeter type case, the subspaces $\cV$ parameterized by $S_\Lambda$ are totally isotropic but not Lagrangian. The first-step decomposition consists of three components:
\begin{equation*}
S_\Lambda=X_P(\id)\amalg X_P(w)\amalg X_P(w'),
\end{equation*}
where $X_P(\id)$ corresponds to $\Phi$-stable subspaces, while $X_P(w)$ (resp. $X_P(w')$) corresponds to non-$\Phi$-stable subspaces $\cV$ such that $\cV+\Phi(\cV)$ is isotropic (resp. non-isotropic).
This decomposition relates to the Kottwitz--Rapoport strata; see Remark \ref{rmk:relation-to-KR}.
Due to the additional piece $X_P(w')$, the structure of $S_\Lambda$ becomes more complicated than in the Coxeter type case.
Nevertheless, we provide a complete description of the stratification in terms of Deligne--Lusztig varieties for any maximal vertex level, see \S \ref{sec:relate-DLV} for a list of results.

\begin{remark}
By definition \cite[Def. 2.4]{GHN2024}, the Coxeter type is only defined in the fully Newton--Hodge decomposable situation. The case we are working on is still fully Newton--Hodge decomposable. There are also works in the different direction, where the Bruhat--Tits stratification is studied beyond the fully Newton--Hodge decomposable situation.
\end{remark}

\vspace{2mm}
We conclude the introduction by summarizing some highlights of our paper. 
First, we define and study local models of Bruhat--Tits strata, this is the first case in which the reducedness of such strata is proven outside the smooth setting.
Second, the crucial lemma in our case is more subtle and intrinsic. Indeed, a correct formulation and proof of this lemma is the starting point of our paper. 
Next, the decomposition of Deligne--Lusztig varieties is more complicated. Unlike the unramified case, we cannot provide a uniform treatment of Deligne--Lusztig varieties due to the lack of a dual trick.
Finally, we provide a more detailed treatment of the construction of the isomorphism between Bruhat--Tits strata and Deligne--Lusztig varieties, resolving an issue in the original construction found in \cite{Vollaard-Wedhorn}. We also compare the differences between EKOR-strata in the group setting \cite{GHN2024} and the Bruhat--Tits stratification in our case. To our knowledge, this is the first paper to relate the decomposition to Kottwitz--Rapoport strata.

\subsection{Acknowledgment}
The authors would like to thank the anonymous referee for helpful suggestions.
Y. Luo would like to thank Michael Rapoport for his encouragement and comments on this paper. Q. He would like to thank Rong Zhou for collaboration and various discussions that were very helpful for the current project. We would also like to thank Xuhua He, Chao Li, Sian Nie, Qingchao Yu and Wei Zhang for helpful discussion and comments.

Q. He and Y. Luo would like to thank the Institute for Advanced Study in Mathematics at Zhejiang University and the Morningside Center of Mathematics at Chinese Academy of Sciences for their hospitality during the Summer 2024 when part of this work was done. Y. Shi is supported by the start-up grant and Qizhen Grant of Zhejiang University.

\subsection{Notation}\label{sec:notation}
\begin{altitemize}
\item Let $F/F_0$ be a ramified quadratic extension of a $p$-adic field, with $p\ge 3$. For any element $a\in F$, we denote its Galois conjugate over $F_0$ by $\bar{a}$.
Let $\pi\in F$ and $\pi_0\in F_0$ be uniformizers satisfying $\pi^2=\pi_0$. We denote by $\F_q$ their common residue field and by $\F$ its algebraic closure.

\item For any $p$-adic local field $K$, we denote by $\breve{K}$ the completion of the maximal unramified extension $K$.
Let $\sigma\in\Gal(\breve{F}_0/F_0)$ be the (arithmetic) Frobenius element. We fix an embedding of rings $i_0: O_{F_0}\hookrightarrow  O_{\breve{F}_0}$ and an embedding $i: O_{F}\hookrightarrow  O_{\breve{F}}$ extending $i_0$. Finally, we define $\bar{i}: O_F\rightarrow O_{\breve{F}}$ by $a\mapsto i(\bar{a})$.

\item Let $(V,h)$ be an $F/F_0$-hermitian space and $\Lambda\subset V$ be a lattice over $F/F_0$. The hermitian dual of $\Lambda$ is defined as
\begin{equation*}
    \Lambda^\sharp:=\{v\in V\mid h(x,\Lambda)\in O_F\}.
\end{equation*}
We call $\Lambda$ a \emph{vertex lattice} if it satisfies
\begin{equation*}
\pi\Lambda^\sharp\subseteq \Lambda\subseteq\Lambda^\sharp.
\end{equation*}
The \emph{type} of a vertex lattice $\Lambda$, denoted $t(\Lambda)$, is defined as $\dim(\Lambda^\sharp/\Lambda)$.

\item For two hermitian lattices $\Lambda_1\subset \Lambda_2$, we use $\Lambda_1\overset{\leq r}{\subseteq}\Lambda_2$ (resp. $\Lambda_1\overset{r}{\subseteq}\Lambda_2$) to indicate that $[\Lambda_1:\Lambda_2]\leq r$ (resp. $[\Lambda_1:\Lambda_2]= r$).

\item For any vertex lattice $\Lambda\subset V$, we define its base change to $\breve{F}$ as $\breve{\Lambda}:=\Lambda\otimes_{O_{F}}O_{\breve{F}}$.

\item Given a vertex lattice $\Lambda\subset V$, we define two $\F_q$-vector spaces:
\begin{equation*}
    V_\Lambda:=\Lambda^{\sharp}/\Lambda,
    \quad\text{and}\quad
    V_{\Lambda^{\sharp}}=\Lambda/\pi\Lambda^\sharp.
\end{equation*}

The space $V_\Lambda$ carries a natural symplectic structure given by
\begin{equation*}
\langle\,,\,\rangle:V_\Lambda\times V_\Lambda\to \F_q,\quad \langle x,y\rangle\mapsto \pi h(\widetilde{x},\widetilde{y})\mod \pi,
\end{equation*}
where $\widetilde{x}$ and $\widetilde{y}$ denote arbitrary lifts of $x$ and $y$ to $\Lambda^\sharp$, resp.

The space $V_{\Lambda^\sharp}$ carries a natural symmetric structure given by
\begin{equation*}
    (\,,\,):V_{\Lambda^\sharp}\times V_{\Lambda^\sharp}\to \F_q,\quad (x,y)\mapsto h(\widetilde{x},\widetilde{y})\mod \pi.
\end{equation*}
where $\widetilde{x}$ and $\widetilde{y}$ denote arbitrary lifts of $x$ and $y$ to $\Lambda$, resp. Standard arguments show that these two forms are non-degenerate. In particular, the space $V_\Lambda$ is always even dimensional, hence the type $t(\Lambda)$ of a vertex lattice $\Lambda$ is an even number in the ramified case. See also \cite[Lem. 3.2]{RTW}.

\item Let $\Nilp_{O_{\breve{F}}}$ denote the category of $O_{\breve{F}}$-schemes $S$ where $\pi$ is locally nilpotent on $S$.
For any such scheme $S$, we denote its special fiber $S\times_{\Spf O_{\breve{F}}}\F$ by $\bar{S}$.
\end{altitemize}

\section{Statements of the main results}
Let $F/F_0$ be a ramified quadratic extension of $p$-adic fields ($p\geq 3$) with uniformizers $\pi$ and $\pi_0$ respectively, satisfying the relation $\pi^2=\pi_0$. Denote by $\F_q$ the residue field of $F$ and let $\F=\ov{\F}_q$ be its algebraic closure. Let $\breve{F}$ be the completion of the maximal unramified extension of $F$. To save notation, we write $O:=O_{F_0}$.

\subsection{Bruhat--Tits strata}
We will fix a framing object and define Rapoport--Zink (RZ) spaces, see \S \ref{sec:RZ-def} for details.
Let $(\bX,\iota_{\bX},\lambda_{\bX})$ be a fixed supersingular hermitian $O_{F_0}$-module of rank $n$ and type $h$ with signature $(n-1,1)$ (Definition \ref{def:basic-triple}). Here, ``supersingular'' means that the relative isocrystal has all relative slopes $1/2$. To this framing object we associate a hermitian space $\bV$ of dimension $n$ over $F$ (see \S \ref{sec:framing}). 
The hermitian space $\bV$ carries a $\sigma$-linear morphism $\mathrm{F}$ and a $\sigma^{-1}$-linear morphism $\mathrm{V}$, both of which come from the relative Dieudonn\'e theory. The action of $\iota_{\bX}(\pi)$ on $\bX$ induces an action $\Pi:\bV\to \bV$. Let $\ep=\ep(\bV):=-\mathrm{Hasse}(\bV)$ denote the negative of its Hasse invariant.

The wedge RZ space $\cN^{[h],\wedge}_{n,\ep}$ associated to this framing object is a formal scheme over $\Spf O_{\breve{F}}$. It represents the moduli functor that assigns to each $S$ in $\Nilp O_{\breve{F}}$ the set of isomorphism classes of tuples $(X,\iota,\lambda;\rho)$, where $(X,\iota,\lambda)$ is a hermitian $O_F$-module over $S$ (in the sense of Definition \ref{def:basic-triple}) and
\begin{equation*}
    \rho:X\times_S \ov{S}\to \bX\times_{\F}\ov{S}
\end{equation*}
is an $O_F$-linear quasi-isogeny of height $0$ over the special fiber $\ov{S}$. The RZ space $\cN_{n,\ep}^{[h]}$ is defined as the flat closure of $\cN_{n,\ep}^{[h],\wedge}$. We denote by $\ov{\cN}_{n,\ep}^{[h]}$ the special fiber of $\cN_{n,\ep}^{[h]}$ and by $\cN_{n,\ep,\red}^{[h]}$ its reduced subscheme.

\begin{remark}
We are working on the RZ space that corresponds to the height $0$ part of the one defined in \cite{RZ}, see \cite[Rem. 3.6]{RSZ-Duke}. It is connected unless it is in the $\pi$-modular case, see \cite[Cor. 5.3.6]{LRZ}, in the latter case, the RZ space has two connected components. This will also be proved independently in this paper.
\end{remark}

We refer the reader to \S \ref{sec:reviewdisplays} and \S \ref{sec:strict-module} for the $O$-Witt vector and for (covariant) relative Dieudonn\'e theory.
Write $O=O_{F_0}$.
Let $\kappa$ be any perfect field over $\F$, we have the $O$-Witt vector ring $W_O(\kappa)$. By relative Dieudonn\'e theory, any $(X,\iota,\lambda;\rho)\in \cN_{n,\ep}^{[h],\wedge}$ defines an $O_F\otimes_O W_O(\kappa)$-lattice, which is stable under the action of $\mathrm{F}$ and $\mathrm{V}$. 
Recall the following lemma:
\begin{lemma}[\protect{\cite[Lem. 6.1]{RTW}}]\label{lem:cycles_stable}
Let $\Lambda$ be a vertex lattice in $\bV$ and let $\kappa$ be any perfect field over $\F$. Let $M\subset\Lambda\otimes W_O(\kappa)$ be a $O_{F}\otimes_{O} W_O(\kappa)$-lattice such that $M\subseteq M^\sharp$. Then $M$ and $M^\sharp$ are stable under $\uF,\uV$ and $\Pi:=\pi\otimes 1$.
\end{lemma}
\begin{proof}
In \cite[Lem. 6.1]{RTW}, the statement is proven for the Witt vector ring $W(\kappa)$ where $\kappa$ is an algebraically closed field. But the same argument carries over directly to our setting.
\end{proof}

Let $\Lambda\subset\bV$ be a vertex lattice of type $t$. 
By Lemma \ref{lem:cycles_stable}, both $\breve{\Lambda}$ and $\breve{\Lambda}^\sharp$ are stable under $\uF,\uV$ and $\Pi$.
By (relative) Dieudonn\'e theory, the lattices $\breve{\Lambda}$ and $\breve{\Lambda}^\sharp$ correspond to strict $O_{F_0}$-modules over $\F$, denoted by $X_{\Lambda}$ and $X_{\Lambda^\sharp}$ resp. together with quasi-isogenies $\rho_{\Lambda}:X_\Lambda\rightarrow \bX$ and $\rho_{\Lambda^\sharp}:X_{\Lambda^\sharp}\rightarrow \bX$, resp. The framing maps are of height $0$ since $\Lambda$ is a vertex lattice. The natural inclusion $\Lambda\subseteq\Lambda^\sharp$ induces an isogeny $\lambda_\Lambda:X_\Lambda\overset{\sim}{\to} X_{\Lambda}^\sharp\cong X_{\Lambda^\sharp}$, which makes $X_\Lambda$ a polarized strict $O_{F_0}$-module of type $t$.

\begin{definition}\label{def:YZ-cycles}
Let $\cL_{\cZ}$ denote the set of all vertex lattices in $\bV$ of type $\geq h$, and let $\cL_\cY$ denote the set of all vertex lattices in $\bV$ of type $\leq h$. We define the following two classes of Bruhat--Tits (BT) strata:
\begin{altenumerate}
\item For any $\Lambda\in\cL_\cZ$, the \emph{$\cZ$-stratum} ${\cZ}(\Lambda)$ is the subfunctor\footnote{Here we are abusing the terminology. We mean that a subfunctor $\mathcal{Z}(\Lambda)^\wedge$ of $\ov{\mathcal{N}}_{n,\ep}^{[h],\wedge}$ satisfies the desired conditions. And $\mathcal{Z}(\Lambda)$ is the pull-back along the embedding $\ov{\mathcal{N}}_{n,\ep}^{[h]}\subseteq \ov{\mathcal{N}}_{n,\ep}^{[h],\wedge}$. Using \cite{Luo}, we can give a moduli description for $\mathcal{Z}(\Lambda)$, but it is not necessary in this paper.} of $\ov{\cN}^{[h]}_{n,\ep}$ that assigns to each $\F$-scheme $S$ the set of tuples $(X,\iota,\lambda,\rho)$ such that the composition $\rho_{\Lambda,X}:=\rho^{-1}\circ (\rho_\Lambda)_S$ is an isogeny.
\item For any $\Lambda\in\cL_\cY$, the \emph{$\cY$-stratum} ${\cY}(\Lambda^\sharp)$ is the subfunctor of $\ov{\cN}^{[h]}_{n,\ep}$ that assigns to each $\F$-scheme $S$ the set of tuples $(X,\iota,\lambda,\rho)$ such that the composition
$\rho_{\Lambda^\sharp,X^\vee}:=\rho^\vee\circ \lambda_{\bX}\circ\rho_{\Lambda^\sharp}$ is an isogeny, where $\rho_{\Lambda^\sharp}=\rho_\Lambda\circ \lambda_{\Lambda}^{-1}$.
\end{altenumerate}
\end{definition}
By \cite[Lem. 2.10]{RZ}, ${\cZ}(\Lambda)$ and ${\cY}(\Lambda^\sharp)$ are closed formal subschemes of $\ov{\cN}^{[h]}_{n,\ep}$. 
By the arguments of \cite[Lem. 4.2]{Vollaard-Wedhorn}, both strata are representable by projective schemes over $\F$.

When $t_1=t_2=h$, the strata $\cZ(\Lambda_1)$ and $\cY(\Lambda_2^\sharp)$ are each a single geometric point.
In \S \ref{sec:LM}, we prove the following results:
\begin{theorem}\label{thm:LM-results}
Let $\Lambda_1$ and $\Lambda_2$ be vertex lattices of type $t_i=t(\Lambda_i)$ with $t_2<h<t_1$. 
\begin{altenumerate}
\item The BT-strata $\cZ(\Lambda_1)$, $\cY(\Lambda^\sharp_2)$, and their intersection $\cZ(\Lambda_1)\cap \cY(\Lambda^\sharp_2)$ are normal and Cohen--Macaulay. In particular, they lie in the reduced subscheme $\cN_{n,\ep,\red}^{[h]}\subset \ov{\cN}_{n,\ep}^{[h]}$.
\item We have dimension formulas:
\begin{altenumerate2}
\item $\dim \cZ(\Lambda_1) = \frac{1}{2}(t_1+h)$;
\item $\dim \cY(\Lambda_2^\sharp) = n-\frac{1}{2}(h+t_2)-1$;
\item $\dim (\cZ(\Lambda_1)\cap \cY(\Lambda_2^\sharp)) = \frac{1}{2}(t_1-t_2)-1$.
\end{altenumerate2}
\item When $h=2\lfloor n/2\rfloor$, all BT-strata and their intersections are smooth.
\item When $h\neq 2\lfloor n/2\rfloor$,
\begin{altenumerate2}
\item $\cZ(\Lambda_1)$ is smooth if and only if $t_1-h=2$;
\item $\cY(\Lambda_2^\sharp)$ is smooth if and only if $h-t_2=2$;
\item $\cZ(\Lambda_1)\cap\cY(\Lambda^\sharp_2)$ is smooth if and only if either $\cZ(\Lambda_1)$ or $\cY(\Lambda_2^{\sharp})$ is smooth.
\end{altenumerate2}
\item When $h\neq 2\lfloor n/2\rfloor$ and $|h-t_i|>2$:
\begin{altenumerate2}
\item The $\cZ$-stratum is Gorenstein if and only if $t_1=3h+4$;
\item The $\cY$-stratum is Gorenstein if and only if $t_2=3h-2n$;
\item $\cZ(\Lambda_1)\cap \cY(\Lambda_2^\sharp)$ is Gorenstein if and only if $2h=t_1+t_2$.
\end{altenumerate2}
\end{altenumerate}
\end{theorem}

The key ingredient in our proof is the local model diagram relating the BT-strata to their corresponding local models.
For a vertex lattice $\Lambda\in\cL_\cZ$ of type $t\geq h$, we construct the following local model diagram:
\begin{equation}\label{equ:LM-diagram}
\begin{tikzcd}
    &\widetilde{\cZ}(\Lambda)\arrow[dl,"\varphi"']\arrow[dr,"\pi"]&\\
    \cZ(\Lambda)&&\ov{\cM}^{[h]}_{n}(t),
\end{tikzcd}
\end{equation}
where $\ov{\cM}^{[h]}_{n}(t)$ is the \emph{strata local model} of type $t$ (Definition \ref{def:local-model}). 
The construction and properties of this diagram are studied in \S \ref{sec:LM-diagram}.
We show that the morphisms $\varphi$ and $\pi$ are smooth of the same relative dimension.
Consequently, the proof of Theorem \ref{thm:LM-results} reduces to verifying the corresponding properties for the strata local model $\ov{\cM}^{[h]}_{n}(t)$, which we establish through explicit local chart computations in \S \ref{sec:LM-compute}.

Note that the local model approach only allows us to determine certain local properties of the BT-strata, such as reducedness and dimensions. To investigate global properties like connectedness, we must rely on the classical approach, which identifies the BT-strata with Deligne--Lusztig varieties, see \S \ref{sec:relate-DLV}.

\subsection{Bruhat--Tits stratification}
The next main result of our paper is the following:
\begin{theorem}\label{thm:main}
Recall that $\cL_{\cZ}$ (resp. $\cL_{\cY}$) is the collection of vertex lattices in $\bV$ of type $\geq h$ (resp. $\leq h$).
The reduced locus decomposes as:
$$\cN_{n,\ep,\red}^{[h]}=\Bigl(\bigcup_{\Lambda_1\in\cL_\cZ}\cZ(\Lambda_1)\Bigr)\cup \Bigl(\bigcup_{\Lambda_2\in\cL_\cY}\cY(\Lambda_2^\sharp)\Bigr),$$
with inclusions characterized by: 
\begin{itemize}
\item $\cZ(\Lambda_1)\subseteq\cZ(\Lambda_1')$ if and only if $\Lambda_1\supseteq\Lambda_1'$;
\item $\cY(\Lambda_2^\sharp)\subseteq\cY(\Lambda_2'^\sharp)$ if and only if $\Lambda_2\subseteq\Lambda_2'$;
\item $\cZ(\Lambda_1)\cap \cY(\Lambda_2^\sharp)\neq\emptyset$ if and only if $\Lambda_1\subseteq \Lambda_2$;
\item $\cZ(\Lambda_1)\subseteq \cY(\Lambda_2^\sharp)$ if and only if $t(\Lambda_1)=h$ and $\Lambda_1\subseteq \Lambda_2$;
\item $\cY(\Lambda_2^\sharp)\subseteq \cZ(\Lambda_1)$ if and only if $t(\Lambda_2)=h$ and $\Lambda_1\subseteq \Lambda_2$. 
\end{itemize}
Moreover, for any vertex lattice $\Lambda\in\cL_{\cZ}\cap\cL_{\cY}$ (i.e., of type $h$), the stratum $\cZ(\Lambda)=\cY(\Lambda^{\sharp})$ is a discrete geometric point, called the \emph{worst point}.
\end{theorem}
We refer the reader to Theorem \ref{thm:main-EKOR} and Corollary \ref{cor:strata-conn} for the connectedness and irreducibility of the BT-strata and of the reduced locus of the RZ space.

\begin{corollary}\label{cor:dim-com}
Combining with the dimension formula in Theorem \ref{thm:LM-results}(2), we obtain:
\begin{altenumerate}
\item For special cases:
\begin{altenumerate2}
\item(\cite[Thm. 1.1]{RTW}) When $n$ is odd and $h=0$, $\dim \cN_{n,\ep,\red}^{[h]}=\frac{1}{2}(n-1)$;
\item(\cite[Thm. 1.1]{RTW}) When $n$ is even and $h=0$ and $\ep=1$, $\dim \cN_{n,\ep,\red}^{[h]}=\frac{1}{2}n-1$;
\item(\cite[Thm. 1.1]{RTW}) When $n$ is even and $h=0$ and $\ep=-1$, $\dim \cN_{n,\ep,\red}^{[h]}=\frac{1}{2}n$;
\item(\cite[Thm. 5.18]{Wu16}) When $h=n$, $\dim \cN_{n,\ep,\red}^{[h]}=\frac{1}{2}n-1$;

\item(\cite[Thm. 5.18]{Wu16}) When $h=n-1$, $\dim \cN_{n,\ep,\red}^{[h]}=\frac{1}{2}(n-1)$;

\item When $h=n-2$ and $\ep=-1$, $\dim \cN_{n,\ep,\red}^{[h]}=\frac{1}{2}n-2$.
\end{altenumerate2}
\item For the other cases 
\begin{equation*}
\dim \cN_{n,\ep,\red}^{[h]}=
\begin{cases}
\max\left\{\frac{1}{2}(n+h-1),n-\frac{1}{2}h+1\right\} & \text{if $n$ is odd};\\
\max\left\{\frac{1}{2}(n+h)-1,n-\frac{1}{2}h+1\right\} & \text{if $n$ is even and $\ep=1$};\\
\max\left\{\frac{1}{2}(n+h),n-\frac{1}{2}h+1\right\} & \text{if $n$ is even and $\ep=-1$}.
\end{cases}
\end{equation*}
\end{altenumerate}
\end{corollary}

\begin{example}\label{exm:boundary}
\begin{altenumerate}
\item When $h=0$, the reduced locus $\cN_{n,\ep,\red}^{[0]}$ is connected and consists of only $\cZ$-strata (note that while $\cY$-strata appear only at the worst points, these can be replaced by $\cZ$-strata):
$$
\cN_{n,\ep,\red}^{[0]}=\bigcup_{\Lambda\in\cL_\cZ}\cZ(\Lambda)\quad\text{such that}\quad\cZ(\Lambda)\subseteq\cZ(\Lambda')
    \quad\text{if and only if}\quad
    \Lambda\supseteq\Lambda'.
$$
Defining the open stratum $\cZ^\circ(\Lambda):=\cZ(\Lambda)\setminus\bigcup_{\Lambda'\subsetneq\Lambda}\cZ(\Lambda'),$  we have
$$
\cN_{n,\ep,\red}^{[0]}=\coprod_{\Lambda\in\cL_\cZ}\cZ^\circ(\Lambda)\quad\text{with closure relations}\quad\ov{\cZ^\circ(\Lambda)}=\cZ(\Lambda)=\coprod_{\substack{\Lambda'\in\cL_\cZ\\
\Lambda'\subseteq\Lambda}}\cZ^\circ(\Lambda').
$$
This recovers the Bruhat--Tits stratification described in \cite{RTW}.
\item When $h=2\lfloor n/2\rfloor$, the reduced locus consists of only $\cY$-strata. This recovers the results of \cite{Wu16} and improves the identification to the scheme-theoretic level. It should be noted that our current work is independent of the aforementioned reference, which remains unpublished, although it has been cited and used in other works, such as \cite{LL2}.
\end{altenumerate}
\end{example}

When $h\neq 0$ or $2\lfloor n/2\rfloor$, both $\cZ$-strata and $\cY$-strata can appear simultaneously. This complicates both the definition of open strata and their closure relations, this is the reason that we stick to closed strata.
\begin{example}
Consider the case where $n=2m$ is even and $h=n-2$.
\begin{altenumerate}
\item When $\ep=1$, the hermitian space $\bV$ is non-split, and $\cL_{\cZ}$ consists of vertex lattices $\Lambda\subset C$ of type $n-2=h$. By Theorem \ref{thm:main}, $\cZ(\Lambda)=\cY(\Lambda^\sharp)$ is also a $\cY$-stratum, hence we obtain a BT stratification involving only $\cY$-strata. In this case, we have the Hecke correspondence (see \cite[\S 5]{LRZ}):
\begin{equation*}
\begin{aligned}
\xymatrix{
&&\mcN_n^{[n,n-2]}\ar[dl]\ar[dr]&\\
\mcN_{n-1}^{[n-2]}\ar@{^(->}[r]&\mcN_n^{[n-2]}&&\mcN_n^{[n]}.
}
\end{aligned}
\end{equation*}
It would be interesting to compare this with the correspondence in \cite[Theorem 1.0.1]{Li-Ra-Zh}. See Figure \eqref{fig:n=4-split} for the case when $n=4$.

\item When $\ep=-1$, the hermitian space $V$ is non-split, and the corresponding level is special parahoric.
The hermitian space $\bV$ is split, and $\cL_{\cZ}$ consists of vertex lattices of types $n-2$ and $n$. The reduced locus $\cN_{n,-1,\red}^{[n-2]}$ decomposes as a union of irreducible components:
\begin{equation*}
    \cN_{n,-1,\red}^{[n-2]}=\bigcup_{\substack{\Lambda_n\subset C\\ t(\Lambda_n)=n}}\cZ(\Lambda_n)\cup \bigcup_{\substack{\Lambda_0\subset C\\ t(\Lambda_0)=0}}\cY(\Lambda_0^\sharp),
\end{equation*}
where $\dim\cZ(\Lambda_n)=n-1$ and $\dim\cY(\Lambda_0^\sharp)=n/2$. The intersection $\cZ(\Lambda_n)\cap\cY(\Lambda_0^{\sharp})$ is nonempty if and only if $\Lambda_0\subset \Lambda_n$, in which case $\dim \cZ(\Lambda_n)\cap\cY(\Lambda_0^{\sharp})=n/2-1$.
This structure is similar to the Balloon-Ground stratification in the unramified case studied in \cite{KR-BG}; see also \cite[\S 14.1]{Li-Ra-Zh}. See Figure \eqref{fig:n=4-nonsplit} for the case when $n=4$.
\end{altenumerate}
\end{example}

\subsection{Relation to Deligne--Lusztig varieties}\label{sec:relate-DLV}
In this subsection, we address the relationship between Bruhat--Tits (BT) strata and Deligne--Lusztig (DL) varieties. This subsection summarizes the key results presented in \S \ref{sec: DL var} and \S \ref{sec:iden}.
Before proceeding, we introduce the following notation:
\begin{altitemize}
\item We refer the reader to Proposition \ref{prop:BT_strata} for the lattice description of the BT-strata. For instance, for any $\Lambda\in\cL_\cZ$ and any perfect field $\kappa$ over $\F$, we have
\begin{align*}
    \cZ(\Lambda)(\kappa)=\{(X,\iota_X,\lambda_X,\rho_X)\in \cN^{[h]}_{n,\ep}(\kappa)\mid    \Lambda\otimes W_O(\kappa) \subseteq M(X)\subseteq M(X)^\sharp\subseteq \Lambda^\sharp\otimes W_O(\kappa)  \},
\end{align*}
where $W_O(\kappa)$ is the ring of  $O$-Witt vectors.

\item The definitions of $S_\Lambda$, $R_{\Lambda^\sharp}$, and $S_{[\Lambda_1,\Lambda_2]}$ are given in equations \eqref{equ:S-Lambda-defn}, \eqref{equ:R-Lambda-defn}, and \eqref{equ:S Lambda12}, resp. For instance, for any $\Lambda\in\cL_\cZ$ and any field $k$ over $\F$, we have
\begin{equation*}
    S_{\Lambda}^{}(k)=\{\cV\subset (\Lambda^\sharp/\Lambda)\otimes_{\F_q}k\mid \cV\text{ is isotropic},\,\text{and}\, \dim \cV=\frac{t-h}{2},\,\text{and}\, \dim(\cV\cap \Phi(\cV))\geq \frac{t-h}{2}-1\}.
\end{equation*}
\end{altitemize}

We now summarize our results. 
\begin{altitemize}
\item For any vertex lattice $\Lambda\in \cL_{\cY}$ of type $t=t(\Lambda)$, since the types are always even integers, we denote $\ttt=t(\Lambda)/2$ and $\tth=h/2$ for more convenient indexing. Theorem \ref{thm:symplectic main} yields the decomposition:
\begin{equation*}
    S_\Lambda=
    \Bigl(\coprod_{0\leq j\leq \tth< i\leq \ttt}X_{P_{ij}}(\sfw_{ij})\Bigr)\amalg
    \Bigl(\coprod_{0\leq j<\tth<i\leq \ttt}X_{P_{ij}}(\sfw'_{ij})\Bigr)
    \amalg X_{P_{\tth\tth}}(\id).
\end{equation*}
This decomposition implies that $S_\Lambda$ is irreducible. 
All group-theoretic data are related to the symplectic group. For the definitions of the parabolic subgroups $P_{ij}$ and the Weyl group elements $\sfw_{ij}$ and $\sfw'_{ij}$, we refer to the discussion preceding Theorem \ref{thm:symplectic main}.
\item For any vertex lattice $\Lambda\in \cL_{\cY}$ of type $t=t(\Lambda)$, we denote by $m=\lfloor n/2\rfloor$, $\tth'=m-\tth$, and $\ttt'=m-t/2$.
Theorem \ref{thm:orthogonal main} yields the decomposition:
\begin{align*}\label{eq:orthogonal stratification}
   R_{\Lambda^\sharp}=
    \begin{cases}
     \displaystyle  \coprod_{\delta\in \{\pm\}}R_{\Lambda^\sharp}^{\delta} =\coprod_{\delta\in\{\pm\}}\big(\coprod_{\substack{0\leq j\le \tth'< i\leq \ttt' }}X_{P_{ij}}(\sfw_{ij}^{\delta})\big) 
     & \text{ if $h=n$},\\
     \bigl(\displaystyle\coprod_{0\leq j< \tth'< i\leq \ttt'}X_{P_{ij}}(\sfw_{ij})\bigr)\amalg
    \bigl(\coprod_{\substack{0\leq j<\tth'< i\leq \ttt',\\
    \delta\in \{\pm\}}} X_{P_{ij}}(\sfw^{\prime,\delta}_{ij}  )\bigr)
    \amalg X_{P_{\tth'\tth'}}(\id) & \text{ if $h=n-2$},\\
         \bigl(\displaystyle\coprod_{0\leq j< \tth'< i\leq \ttt'}X_{P_{ij}}(\sfw_{ij})\bigr)\amalg
    \bigl(\coprod_{0\leq j<\tth'< i\leq \ttn-\ttt}X_{P_{ij}}(\sfw'_{ij})\bigr)
    \amalg X_{P_{\tth'\tth'}}(\id) & \text{ otherwise },
    \end{cases}
\end{align*}
This decomposition implies that $R_{\Lambda^\sharp}$ has two disjoint irreducible components when $n=2m$ is even and $h=n$; in all other cases, $R_{\Lambda^\sharp}$ is irreducible, see also \cite[Cor. 5.3.6]{LRZ}.
All group-theoretic data are related to the orthogonal group.
For the definitions of the parabolic subgroups $P_{ij}$ and Weyl group elements $\sfw_{ij}$, $\sfw^{\delta}_{ij}$, and $\sfw'_{ij}$, we refer to the discussion preceding Theorem \ref{thm:orthogonal main}.

\item Let $\Lambda_1\in \cL_{\cZ}$ (resp. $\Lambda_2\in \cL_{\cY}$) be a vertex lattice of type $t_1=t(\Lambda_1)$ (resp. $t_2=t(\Lambda_2)$) such that $\Lambda_1\subseteq\Lambda_2$, we denote $\tth=h/2$, $\ttt_1=t_1/2$, and $\ttt_2=t_2/2$.
Proposition \ref{Prop: decom SLambda1Lambda2 into DL var} yields the decomposition:
\begin{equation*}
    S^{}_{[\Lambda_1,\Lambda_2]}=\coprod_{\ttt_2\leq j\leq \tth\leq i\leq \ttt_1}X_{P_{ij}}(\sfw^{}_{ij}),
\end{equation*}
This decomposition implies that $S_{[\Lambda_1,\Lambda_2]}$ is irreducible.
All group-theoretic data are related to the general linear group.
For the definitions of the parabolic subgroups $P_{ij}$ and the Weyl group elements $\sfw_{ij}$, we refer to the discussion before Proposition \ref{Prop: decom SLambda1Lambda2 into DL var}.
\end{altitemize}

\begin{theorem}\label{thm:main-EKOR}
Let $\Lambda_1\in\cL_\cZ$ and $\Lambda_2\in\cL_\cY$ be vertex lattices.
\begin{altenumerate}
\item There are canonical isomorphisms of schemes over $\F$,
\begin{align}\label{equ:psi-isos}
    \Psi_{\cZ}: \cZ(\Lambda_1)\cong S_{\Lambda_1 }\quad \text{and}\quad  \Psi_{\cY}: \cY(\Lambda_2^\sharp)\cong R_{\Lambda_2^\sharp}.
\end{align}
Consequently, all $\cZ$-strata $\cZ(\Lambda_1)$ are irreducible. The $\cY$-stratum $\cY(\Lambda_2^\sharp)$ consists of two disjoint irreducible components when $n=2m$ is even and $h=n$, and is irreducible otherwise.

\item For the case where $\Lambda_1\subseteq \Lambda_2$, there is an isomorphism
\begin{equation*}
    \cZ(\Lambda_1)\cap \cY(\Lambda_2^\sharp)\cong S_{[\Lambda_1,\Lambda_2]}.
\end{equation*}
Consequently, the intersection $\cZ(\Lambda_1)\cap \cY(\Lambda_2^\sharp)$ is irreducible when $t(\Lambda_2)<h<t(\Lambda_1)$. It is a discrete geometric point when $t(\Lambda_2)=h$ or $t(\Lambda_1)=h$.
\end{altenumerate}
\end{theorem}
The isomorphisms $\Psi_{\cZ}$ and $\Psi_{\cY}$ are established in Theorems \ref{prop:ZcycleDLV} and \ref{prop:YcycleDLV}, respectively. Their geometric properties follow from Theorems \ref{thm:symplectic main} and \ref{thm:orthogonal main}, respectively. The intersection isomorphism is proven in Proposition \ref{prop: Z int Y}, while its geometric properties are derived from Proposition \ref{Prop: decom SLambda1Lambda2 into DL var}.

\begin{remark}
In \S \ref{sec:iden}, we provide two different constructions for the isomorphism. A priori, both depend on the choice of a uniformizer $\pi_0\in F_0$ and a unit $\delta\in O_{\breve{F}_0}^\times$ satisfying $\sigma(\delta)=-\delta$. The choice of uniformizer is necessary to apply relative Dieudonn\'e theory, while the choice of unit is necessary to construct the bilinear form on $\bV$. Similar choices have been made in previous works, e.g. \cite{Vollaard-Wedhorn}.
However, the isomorphisms are in fact independent of these choices. A different choice would modify the construction only by a unit scalar, since we are identifying lattices with subspaces, this does not affect the isomorphisms.
\end{remark}

Note that the geometric properties of the BT-strata, including their dimension and normality, can also be derived from the properties of the corresponding DL varieties.

\begin{corollary}\label{cor:strata-conn}
The reduced locus $\cN_{n,\ep,\red}^{[h]}$ is connected except when $h=n$, in which case it has exactly two connected components.
\end{corollary}

\begin{remark}
\begin{altenumerate}
\item Theorem \ref{thm:symplectic main}, Theorem \ref{thm:orthogonal main}, and Proposition \ref{Prop: decom SLambda1Lambda2 into DL var} provide moduli descriptions for all the DL varieties involved. Via the isomorphisms in \eqref{equ:psi-isos}, we can explicitly identify the corresponding strata in the reduced locus.
\item An analogous result exists for the orthogonal case with vertex level. The reader may compare our Theorem \ref{thm:main-EKOR} with \cite[Thm. 7.26]{HZBasic}. 
\item See also \cite[\S 5]{LRZ} for a proof of Corollary \ref{cor:strata-conn} using the group-theoretical method.
\end{altenumerate}
\end{remark}

\subsection{Kottwitz--Rapoport stratification}
To conclude this section, we discuss the relationship between the Kottwitz--Rapoport (KR) stratification and the Bruhat--Tits (BT) stratification. While this material is not essential for the main results of our paper, it presents several interesting connections.
We begin by defining the (closed) Kottwitz--Rapoport strata in the reduced locus of the Rapoport--Zink space.
\begin{definition}
Let us define two $\F$-schemes associated with the reduced locus of $\cN_{n,\ep,\red}^{[h]}$:
\begin{altenumerate}
\item The $\F$-scheme $\cZ$ is the reduced locus of the subfunctor of $\ov{\cN}_{n,\ep}^{[h]}$ consisting of tuples $(X,\iota,\lambda,\rho)$ over an $\F$-scheme $S$ that satisfying:
$\lambda^\vee(\Fil(X^\vee))\subseteq\iota(\pi) D(X)$.
\item The $\F$-scheme $\cY$ is the reduced locus of the subfunctor of $\ov{\cN}_{n,\ep}^{[h]}$ consisting of tuples $(X,\iota,\lambda,\rho)$ over an $\F$-scheme $S$ that satisfying:
$\lambda(\Fil(X))\subseteq\iota(\pi)D(X^\vee)$.
\end{altenumerate}
\end{definition}

\begin{remark}
For strongly non-special $h$, i.e., when $h\neq 0,n-2$ and $n$ (see \cite[Def. 1.3.1]{Luo}), the genuine Kottwitz--Rapoport stratification (see \cite[\S 1.6]{He-KR-strata}) decomposes as:
\begin{equation*}
\cN_{n,\ep,\red}^{[h]}=(\cZ\setminus \cY)\amalg (\cY\setminus \cZ)\amalg (\cZ\cap \cY\setminus \mathrm{WT})\amalg \mathrm{WT}
\end{equation*}
where $\mathrm{WT}$, called the \emph{worst points}, is the disjoint union of all $\cZ(\Lambda)$ satisfying $t(\Lambda)=h$ (see Theorem \ref{thm:main}).
This stratification is intimately related to the Kottwitz--Rapoport stratification of the special fiber of the corresponding local model for $\cN_{n,\ep}^{[h]}$. In the local model setting, the Kottwitz--Rapoport strata are indexed by admissible sets (see \cite[\S 11]{PR2008} or \cite[Thm. 8.1]{Zhu-coherent}). When $h$ is strongly non-special, a direct computation shows that the admissible set consists of four elements:
\begin{equation*}
\begin{aligned}
\xymatrix @R=0.3pc{
w_1\ar[dr]&&\\
&w_{12}\ar[r]&\id,\\
w_2\ar[ur]&&
}
\end{aligned}
\end{equation*}
where the arrows indicate the Bruhat order. The extreme elements $w_1$ and $w_2$ correspond to $\cZ\setminus \cY$ and $\cY\setminus \cZ$ respectively, while $w_{12}$ corresponds to $\cZ\cap \cY\setminus \mathrm{WT}$ and $\id$ corresponds to the worst point.
In particular, the subvarieties $\cZ$ and $\cY$ in our definition correspond to the irreducible components of the special fiber of the local model.
In the special case when $h=n-2$, the variety $\cZ$ further decomposes into two components; we will address this case in Remark \ref{rmk:Yu-case}.
\end{remark}

We show in \S \ref{sec:KR-strata} that:
\begin{proposition}
\begin{altenumerate}
\item The reduced locus decomposes as:
\begin{equation*}
\cN_{n,\ep,\red}^{[h]}=\cZ\cup \cY.
\end{equation*}

\item For all lattices $\Lambda_1\in\cL_{\cZ}$ and $\Lambda_2\in\cL_{\cY}$, we have 
$\cZ(\Lambda_1)\subset \cZ$ and $\cY(\Lambda_2^\sharp)\subset \cY$.
\end{altenumerate}
\end{proposition}

However, these inclusions are strictly proper. Indeed, in the strongly non-special case, as shown in Proposition \ref{prop:crutial-lem}, we have:
\begin{equation*}
\bigcup_{\Lambda_1\in\cL_\cZ}\cZ(\Lambda_1)\subsetneq \cZ
\quad\text{and}\quad
\bigcup_{\Lambda_2\in\cL_\cY}\cY(\Lambda_2^\sharp)\subsetneq \cY.
\end{equation*}
This reveals a non-trivial interaction between the BT-stratification and the KR-stratification: the $\cZ$-KR-stratum receives contributions from certain $\cY$-BT-strata, and vice versa (except when there are no $\cZ(\Lambda)$, resp. $\cY(\Lambda^\sharp)$). The precise nature of these contributions can be determined explicitly from the moduli description.

In \S \ref{sec:proof-KR-DLV}, we prove the following result:
\begin{proposition}\label{prop:KR-DLV}
\begin{altenumerate}
\item The isomorphism $\Psi_{\cZ}$ restricts to give:
\begin{altenumerate2}
\item $\displaystyle\cZ(\Lambda_1)\setminus \cY \cong\coprod_{i=0}^{\tth-1}X_{P_{i\tth}}(\sfw_{i\tth})$;
\item $\displaystyle\cZ(\Lambda_1)\cap \cY \cong \Bigl(\coprod_{0\leq j<\tth< i\leq \ttt}X_{P_{ij}}(\sfw_{ij})\Bigr)\amalg
    \Bigl(\coprod_{0\leq j<\tth<i\leq \ttt}X_{P_{ij}}(\sfw'_{ij})\Bigr)
    \amalg X_{P_{\tth\tth}}(\id)$.
\end{altenumerate2}
\item The isomorphism $\Psi_\cY$ similarly restricts to give:
\begin{altenumerate2}
\item $\displaystyle
\cY(\Lambda_2^\sharp)\setminus \cZ \cong\coprod_{i=0}^{\tth'-1}X_{P_{i\tth'}}(\sfw_{i\tth'})$;
\item  $\displaystyle\cY(\Lambda_2^\sharp)\cap \cZ \cong\left\{\begin{array}{ll} \emptyset & \text{ if $h=n$,}\\
    \displaystyle\Bigl(\coprod_{0\leq j<\tth'< i\leq \ttt'}X_{P_{ij}}(\sfw_{ij})\Bigr)\amalg
    \Bigl(\coprod_{\substack{0\leq j<\tth'<i\leq \ttt',\\ \delta\in \{\pm\}}}X_{P_{ij}}(\sfw^{\prime,\delta}_{ij})\Bigr)
    \amalg X_{P_{\tth'\tth'}}(\id) & \text{ if $h=n-2$,}\\
     \displaystyle\Bigl(\coprod_{0\leq j<\tth'< i\leq \ttt'}X_{P_{ij}}(\sfw_{ij})\Bigr)\amalg
    \Bigl(\coprod_{0\leq j<\tth'<i\leq \ttt'}X_{P_{ij}}(\sfw'_{ij})\Bigr)
    \amalg X_{P_{\tth'\tth'}}(\id) & \text{ otherwise}.\end{array}\right.$
\end{altenumerate2}    
\end{altenumerate}    
\end{proposition}
These decompositions play a key role in refining KR strata to EKOR-strata (cf. \cite[\S 2.4]{GHN2024}).

\section{Rapoport--Zink spaces}
In this section, we study the ramified unitary Rapoport--Zink spaces with vertex-level structures.

\subsection{Review of strict $O_{F_0}$-modules}\label{sec:strict-module}
We begin with a review of strict $O_{F_0}$-modules. For a comprehensive treatment, we refer the reader to \cites{Mihatsch_2022,kudla2023padicuniformizationunitaryshimura,LMZ}. Throughout this work, we assume $p\neq 2$.
Let $F_0/\Q_p$ be a finite extension of $p$-adic fields with a fixed uniformizer $\pi_0\in O_{F_0}$. We assume that the residue field is finite of order $q$. 

For an $O_{F_0}$-scheme $S$, a \emph{strict $O_{F_0}$-module} over $S$ is a pair $(X,\iota)$ where $X$ is a $p$-divisible group over $S$ and $\iota: O_{F_0}\rightarrow \End(X)$ is an action 
such that the induced $O_{F_0}$-action on  $\Lie(X)$ agrees with the structure morphism $O_{F_0}\rightarrow \Oo_S$.
A strict $O_{F_0}$-module $(X,\iota)$ is called \emph{formal} if its underlying $p$-divisible group $X$ is formal. The dimension of a strict $O_{F_0}$-module is the dimension of its underlying $p$-divisible group. We refer the reader to the aforementioned references for the notion of the (relative) \emph{height} of strict $O_{F_0}$-modules.

Similar to $p$-divisible groups, we have the following exact sequence of $\mcO_S$-modules for strict $O_{F_0}$-modules (see \cite[Prop. B.3.3]{Fargues08}):
\begin{equation}\label{equ:exact-seq-for-O-module}
    0\rightarrow \Fil(X)\rightarrow D(X)\rightarrow \Lie(X)\rightarrow 0.
\end{equation}
where $D(X)$ is the (relative) de Rham realization and $\Fil(X)\subseteq D(X)$ is the Hodge filtration. According to (relative) Grothendieck--Messing theory (see \cite[Lem. 3.12]{ACZ}, or \cite[Prop. B.8.2]{Fargues08}), deformations of $X$ along $O_{F_0}$-pd-thickenings are in canonical bijection with liftings of the Hodge filtration.
See also Definition \ref{def:pd-structure} and Theorem \ref{thm:ACZ-main}.

We refer the reader to \S \ref{sec:OWittvector} for the definition of $O$-Witt vectors $W_O(R)$. Let $O=O_{F_0}$.
When $R$ is a perfect algebra, a \emph{relative Dieudonn\'e module} is a $W_O(R)$-module $M$ equipped with a $\sigma$-linear operator $\uF$ and a $\sigma^{-1}$-linear operator $\uV$ satisfying the relation $\uF \uV=\uV \uF=\pi_0\id$. 
As shown in \cite[\S B.8]{Fargues08}, there is an equivalence of categories between strict $O_{F_0}$-modules and relative Dieudonn\'e modules. Under this equivalence $X\mapsto M(X)$, the exact sequence \eqref{equ:exact-seq-for-O-module} can be identified with 
$$
0\to \uV M(X)/\pi_0 M(X)\to M(X)/\pi_0 M(X)\to M(X)/\uV M(X)\to 0.
$$

Given a strict $O_{F_0}$-module, Faltings \cite{Faltings-dual} defines the \emph{Faltings dual} $X^\vee$, which is a strict $O_{F_0}$-module over the same base. In the remaining part of the paper, $(-)^\vee$ will always refer to the Faltings dual.

The relative Dieudonn\'e module is compatible with the Faltings dual, see \cite[Prop. B.40]{Fargues08}. The duality structure induces a perfect pairing
\begin{equation*}
    D(X)\times D(X^\vee)\rightarrow \Oo_S
\end{equation*}
under which $\Fil(X)\subset D(X)$ and $\Fil(X^\vee)\subset D(X^\vee)$ are orthogonal complements.

We have $(X^\vee)^\vee=X$. A \emph{relative polarization} is an isogeny $\lambda:X\to X^\vee$ such that $\lambda^\vee=-\lambda$. A strict $O_{F_0}$-module is called \emph{biformal} if both $X$ and $X^\vee$ are formal.

\subsection{Framing object}\label{sec:framing}
\begin{definition}\label{def:basic-triple}
Let $h,n$ be integers with $0\leq h\leq n$ and $h$ even. For any $S\in\Nilp O_{\breve{F}}$, a \emph{hermitian $O_{F}$-module} of rank $n$ and type $h$ (with signature $(n-1,1)$) over $S$ is a triple $(X,\iota,\lambda)$ satisfying:
\begin{altenumerate}
\item $X$ is a strict biformal $O_{F_0}$-module over $S$ of height $2n$ and dimension $n$.
\item $\iota:O_F\rightarrow \End(X)$ is an action of $O_F$ on $X$ extending the $O_{F_0}$-action. The induced action of $O_F$ on $\Fil(X)$ satisfies:
\begin{itemize}
\item\,(Kottwitz condition) The characteristic polynomial satisfies\footnote{This is equivalent to requiring $\mathrm{char}(\iota(\pi)\mid\Lie (X))=(T-\pi)^{n-1}(T+\pi)$ 
for the Lie algebra of $X$.}
\begin{equation*}
    \mathrm{char}(\iota(\pi)\mid\Fil (X))=(T-\pi)(T+\pi)^{n-1};
\end{equation*}
\item\,(Wedge condition) The following relations hold:
\begin{equation*}
    \bigwedge^2(\iota(\pi)+\pi \mid \Fil(X))=0;
    \quad
    \bigwedge^n(\iota(\pi)-\pi \mid \Fil(X))=0.
\end{equation*}
\item\,(Spin condition) When $n$ is even and $h=n$, we require that $\iota(\pi)+\pi$ is non-vanishing on $\Fil(X)$.
\end{itemize}
\item $\lambda$ is a (relative) polarization of $X$ that is $ O_{F}/ O_{F_0}$-semilinear in the sense that the Rosati involution $\Ros_\lambda$ induces the non-trivial involution $\overline{(-)}\in\Gal(F/F_0)$ on $\iota: O_F\rightarrow \End(X)$.
\item We require that $\ker[\lambda]\subseteq X[\iota(\pi)]$ with order $q^{h}$.
\end{altenumerate}

By the kernel condition in (4), there exists a unique isogeny $\lambda^\vee$ such that the composition
\begin{equation*}
    X\xrightarrow{\lambda}X^\vee\xrightarrow{\lambda^\vee} X
\end{equation*}
equals $\iota(\pi)$.

An isomorphism between two such triples $(X_1,\iota_1,\lambda_1)\xrightarrow{\sim}(X_2,\iota_2,\lambda_2)$ is an $O_F$-linear isomorphism $\varphi:X_1\xrightarrow{\sim} X_2$ satisfying $\varphi^*(\lambda_2)=\lambda_1$.
\end{definition}

The signature conditions imposed on hermitian $O_F$-modules ensure that the associated Rapoport--Zink space has the property of ``topological flatness'' (see Proposition \ref{prop:rz-top flat}). In the special case where $h=n$, the spin condition is essential, see \cite[Remark 3.11]{RSZ-Duke}.

Let $(\bX,\iota_{\bX},\lambda_{\bX})$ be a hermitian $O_F$-module of dimension $n$ over $\Spec \F$. The associated rational Dieudonn\'e module $N=M(\bX)[1/\pi_0]$ is a $2n$-dimensional $\breve{F}_0$-vector space equipped with a $\sigma$-linear operator $\uF$ and a $\sigma^{-1}$-linear operator $\uV$. Throughout this paper, we restrict our attention to the \emph{supersingular} case, which means the rational Dieudonn\'e module has all relative slopes $\frac{1}{2}$.

The $O_F$-action $\iota_{\bX}: O_F\rightarrow \End(\bX)$ induces an action on $N$ that commutes with both operators $\uF$ and $\uV$.
The polarization of $\bX$ induces a skew-symmetric $\breve{F_0}$-bilinear form $\langle\cdot,\cdot\rangle$ on $N$ satisfying
\begin{equation*}
    \langle \uF x,y\rangle=\langle x,\uV y\rangle^\sigma,
    \quad
    \langle \iota(a)x,y\rangle=\langle x,\iota(\bar{a})y\rangle,
    \quad \text{for any } x,y\in N, a\in O_F.
\end{equation*}
Furthermore, $N$ is an $n$-dimensional $\breve{F}$-vector space equipped with a $\breve{F}/\breve{F_0}$-hermitian form $h(\cdot,\cdot)$ defined by:
\begin{equation*}
    h(x,y):=\delta(\langle \pi x,y\rangle+\pi \langle x,y\rangle),
\end{equation*}
where $\delta$ is a fixed element in $ O^\times_{\breve{F_0}}$ satisfying $\sigma(\delta)=-\delta$. 
The bilinear form $\langle\cdot,\cdot\rangle$ can be recovered from $h(\cdot,\cdot)$ via the relation:
\begin{equation*}
    \langle x,y\rangle=\frac{1}{2\delta}\tr_{\breve{F}/\breve{F}_0}(\pi^{-1}h(x,y)).
\end{equation*}

Let $\tau:=\Pi\uV^{-1}$, where $\Pi$ is the induced action of $\iota(\pi)$. Let $\bV:=N^{\tau=1}$. Then $\bV$ is an $F$-vector space of dimension $n$, and we have $N=\bV\otimes_{F_0}\breve{F_0}$. The $F/F_0$-hermitian form $h(\cdot,\cdot)$ restricts to $\bV$, and we retain this notation for the restricted form.

We define the \emph{sign} of $\bX$, denoted $\ep=\ep(\bX)$, as $-\mathrm{Hasse}(\bV)$, where $\mathrm{Hasse}(-)$ is the Hasse invariant of the hermitian space, taking values in ${\pm 1}$. For any dimension $n$, type $h$, and sign $\ep$, we denote by $(\bX^{[h]}_{n,\ep},\iota{\bX^{[h]}_{n,\ep}},\lambda_{\bX^{[h]}_{n,\ep}})$ the corresponding \emph{framing object}, which is a hermitian $O_F$-module over $\Spec \F$. For the existence and uniqueness of these framing objects, we refer the reader to \cite[\S 5]{LRZ}.

\subsection{Two Rapoport--Zink spaces}\label{sec:RZ-def}
In this subsection, we introduce two different Rapoport--Zink (RZ) spaces.

\begin{definition}
Let  $(\bX^{[h]}_{n,\ep},\iota_{\bX^{[h]}_{n,\ep}},\lambda_{\bX^{[h]}_{n,\ep}})$ be a framing object over $\F$ of rank $n$ and type $h$.
\begin{altenumerate}
\item The wedge (relative) RZ space $\cN_{n,\ep}^{[h],\wedge}$ is the functor
\begin{equation*}
    \cN_{n,\ep}^{[h],\wedge}\longrightarrow \Spf O_{\breve{F}}
\end{equation*}
that assigns to each scheme $S$ the set of isomorphism classes of tuples  $(X,\iota,\lambda,\rho)$, where:
\begin{itemize}
\item\, $(X,\iota,\lambda)$ is a unitary $O_F$-module over $S$ of dimension $n$ and type $h$ in the sense of Definition \ref{def:basic-triple}. 
\item\, $\rho:X\times_S \overline{S}\rightarrow \bX^{[h]}_{n,\ep}\times_{\F}\overline{S}$ is an $ O_F$-linear quasi-isogeny of height $0$ over the reduction $\overline{S}:=S\times_{\Spf  O_{\breve{F}}}\Spec\F$ such that $\rho^*(\lambda_{\bX^{[h]}_{n,\ep},\overline{S}})=\lambda_{\overline{S}}$.
\end{itemize}
\item  The (relative) RZ space $\cN^{[h]}_{n,\ep}$ is defined as the closed formal subscheme of $\cN^{[h],\wedge}_{n,\ep}$  cut out by the ideal sheaf $\Oo_{\cN^{[h],\wedge}_{n,\ep}}[\pi_0^\infty]\subset \Oo_{\cN^{[h],\wedge}_{n,\ep}}$. This is the maximal closed formal subscheme of $\cN_{n,\ep}^{[h],\wedge}$ that is flat over $O_{\breve{F}}$.
\end{altenumerate}
\end{definition}
By \cite{RZ}, the RZ spaces $\cN_{n,\ep}^{[h],\wedge}$ and $\cN_{n,\ep}^{[h]}$ are representable by formal schemes locally of finite type over $\Spf O_{\breve{F}}$. Both spaces have relative dimension $n-1$ (see Proposition \ref{prop:rz-top flat}). In two special cases, these spaces coincide: the self-dual case (when $h=0$) and the $\pi$-modular case (when $n$ is even and $h=n$), as shown in \cites{Pappas2000,RSZ-Duke}. In all other cases, the spaces are distinct. While there exists a moduli description of the RZ space $\cN_{n,\ep}^{[h]}$ using the strengthened spin condition \cite{Luo}, we do not require it for our purposes.

\begin{proposition}\label{prop:rz-top flat}
For any field $k$ over $\F$, the closed immersion between the closed formal subschemes yields an equality of geometric points $\cN_{n,\ep}^{[h]}(k)=\cN_{n,\ep}^{[h],\wedge}(k)$. Consequently, these spaces have identical reduced loci.
\end{proposition}
\begin{proof}
In the $\pi$-modular case, i.e., when $n=h$, this proposition is established in \cite[Prop. 3.4]{Wu16}. We therefore assume $n\neq h$ and set $\tth=h/2$.
Following \cite{RZ} and applying the linear modification of \cite[Prop. 2.3]{Pappas2000} to $\cN_{n,\ep}^{[h]}$, we have the local model diagrams for $\cN_{n,\ep}^{[h],\wedge}$ and $\cN_{n,\ep}^{[h]}$:
\begin{equation}\label{equ:RZ lm diagram}
\begin{aligned}
\xymatrix{
&\widetilde{\cN}_{n,\ep}^{[h]}\ar@{^(->}[d]\ar[dl]\ar[dr]&\\
\cN_{n,\ep}^{[h]}\ar@{^(->}[d]&\widetilde{\cN}_{n,\ep}^{[h],\wedge}\ar[dl]\ar[dr]&\cM_{n}^{[h]}\ar@{^(->}[d]\\
\cN_{n,\ep}^{[h],\wedge}&&\cM_n^{[h],\wedge}
}
\end{aligned}
\end{equation}
where all diagonal arrows are smooth of equal relative dimension. The right-hand-side spaces are defined in Definition \ref{def:local-model}.

Using this diagram, our problem reduces to showing that the local model $\cM_n^{[h],\wedge}$ is ``topologically flat'', that is, the closed immersion $\cM_{n}^{[h]}\hookrightarrow \cM_{n}^{[h],\wedge}$ is defined by a nilpotent ideal.
To be more precise, for any point $(X,\iota,\lambda,\rho)\in\cN_{n,\ep}^{[h],\wedge}(k)$, \cite[Appendix]{RZ} ensures the existence of an \'etale extension $\Spec R\to \Spec k$ equipped with the following trivialization:
\begin{equation}\label{equ:RZ trivialization}
\bigl[
\cdots\to D(X)_R\to D(X^\vee)_R\to\cdots
\bigr]\simeq 
\bigl[
\cdots\to \bfL_{-\tth}\otimes_{O_{F_0}}R\to \bfL_{\tth}\otimes_{O_{F_0}}R\to\cdots
\bigr].
\end{equation}
where $\bfL_i$ denotes the standard lattice chain defined in \eqref{eq:standard-lattice}. Under this trivialization, the Hodge filtration determines an $R$-point $(\cF_i\subset \bfL_{i}\otimes_{O_{F_0}}R)$ of $\cM_n^{[h],\wedge}$ (see Definition \ref{def:local-model}).

By the linear modification of the local model (\cite[Prop. 2.3]{Pappas2000}), a $k$-point $(X,\iota,\lambda,\rho)\in \cN_{n,\ep}^{[h],\wedge}(k)$ lies in $\cN_{n,\ep}^{[h]}(k)$ if and only if its Hodge filtration defines an $R$-point $(\cF_i\subset \bfL_{i,R})\in \cM_{n}^{[h]}(R)$ after trivialization.
Thus, our problem reduces to showing that $\cM_{n}^{[h]}(R)=\cM_{n}^{[h],\wedge}(R)$ for any \'etale $k$-algebra $R$, which would follow from the topological flatness of the wedge local model.

Now we establish the topological flatness. For odd $n=2m+1$, this follows directly from \cite{Smithling2011}. For even $n=2m$, by \cite[Thm. 9.6.1]{Smithling_2014}, we only need to verify that both local models are indexed by the same permissible elements in the double quotient $W_{{\tth}}\backslash \widetilde{W}\slash W_{{\tth}}$ of the affine Weyl group.

In \cite{Smithling_2014}, it is shown that spin-permissibility (distinct from the spin condition in Definition \ref{def:basic-triple}) is equivalent to permissibility. Thus, we only need to prove that wedge-permissibility implies spin-permissibility in our case. For the relevant terminology, see \cite[\S 7]{Smithling_2014}. 

By \cite[Proof of Prop. 7.2.2.]{Smithling_2014}, if an element $\widetilde{w}\in W/W_{\tth}$ is naive-permissible, then it is a $2$-face. From \cite[Prop. 7.2.2.]{Smithling_2014}, we have $0\leq \mu^{\widetilde{w}}_\tth(j)\leq 2$ for all $j$. Moreover, by \cite[Lem. 4.3.7]{Smithling_2014}, we have the following basic inequality:
\begin{equation}\label{equ:2-face}
1\leq \mu_{\tth}^{\widetilde{w}}(j)+\mu_{\tth}^{\widetilde{w}}(n+1-j)\leq 2 \quad \text{for all}\quad j\in\{\tth+1,\cdots,n-\tth\}.
\end{equation}
Consequently, for any $j\in \{\tth+1,\cdots,n-\tth\}$, if $\mu_{\tth}^{\widetilde{w}}(j)\neq 0$, then either $\mu_{\tth}^{\widetilde{w}}(j)=\mu_{\tth}^{\widetilde{w}}(n+1-j)=1$, or $\mu_{\tth}^{\widetilde{w}}(j)=2, \mu_{\tth}^{\widetilde{w}}(n+1-j)=0$.

For wedge-permissible elements $\widetilde{w}$, \cite[Prop. 7.3.2]{Smithling_2014} yields
\begin{equation}\label{equ:wedge-perm}
\#\{j\mid \mu_\tth^{\widetilde{w}}(j)=0\}\leq 1
\end{equation}
According to \cite[Prop. 7.4.7]{Smithling_2014}, spin-permissibility is equivalent to wedge-permissibility together with the $(P3)$ condition:
\begin{enumerate}
\item If $\mu_\tth^{\widetilde{w}}$ is self-dual, i.e. if $\mu_{\tth}^{\widetilde{w}}=\mu_{-\tth}^{\widetilde{w}}$ (\cite[Def. 4.3.9]{Smithling_2014}), then $\#\{j\mid \mu_{\tth}^{\widetilde{w}}(j)=0\}\equiv 1\mod 2$;
\item In any case\footnote{While the statement of \cite[Prop. 7.4.7]{Smithling_2014} may be ambiguous, note that the existence of  $j\in \{\tth+1,\cdots,n-\tth\}$ such that $\mu_\tth^{\widetilde{w}}(j)=1$ is equivalent to $a_\tth>a_{\tth}^\perp$, see the bottom of page 333 (resp. 334) for the definition of $a_\tth$ (resp. $a_{\tth}^\perp$) and the middle of page 336 for the proof that this implies the spin condition.}, there exists $j\in \{\tth+1,\cdots,n-\tth\}$ such that $\mu_\tth^{\widetilde{w}}(j)=1$.
\end{enumerate}

When $\tth\leq m-2$, we have $\#\{\tth+1,\ldots,n-\tth\}\geq 4$. In this case, wedge-permissibility \eqref{equ:wedge-perm} ensures the existence of at least one $j\in \{\tth+1,\ldots,n-\tth\}$ satisfying $\mu_\tth^{\widetilde{w}}(j)=1$.

When $\tth=m-1$, if there exists no $j\in \{\tth+1,\ldots,n-\tth\}=\{m,m+1\}$ such that $\mu_\tth^{\widetilde{w}}(j)=1$, then by \cite[Lem. 7.4.5]{Smithling_2014}, $\mu_\tth^{\widetilde{w}}$ is self-dual. Therefore, $\mu_\tth^{\widetilde{w}}$ satisfies $(P3)(1)$ and hence the spin condition.
\end{proof}

Let $N=M(\bX^{[h]}_{n,\ep})[\frac{1}{\pi_0}]$ denote the rational Dieudonn\'e module of the framing object. The following proposition is a direct generalization of \cite[Prop. 2.2]{RTW}.
\begin{proposition}\label{prop:RZ-lattice}
Let $\kappa$ be a perfect field over $\F$.
There is a bijection between $\cN_{n,\ep,\red}^{[h]}(\kappa)$ and the set of $W_O(\kappa)$-lattices 
\begin{equation*}
\left\{M\subset N\otimes_O W_O(\kappa)\mid \pi M^\sharp\subseteq M\stackrel{h}{\subset}M^\sharp,\quad\Pi M\subseteq \tau^{-1}(M)\subseteq\Pi^{-1}M, \quad M\stackrel{\leq 1}{\subseteq}(M+\tau(M))\right\}.
\end{equation*}
When $n$ is even and $h=n$, the last relation is replaced by $M\stackrel{1}{\subset}(M+\tau(M))$. We refer the reader to \S \ref{sec:notation} for notations $\overset{\leq 1}{\subseteq}$ and $\overset{ 1}{\subseteq}$.
\end{proposition}
\begin{proof}
We assume first $n\neq h$. For any point in $\cN_{n,\ep,\red}^{[h]}(\kappa)=\cN_{n,\ep,\red}^{[h],\wedge}(\kappa)$, its Dieudonn\'e module establishes a bijection between $\ov{\cN}_{n,\ep}^{[h]}(\kappa)$ and the set of $W_O(\kappa)$-lattices:
\begin{equation*}
\left\{M\subset N\otimes_O W_O(\kappa)\mid \pi M^\sharp\subseteq M\stackrel{h}{\subset}M^\sharp,\quad \Pi M\subseteq M,\quad \pi_0 M\subset \uV M\stackrel{n}{\subset}M,\quad \uV M\stackrel{\leq 1}{\subseteq}\uV M+\Pi M\right\}.
\end{equation*}
where $M\stackrel{h}{\subset}M^\sharp$ comes from the polarization; $\Pi M\subset M$ due to the $O_F$-stability; $\uV M\stackrel{\leq 1}{\subset}\uV M+\Pi M$ follows from the Kottwitz and wedge conditions, where $\uV M/\pi_0 M\subset M/\pi_0 M$ is identified with the Hodge filtration.
Moreover, we have the following equivalences: (1) $\pi_0 M\subset \uV M\subset M$ is equivalent to $\Pi M\subset \tau^{-1}(M)\subset \Pi^{-1}M$; (2) $\uV M\stackrel{\leq 1}{\subset}\uV M+\Pi M$ is equivalent to $M\stackrel{\leq 1}{\subset}(M+\tau(M))$.
This completes the proof.

In the $\pi$-modular case, i.e. when $n=h$, the same argument applies. By Definition \ref{def:basic-triple}, in this case, we further require the spin condition, which forces $M\stackrel{\leq 1}{\subset}(M+\tau(M))$ to become the stricter $M\stackrel{1}{\subset}(M+\tau(M))$.
\end{proof}

\section{Local models of basic locus strata}\label{sec:LM}
In this section, we introduce local models of Bruhat--Tits strata.

\subsection{Standard polarized lattice chain}\label{sec:standard-lattice}
In this subsection, we recall the basic setup for local models.

Consider the vector space $F^n$ with the standard $F$-basis $e_1,\cdots,e_n$. We equip it with a split $F/F_0$-Hermitian form
\begin{equation*}
    h: F^n\times F^n\rightarrow F,\quad h(ae_i,be_j)=\bar{a}b\delta_{i,n+1-j},\quad a,b\in F.
\end{equation*}
Attached to $h$ are the respective alternating and symmetric $F_0$-bilinear forms $F^n\times F^n\rightarrow F_0$ given by
\begin{equation*}
\langle x,y\rangle:=\frac{1}{2}\tr_{F/F_0}(\pi^{-1}h(x,y))
\quad\text{and}\quad
(x,y):=\frac{1}{2}\tr_{F/F_0}(h(x,y)).
\end{equation*}		
For each integer $i=bn+c$ with $0\leq c<n$, define the standard $ O_F$-lattices
\begin{equation}\label{eq:standard-lattice}
	\bfL_i:=\sum_{j=1}^c\pi^{-b-1} O_F e_j+\sum_{j=c+1}^n\pi^{-b} O_F e_j\subset F^n.
\end{equation}
For all $i$, the $\langle\,,\,\rangle$-dual $\bfL^\vee$ of $\bfL_i$ in $F^n$ is $\bfL_{-i}$, by which we mean that
\begin{equation*}
    \{x\in F^n\mid \langle \bfL_i,x\rangle\subset O_{F_0}\}=\bfL_{-i}.
\end{equation*}
By restriction, we have a perfect $O_{F_0}$-bilinear pair:
\begin{equation*}\label{moduli_setup:dual}
	\bfL_i\times\bfL_{-i}\xrightarrow{\langle\,,\,\rangle}  O_{F_0}
\end{equation*}
Similarly, $\bfL_{n-i}$ is the $(\,,\,)$-dual of $\bfL_i$ in $F^n$. The $\bfL_i$'s form a complete, periodic, self-dual lattice chain
\begin{equation*}
    \cdots\subset\bfL_{-2}\subset\bfL_{-1}\subset\bfL_0\subset\bfL_1\subset\bfL_2\subset\cdots.
\end{equation*}
For any subset $I\subset\{1,\cdots,m\}$ with $m=\lfloor \frac{n}{2}\rfloor$, we define the standard polarized chain $\bfL_I$ as a sub-lattice chain with indices $i\in\pm I+n\Z$.

For even integers $r,s,t$, let $\ttr=\frac{1}{2}r$, $\tts=\frac{1}{2}s$, and $\ttt=\frac{1}{2}t$.
We define the following index sets:
$[s] := \{\pm \tts\} + n\mathbb{Z}$,
$[s,t] := \{\pm \tts, \pm \ttt\} + n\mathbb{Z}$, and 
$[r,s,t] := \{\pm \ttr, \pm \tts, \pm \ttt\} + n\mathbb{Z}$.
For any such index set $I$, we define the standard lattice chain $\mathbf{L}_I := \{\mathbf{L}_i\}_{i \in I}$. This chain is self-dual in the sense of \cite{RZ}. We use these half-integral indices to maintain consistency with the local model indexing conventions used in \cites{Pappas_Rapoport_2009,Luo}.

\subsection{Local models}
We first recall the definition of the relative local model.
For any $O_F$-algebra $R$, let $\mathbf{L}_{i,R}$ denote the tensor product $\mathbf{L}_{i}\otimes_{O_{F_0}}R$.
Let $\Pi := \pi \otimes 1$ and $\pi := 1 \otimes \pi$.

\begin{definition}\label{def:local-model}
\begin{altenumerate}
\item The \emph{wedge local model} $\cM^{[h],\wedge}_n$ is a projective scheme over $\Spec  O_F$. It represents the moduli problem that assigns to each $ O_F$-algebra $R$ the set of all families $\left(\cF_i \subset \bfL_{i,R}\right)_{i\in [h]}$ such that:
\begin{altitemize}		
	\item[$\mathrm{LM1.}$]
		for all $i\in [h]$, $\cF_i$ is an $ O_F \otimes_{ O_{F_0}} R$-submodule of $\bfL_{i,R}$ and an $R$-direct summand of rank $n$;
	\item[$\mathrm{LM2.}$]
		for all $i,j\in [h]$ with $i < j$, the natural map $\bfL_{i,R} \to \bfL_{j,R}$ carries $\cF_i$ into $\cF_j$;
	\item[$\mathrm{LM3.}$]
		for all $i\in [h]$, the isomorphism $\bfL_{i,R} \xrightarrow{\Pi} \bfL_{i-n,R}$ identifies
		\begin{equation*}
		    \cF_i \isoarrow \cF_{i-n};
		\end{equation*}
	\item[$\mathrm{LM4.}$]
		for all $i\in [h]$, the perfect $R$-bilinear pairing
        \begin{equation*}
           \bfL_{i,R} \times \bfL_{-i,R}
			\xrightarrow{\langle-,-\rangle \otimes R} R
        \end{equation*}
		identifies $\cF_i^\perp$ with $\cF_{-i}$ inside $\bfL_{-i,R}$; and
	\item[$\mathrm{LM5.}$]
        for all $i\in [h]$, the action of $\Pi=\pi\otimes 1 \in O_F\otimes_{O_{F_0}} R$ on $\cF_i$ satisfies the following signature conditions:
     \begin{itemize}
     	\item(Kottwitz condition) The characteristic polynomial satisfies:
\begin{equation*}
    \mathrm{char}(\Pi\mid\cF_i)=(T-\pi)(T+\pi)^{n-1};
\end{equation*}
\item(Wedge condition) The following operators vanish:
\begin{equation*}
    \bigwedge^2(\Pi+\pi \mid \cF_i)=0;
    \quad
    \bigwedge^n(\Pi-\pi \mid \cF_i)=0.
\end{equation*}
\item(Spin condition) 
For even $n=2m$ and $h=n$, the operator $\Pi+\pi$ is non-zero on $\cF_m$
\end{itemize}
    
\end{altitemize}
\item We define the \emph{local model} $\cM_n^{[h]}$ as the flat closure of $\cM_n^{[h],\wedge}$. To be more precise, it is the scheme-theoretic closure of the generic fiber of the wedge local model:
\begin{equation*}
\begin{aligned}
\xymatrix{
\cM_{n}^{[h]}\ar@{^(->}[r]&\cM_{n}^{[h],\wedge}\\
\cM_{n,\eta}^{[h],\wedge}.\ar@{^(->}[u]\ar@{^(->}[ur]&
}
\end{aligned}
\end{equation*}
\end{altenumerate}
\end{definition}

\subsection{Strata local model}
In this subsection, we define the strata local model and study its basic properties.

\begin{definition}
Let $R$ be an $\F_q$-algebra and let $\bfL_{[h]}$ and $\bfL_{[h,t]}$ be two standard lattice chains.
\begin{altenumerate}
\item The \emph{pivoting filtration\footnote{This terminology is due to S. Kudla.} of type $t$} is the fixed filtration with indices in $[t]$:
\begin{equation*}
    \left(\cF_i:=\Pi\bfL_{i,R}\subset \bfL_{i,R}\right)_{i\in[t]}.
\end{equation*}
\item For any $R$-point $\left(\cF_i \subset \bfL_{i,R}\right)_{i\in [h]}$ in $\cM^{[h]}_{n,\F}$, we say the point is \emph{pinned by the pivoting filtration} of type $t$ if
for any $i<j$ with either $i\in[h], j\in[t]$ or $i\in[t], j\in[h]$, the natural morphism  $\bfL_{i,R}\to \bfL_{j,R}$ maps $\cF_i$ into $\cF_j$. Here, when $i\in [t]$, $\cF_i$ refers to the pivoting filtration.

\item The \emph{strata local model}
$\ov{\cM}^{[h]}_{n}(t)$ is defined as the closed subscheme of the special fiber $\ov{\cM}^{[h]}_{n}$ that parameterizes all points pinned by the pivoting filtration of type $t$.
\end{altenumerate}
\end{definition}

We will relate these to the Bruhat--Tits strata in the next subsection.

\begin{theorem}\label{thm:LM-reduced}
The strata local model $\ov{\cM}^{[h]}_{n}(t)$ is reduced of finite type.
\end{theorem}
\begin{proof}
It is clear that strata local models are of finite type, since they can be embedded into some partial flag varieties. 
To prove the reducedness,
we use the theory of unitary affine flag varieties $\cF\ell_I$ as developed in \cite{PR2008}.
Let $\F((u))/\F((t))$ be the ramified quadratic extension of function fields with $u^2=t$. The standard polarized lattice chains $\bL_i$ for $\F((u))/\F((t))$ are defined as in \S \ref{sec:standard-lattice}.
Via the isomorphism $\bfL_i/\pi_0\bfL_i\simeq u^{-1}\bL_i/u\bL_i$, one can embed the special fiber of the local model $\ov{\cM}_n^{[h]}$ into the affine flag variety.
The pivoting filtration $\left(\Pi\bfL_{i,\F}\subset \bfL_{i,\F}\right)_{i\in[t]}$ lifts to the geometric point $\left(
u\bL_i\subset \bL_i\subset u^{-1}\bL_i
\right)_{i\in [t]}=:*_{[t]}\in\Fl_{[t]}(\F)$, which is an $L^+P_{[t]}$-invariant point.
Consider the following diagram:
\begin{equation*}
\begin{aligned}
\xymatrix{
&\Fl_{[h,t]}\ar[dl]_{p_h}\ar[dr]^{p_t}&\\
\ov{\cM}^{[h]}_{n}\subset \Fl_{[h]}&& \Fl_{[t]}\ni *_{[t]}
}
\end{aligned}
\end{equation*}
From the moduli description, we have the equality:
\begin{equation*}
    \ov{\cM}^{[h]}_{n}(t)=p_h^{-1}(\ov{\cM}^{[h]}_{n})\cap p_t^{-1}(*_{[t]}).
\end{equation*}
To be more precise, an $R$-point of $p_h^{-1}(\ov{\cM}^{[h]}_{n})$ is a filtration 
$\left(\cF_i\subset \bfL_{i,R}\right)_{i\in[h,t]}$ such that the subfamily $\left(\cF_i\subset \bfL_{i,R}\right)_{i\in[h]}$ defines a point of $\ov{\cM}^{\loc}_{[h]}(R)$. 
The intersection $p_h^{-1}(\ov{\cM}^{[h]}_{n})\cap p_t^{-1}(*_{[t]})$ parameterizes those filtrations $\left(\cF_i\subset \bfL_{i,R}\right)_{i\in[h,t]}$ where $\left(\cF_i\subset \bfL_{i,R}\right)_{i\in[t]}$ is the pivoting filtration.

Since $*_{[t]}$ is $L^+P_{[t]}$-invariant, it is a Schubert variety. Thus, its preimage $p_t^{-1}(*_{[t]})$ is a union of Schubert varieties.
By the Coherence conjecture, the special fiber $\ov{\cM}^{[h]}_{n}\subset \Fl_{[h]}$ of the local model is a union of Schubert varieties (see \cite[Thm. 4.1]{Pappas_Rapoport_2009}).
Therefore, $\ov{\cM}^{[h]}_{n}(t)$ is an intersection of unions of Schubert varieties. It is compatibly Frobenius split in a sufficiently large Schubert variety and hence reduced, see \cite[\S 2]{Gortz2001}.
\end{proof}

One can also define the strata local model $\ov{\cM}^{[h]}_{n}(t_1,t_2)$ with respect to two classes of pivoting filtrations.
The cases where $t_1,t_2 < h$ or $t_1,t_2 > h$ are degenerate.
For $t_1 < h < t_2$, we obtain a closed subscheme of the special fiber of the local model that fits into the following diagram:
\begin{equation*}
\begin{aligned}
\xymatrix{
\bfL_{-\ttt_1,R}\ar[r]&\bfL_{-\tth,R}\ar[r]&\bfL_{-\ttt_2,R}\ar[r]&\bfL_{\ttt_2,R}\ar[r]&\bfL_{\tth,R}\ar[r]&\bfL_{\ttt_1,R}\\
\Pi\bfL_{-\ttt_1,R}\ar[r]\ar@{^(->}[u]&\cF_{-\tth,R}\ar@{^(->}[u]\ar[r]&\Pi\bfL_{-\ttt_2,R}\ar@{^(->}[u]\ar[r]&\Pi\bfL_{\ttt_2,R}\ar@{^(->}[u]\ar[r]&\cF_{\tth,R}\ar@{^(->}[u]\ar[r]&\Pi\bfL_{\ttt_1,R}\ar@{^(->}[u]
}
\end{aligned}.
\end{equation*}

The strata local model $\ov{\cM}^{[h]}_{n}(t_1,t_2)$ is again reduced.
This follows from the same strategy as in Theorem \ref{thm:LM-reduced}, namely, considering projections of affine flag varieties
\begin{equation*}
\begin{aligned}
\xymatrix{
&\Fl_{[h,t_1,t_2]}\ar[dl]_{p_{t_1}}\ar[d]^{p_h}\ar[dr]^{p_{t_2}}&\\
\Fl_{[t_1]}&\Fl_{[h]}&\Fl_{[t_2]}.
}	
\end{aligned}
\end{equation*}
The strata local model is now the intersection
\begin{equation*}
    \ov{\cM}^{[h]}_{n}(t_1,t_2)=p_h^{-1}(\ov{\cM}^{[h]}_{n})\cap p_{t_1}^{-1}(*_{[t_1]})\cap p_{t_2}^{-1}(*_{[t_2]}).
\end{equation*}

\begin{remark}
The strata local model can also be defined when $F/F_0$ is unramified. In this case, the pivoting filtration is the filtration $(\uV\Lambda/\pi\Lambda\subset \Lambda/\pi\Lambda)$,
where $\uV$ is the Verschiebung of the rational Dieudonn\'e module of the framing object. 
This lifts to a point $(\pi\Lambda \subset \uV\Lambda \subset \Lambda)$ in the unramified unitary affine flag variety as a closed Schubert cell.

Since the unramified unitary group splits over $W_{O_{F_0}}(\mathbb{F})[\frac{1}{\pi_0}]$, we can reduce the problem to the local model and the affine Grassmannian of $\GL_n$. In this case, the vertex lattices correspond to the ``vertex lattices'' $\cL_i^{\pm}$ studied in \cite{Cho18}.

For example, the strata local model $\ov{\cM}^{[0]}_{n}(t)$ for $t=2k+1$ parameterizes one-dimensional subspaces $\cF_0\subset\Lambda_{0,R}$ with the following factorization:
\begin{equation*}
\begin{aligned}
\xymatrix{
\bfL_{-k-1,R}\ar[r]&\bfL_{0,R}\ar[r]&\bfL_{k,R}\\
(0)\ar[r]\ar@{^(->}[u]&\cF_{0,R}\ar@{^(->}[u]\ar[r]\ar@{^(->}[u]&(0)\ar@{^(->}[u].
}	
\end{aligned}
\end{equation*}
While this does not correspond to any classical local model, the proof of Theorem \ref{thm:LM-reduced} still applies. A direct computation shows that $\ov{\cM}^{[0]}_{n}(t)$ is smooth.
\end{remark}

\subsection{Local model diagram}\label{sec:LM-diagram}
In this subsection, we establish the local model diagram connecting the Bruhat--Tits strata to the strata local model. This construction is conceptual in nature, and does not require a moduli-theoretic description.
\begin{definition}
Let $\Lambda$ be a vertex lattice.
\begin{altenumerate}
\item For $t(\Lambda)\geq h$, define $\widetilde{\cZ}(\Lambda)$ to be a projective scheme over $\F$ that represents the functor sending each $\F$-algebra $R$ to the set of tuples $(X,\iota,\lambda,\rho;f)$, where:
\begin{itemize}
\item\, $(X,\iota,\lambda,\rho)\in\cZ(\Lambda)(R)$ is an $R$-point of the $\cZ$-stratum $\cZ(\Lambda)$;
\item\, $f$ is an isomorphism between the standard lattice chain $\cL_{[h,t],R}$ and the lattice chain of de Rham realizations:
\begin{equation*}
f:\begin{aligned}
\xymatrix{
\bfL_{-\ttt,R}\ar[r]\ar[d]^{\sim}&\bfL_{-\tth,R}\ar[r]\ar[d]^{\sim}&\bfL_{\tth,R}\ar[r]\ar[d]^{\sim}&\bfL_{\ttt,R}\ar[d]^{\sim}\\
D(X_\Lambda)\ar[r]^{\rho_{\Lambda,*}}&D(X)\ar[r]^{\lambda_*}& D(X^\vee)\ar[r]^{\rho_{\Lambda,*}^\vee}& D(X_{\Lambda^\sharp}).
}	
\end{aligned}
\end{equation*}
\end{itemize}
\item For $t(\Lambda)\leq h$, define $\widetilde{\cY}(\Lambda^\sharp)$ to be a projective scheme over $\F$ that represents the functor sending each $\F$-algebra $R$ to the set of tuples $(X,\iota,\lambda,\rho;f)$, where:
\begin{itemize}
\item\, $(X,\iota,\lambda,\rho)\in\cY(\Lambda^\sharp)(R)$ is an $R$-point of the $\cY$-stratum $\cY(\Lambda^\sharp)$;
\item\, $f$ is an isomorphism between the standard lattice chain $\cL_{[h,t],R}$ and the lattice chain of de Rham realizations:
\begin{equation}\label{equ:trivilization}
f:\begin{aligned}
\xymatrix{
\bfL_{\ttt,R}\ar[r]\ar[d]^{\sim}&\bfL_{\tth,R}\ar[r]\ar[d]^{\sim}&\bfL_{n-\tth,R}\ar[r]\ar[d]^{\sim}&\bfL_{n-\ttt,R}\ar[d]^{\sim}\\
D(X_{\Lambda^\sharp})\ar[r]&D(X^\vee)\ar[r]^{\lambda^\vee_*}& D(X)\ar[r]& D(X_{\Lambda}).
}
\end{aligned}
\end{equation}
\end{itemize}
\end{altenumerate}
\end{definition}

We now construct the local model diagram for the $\cZ$-strata. The construction and properties for the $\cY$-strata follow analogously.
For a vertex lattice $\Lambda$ of type $t(\Lambda)\geq h$, we construct the following local model diagram: 
\begin{equation}\label{equ:local-model-diagram}
\begin{aligned}
\xymatrix{
&\tilde{\cZ}(\Lambda)\ar[ld]_{\varphi}\ar[rd]^{\pi}&\\
\cZ(\Lambda)&&\ov{\cM}^{[h]}_{n}(t).
}	
\end{aligned}
\end{equation}
The map $\varphi$ is defined by forgetting the trivialization $f$:
\begin{equation*}
    \varphi:\widetilde{\cZ}(\Lambda)\rightarrow\cZ(\Lambda),\quad (X,\iota,\lambda,\rho;f)\mapsto (X,\iota,\lambda,\rho).
\end{equation*}
The map $\pi$ is defined by passing the Hodge filtration to the standard lattice chain via $f$.
Let $\mathcal{G}:=\mathrm{Aut}(\bfL_{[h,t]})$ be the group scheme over $O_{F_0}$ consisting of automorphisms of $\bfL_{[h,t]}$ that are compatible with the $O_F$-action and the duality relation. By \cite[Thm. 3.16]{RZ}, $\cG\otimes_{O_{F_0}}O_F$ is a smooth group scheme over $O_F$. Its reduction $\cG_\F:=\cG\otimes_{O_{F_0}}\F$ acts on $\varphi$ by acting on the top row of \eqref{equ:trivilization}. 
Furthermore, given any two trivializations $f_1$ and $f_2$, the composition $f_1\circ f_2^{-1}$ defines an automorphism of $\bfL_{[h,t]}$. We therefore obtain the following result:
\begin{theorem}[\protect{\cite[Thm. 3.16]{RZ}}]
The morphism $\varphi$ is a $\mathcal{G}_\F$-torsor. In particular, $\varphi$ is a smooth morphism of relative dimension $\dim\mathcal{G}_\F$.\qed
\end{theorem}

\begin{remark}\label{rmk:not compatible with other lm diagram}
From the definition, we see that \eqref{equ:local-model-diagram} is not the pull-back of the local model diagram for Shimura varieties or RZ spaces. For example, in our situation, the local model diagram for an RZ space is given in \eqref{equ:RZ lm diagram}. The projection map $\widetilde{\cN}_{n,\ep}^{[h]}\to \cN_{n,\ep}^{[h]}$ parameterizes the trivialization \eqref{equ:RZ trivialization}. Therefore, it is an $\mathrm{Aut}(\bfL_{[h]})$-torsor rather than an $\mathrm{Aut}(\bfL_{[h,t]})$-torsor.
\end{remark}

The morphism $\pi$ is defined by attaching the Hodge filtration of the strict $ O_{F_0}$-modules to the lattice chain via $f$.
To be more precise, for an $R$-point of $\widetilde{\cZ}(\Lambda)$, we have morphisms 
\begin{equation*}
\begin{aligned}
\xymatrix{
X_{\Lambda,R}\ar[r]^{\rho_{\Lambda}}&X\ar[r]^{\lambda}& X^\vee\ar[r]^{\rho_{\Lambda^\sharp}}& X_{\Lambda^\sharp,R}.
}	
\end{aligned}
\end{equation*}
By taking the Hodge filtration, we obtain a chain of filtrations:
\begin{equation*}
\begin{aligned}
\xymatrix{
D(X_\Lambda)\ar[r]&D(X)\ar[r]&D(X^\vee)\ar[r]&D(X_{\Lambda^\sharp})\\
\Fil(X_\Lambda)\ar@{^(->}[u]\ar[r]&\Fil(X)\ar@{^(->}[u]\ar[r]&\Fil(X^\vee)\ar@{^(->}[u]\ar[r]&\Fil(X_{\Lambda^\sharp})\ar@{^(->}[u]
}	
\end{aligned}
\end{equation*}

Recall that we define $X_{\Lambda,R}=X_{\Lambda}\times_\F R$ and $X_{\Lambda^\sharp,R}=X_{\Lambda^\sharp}\times_\F R$. 
Since $\Lambda$ is $\tau$-stable, we have $\Pi M(X_\Lambda)=VM(X_{\Lambda})$, where $M(X_{\Lambda})=\Lambda$ by construction, see the discussion before Definition \ref{def:YZ-cycles}. The same relation holds for $X_{\Lambda^\sharp}$. Passing these relations to the Hodge filtrations, we get
\begin{equation*}
\Fil(X_\Lambda)=\Pi D(X_\Lambda),
\quad 
\Fil(X_{\Lambda^\sharp})=\Pi D(X_{\Lambda^\sharp}).
\end{equation*}
Applying $f$, we obtain the desired filtration in $\ov{\cM}^{[h]}_{n}(t)$:
\begin{equation*}
\begin{aligned}
\xymatrix{
\bfL_{-\ttt,R}\ar[r]&\bfL_{-\tth,R}\ar[r]&\bfL_{\tth,R}\ar[r]&\bfL_{\ttt,R}\\
\Pi\bfL_{-\ttt,R}\ar@{^(->}[u]\ar[r]&f(\Fil(X))\ar@{^(->}[u]\ar[r]&f(\Fil(X^\vee))\ar@{^(->}[u]\ar[r]&\Pi\bfL_{\ttt,R}\ar@{^(->}[u]
}	
\end{aligned}
\end{equation*}

\begin{proposition}[Grothendieck--Messing]
    The morphism $\pi$ is formally smooth of relative dimension equal to $\dim\mathcal{G}_{\F}$.
\end{proposition}
\begin{proof}
Let $R_0$ an $\F$-algebra and $\Spec R_0\hookrightarrow \Spec R$ be a first-order thickening. 
Let $x\in \ov{\cM}^{[h]}_{n}(t)(R)$ be an $R$-point with reduction $\ov{x}\in \ov{\cM}^{[h]}_{n}(t)(R_0)$. Let $f_0\in\widetilde{\cZ}(\Lambda)(R_0)$ be an $R_0$-point that lifts $\bar{x}$:
\begin{equation*}
\begin{aligned}
\xymatrix{
&&\widetilde{\cZ}(\Lambda)\ar[d]^{\pi}\\
\Spec R_0\ar[urr]^{f_0}\ar@{^(->}[r]&\Spec R\ar@{-->}[ur]_{f}\ar[r]^{x}&\ov{\cM}^{[h]}_{n}(t).
}	
\end{aligned}
\end{equation*}
To show formal smoothness, we need to construct a lift $f$ as shown in the diagram. The existence of such a lift is equivalent to finding compatible lifts of the following data:
\begin{align}
X_{\Lambda,R}\to X\to X^\vee\to X_{\Lambda^\sharp,R}\quad
\rightsquigarrow{}& 
\quad X_{\Lambda,R_0}\to \ov{X}\to \overline{X}^\vee\to X_{\Lambda^\sharp,R_0}.\label{lmd_GM:diagram1}\\
\begin{tikzcd}[ampersand replacement=\&,column sep = 0.8em]
    \bfL_{-\ttt,R}\arrow[r]\arrow[d,"\sim"]\&\bfL_{-\tth,R}\arrow[r]\arrow[d,"\sim"]\&\bfL_{\tth,R}\arrow[r]\arrow[d,"\sim"]\&\bfL_{\ttt,R}\arrow[d,"\sim"]\\
    \bD(X_\Lambda)_{R}\arrow[r]\&\bD(\overline{X})_{R}\arrow[r]\&\bD(\overline{X}^\vee)_{R}\arrow[r]\&\bD(X_{\Lambda^\sharp})_{R}.
\end{tikzcd}
\rightsquigarrow{}&
\begin{tikzcd}[ampersand replacement=\&,column sep = 0.8em]
    \bfL_{-\ttt,R_0}\arrow[r]\arrow[d,"\sim"]\&\bfL_{-\tth,R_0}\arrow[r]\arrow[d,"\sim"]\&\bfL_{\tth,R_0}\arrow[r]\arrow[d,"\sim"]\&\bfL_{\ttt,R_0}\arrow[d,"\sim"]\\
    \bD(X_\Lambda)_{R_0}\arrow[r]\&\bD(\overline{X})_{R_0}\arrow[r]\&\bD(\overline{X}^\vee)_{R_0}\arrow[r]\&\bD(X_{\Lambda^\sharp})_{R_0}.\label{lmd_GM:diagram2}
\end{tikzcd}.
\end{align}
By the (relative) Grothendieck--Messing theorem, lifting the diagram  \eqref{lmd_GM:diagram1} is equivalent to lifting the Hodge filtration:
\begin{align}
\begin{tikzcd}[ampersand replacement=\&,  column sep = 0.5em]
    \bD(X_\Lambda)_{R}\arrow[r]\&\bD(\ov{X})_{R}\arrow[r]\&\bD(\ov{X}^\vee)_{R}\arrow[r]\&\bD(X_{\Lambda^\sharp})_{R}\\
\Pi \bD(X_\Lambda)_{R}\arrow[u, hookrightarrow]\arrow[r]\&\cF_{X}\arrow[r]\arrow[u, hookrightarrow]\&\cF_{X^\vee}\arrow[u, hookrightarrow]\arrow[r]\&\Pi \bD(X_{\Lambda^\sharp})_{R}\arrow[u, hookrightarrow]
\end{tikzcd}
\rightsquigarrow{}&\begin{tikzcd}[ampersand replacement=\&,  column sep = 0.5em]
    \bD(X_\Lambda)_{R_0}\arrow[r]\&D(\ov{X})\arrow[r]\&D(\ov{X}^\vee)\arrow[r]\&\bD(X_{\Lambda^\sharp})_{R_0}\\
\Pi \bD(X_\Lambda)_{R_0}\arrow[u, hookrightarrow]\arrow[r]\&\Fil(\overline{X})\arrow[r]\arrow[u, hookrightarrow]\&\Fil(\overline{X}^\vee)\arrow[u, hookrightarrow]\arrow[r]\&\Pi \bD(X_{\Lambda^\sharp})_{R_0}\arrow[u, hookrightarrow]
\end{tikzcd}.\label{lmd_GM:diagram3}
\end{align}
Combining \eqref{lmd_GM:diagram3} with \eqref{lmd_GM:diagram2}, we reduce the problem to finding a lifting:
\begin{align}
\begin{tikzcd}[ampersand replacement=\&,column sep = 0.8em]
    \bfL_{-\ttt,R}\arrow[r]\& \bfL_{-\tth,R}\arrow[r]\&\bfL_{\tth,R}\arrow[r]\& \bfL_{\ttt,R}\\
\Pi\bfL_{-\ttt,R}\arrow[u, hookrightarrow]\arrow[r]\&\widetilde\cF_{-\tth}\arrow[r]\arrow[u, hookrightarrow]\&\widetilde\cF_{\tth}\arrow[u, hookrightarrow]\arrow[r]\&\Pi\bfL_{\ttt,R}\arrow[u, hookrightarrow]
\end{tikzcd}
\rightsquigarrow{}&\begin{tikzcd}[ampersand replacement=\&,column sep = 0.8em]
    \bfL_{-\ttt,R_0}\arrow[r]\& \bfL_{-\tth,R_0}\arrow[r]\&\bfL_{\tth,R_0}\arrow[r]\& \bfL_{\ttt,R_0}\\
\Pi\bfL_{-\ttt,R_0}\arrow[u, hookrightarrow]\arrow[r]\&\cF_{-\tth}\arrow[r]\arrow[u, hookrightarrow]\&\cF_{\tth}\arrow[u, hookrightarrow]\arrow[r]\&\Pi\bfL_{\ttt,R_0}\arrow[u, hookrightarrow]
\end{tikzcd}.\label{lmd_GM:diagram4}
\end{align}
Such a lifting can be obtained from the map $x\rightsquigarrow \bar{x}$ by compatibility. Moreover, the uniqueness of the lifting $f$ follows directly from construction. Furthermore, we observe that there are no additional constraints on the isomorphism \eqref{lmd_GM:diagram2}, which implies that the map $\pi$ has relative dimension $\dim\mathcal{G}_\F$.
\end{proof}

\begin{corollary}\label{cor:red of BT}
    The moduli functors ${\cZ}(\Lambda)$ and ${\cY}(\Lambda^\sharp)$ are reduced. 
\end{corollary}
\begin{proof}
Consider the local model diagram \eqref{equ:local-model-diagram}. Since the morphisms  $\varphi$ and $\pi$ are smooth of equal dimension, each point in the BT-strata is \'etale locally isomorphic to a point in the corresponding stratum of the local model. The reducedness then follows from Theorem \ref{thm:LM-reduced}.
\end{proof}

\begin{remark}
\begin{altenumerate}
\item The local model diagram between $\cZ(\Lambda_1)\cap\cY(\Lambda_2^\sharp)$ and $\cM^{[h]}_{n,\F}(t_1,t_2)$ can be constructed and proved using analogous arguments.
\item When $F/F_0$ is unramified, the local model diagram can be constructed in the same manner, and the proofs follow the same line of reasoning.
\end{altenumerate}
\end{remark}

\subsection{Local chart computations}\label{sec:LM-compute}
The strata local model of BT-strata provides a powerful tool for studying the geometry of BT-strata without reference to Deligne--Lusztig varieties. In this subsection, we compute the local charts of the strata local models and establish Theorem \ref{thm:LM-results}.

\subsubsection{Strata local model of $\cZ$-strata}
We begin by computing the strata local model associated to $\cZ(\Lambda_1)$.
By an unramified base change, we can reduce to the case where the hermitian form is split. 

The strata local model admits an action of the loop group, as indicated in the proof of Theorem \ref{thm:LM-reduced}. Therefore, to deduce local properties and obtain the dimension formula in Theorem \ref{thm:LM-results}, it suffices to compute an affine chart at the worst point, which corresponds to the minimal orbit under the action of the loop group.
In our case, this is the unique closed orbit in the local model $\cM_n^{[h]}(\F)$. This closed orbit consists of the point corresponding to the filtration 
$$
(\cF_i\subset \bfL_{i,\F})_{i\in [h]}=(\Pi\Lambda_i\otimes\F\subset \bfL_{i,\F})_{i\in [h]},
$$
see \cite[Lemma 3.1.1]{Luo}. A direct calculation shows that this point lies in the strata local model.

In this case, the strata local model parameterizes lattice chains:
\begin{equation*}
\begin{aligned}
\xymatrix{
\bfL_{-\ttt_1,R}\ar[r]^{\lambda_1}&\bfL_{-\tth,R}\ar[r]^{\lambda}&\bfL_{\tth,R}\ar[r]^{\lambda_2}&\bfL_{\ttt_1,R}\\
\Pi\bfL_{-\ttt_1,R}\ar[r]\ar@{^(->}[u]&\cF_{-\tth}\ar[r]\ar@{^(->}[u]&\cF_{\tth}\ar[r]\ar@{^(->}[u]&\Pi\bfL_{\ttt_1,R}\ar@{^(->}[u].
}	
\end{aligned}
\end{equation*}
We select the same affine charts $U$ as those given in \cite[\S 3.1]{Luo}:
\begin{equation}\label{equ:ramified-Z-coordinate}
	\cF_\tth=\mathrm{colspan}
\begin{tikzpicture}[>=stealth,thick,baseline]
	\matrix [matrix of math nodes,left delimiter=(,right delimiter=)](A){ 
					D		&	0	&	C\\
					F	&	X_4		&	E\\
					B		&	0	&	A\\
                    I&&\\
                    &I&\\
                    &&I\\
				};
	\filldraw[purple] (2.25,1.3) circle (0pt) node [anchor=east]{\scriptsize $h/2$};
	\filldraw[purple] (2.4,0.8) circle (0pt) node [anchor=east]{\scriptsize $n-h$};
	\filldraw[purple] (2.25,0.3) circle (0pt) node [anchor=east]{\scriptsize $h/2$};
	\filldraw[purple] (-0.4,2) circle (0pt) node [anchor=east]{\scriptsize $h/2$};
	\filldraw[purple] (0.5,2) circle (0pt) node [anchor=east]{\tiny $n-h$};
	\filldraw[purple] (0.7,2) circle (0pt) node [anchor=center]{\scriptsize $h/2$};
\end{tikzpicture},
\quad
	\cF_{-\tth}=\mathrm{colspan}
\begin{tikzpicture}[>=stealth,thick,baseline]
	\matrix [matrix of math nodes,left delimiter=(,right delimiter=)](A){ 
					-D		&	E^{\ad}	&	-C\\
					0	&	X_4		&	0\\
					-B		&	-F^{\ad}	&	-A\\
                    I&&\\
                    &I&\\
                    &&I\\
				};
	\filldraw[purple] (2.85,1.3) circle (0pt) node [anchor=east]{\scriptsize $h/2$};
	\filldraw[purple] (3,0.8) circle (0pt) node [anchor=east]{\scriptsize $n-h$};
	\filldraw[purple] (2.85,0.3) circle (0pt) node [anchor=east]{\scriptsize $h/2$};
	\filldraw[purple] (-0.8,2) circle (0pt) node [anchor=east]{\scriptsize $h/2$};
	\filldraw[purple] (0.5,2) circle (0pt) node [anchor=east]{\scriptsize $n-h$};
	\filldraw[purple] (1,2) circle (0pt) node [anchor=center]{\scriptsize $h/2$};
\end{tikzpicture}.
\end{equation}
For notational convenience, we denote by $\cF_\bullet$ the matrix whose columns span  $\cF_\bullet$.
Let $X^{\ad}=HX^tH$, where $H$ is the antidiagonal identity matrix. The transition map $\lambda^\vee$ factors as the composition:
\begin{equation}\label{equ:Z-tran-zero}
\begin{aligned}
\xymatrix{\lambda^\vee:\bfL_{\tth,R}\ar[r]&
\bfL_{\ttt_1,R}\ar[r]&\bfL_{n-\ttt_1,R}\ar[r]&\bfL_{n-\tth,R}}	.
\end{aligned}
\end{equation}
We have $\lambda^\vee(\cF_\tth)=\Pi\bfL_{n-\tth,R}$.
By \cite[Thm. 6.3.2]{Luo}, we obtain $A=B=C=D=0$. The coordinates can be further refined as follows:
\begin{equation}
\cF_\tth=
\begin{tikzpicture}[>=stealth,thick,baseline]
	\matrix [matrix of math nodes,left delimiter=(,right delimiter=)](A){ 
    0&0&0&0&0\\
    f_1&Z_{11}&Z_{12}&Z_{13}&e_1\\
    f_2&Z_{21}&Z_{22}&Z_{23}&e_2\\
    f_3&Z_{31}&Z_{32}&Z_{33}&e_3\\
    0&0&0&0&0\\
    I&&&&\\
    &I&&&\\
    &&I&&\\
    &&&I&\\
    &&&&I\\
				};
	\filldraw[purple] (3.1,2.45) circle (0pt) node [anchor=east]{\tiny $h/2$};
    \filldraw[purple] (3.2,1.85) circle (0pt) node [anchor=east]{\tiny $\frac{t_1-h}{2}$};
	\filldraw[purple] (3.35,1.25) circle (0pt) node [anchor=east]{\tiny $n-t_1$};
    \filldraw[purple] (3.2,0.7) circle (0pt) node [anchor=east]{\tiny $\frac{t_1-h}{2}$};
	\filldraw[purple] (3.1,0.2) circle (0pt) node [anchor=east]{\tiny $h/2$};
	\filldraw[purple] (-1.25,3) circle (0pt) node [anchor=east]{\tiny $h/2$};
    \filldraw[purple] (-0.4,3) circle (0pt) node [anchor=east]{\tiny $\frac{t_1-h}{2}$};
	\filldraw[purple] (0.5,3) circle (0pt) node [anchor=east]{\tiny $n-t_1$};
    \filldraw[purple] (1.3,3) circle (0pt) node [anchor=east]{\tiny $\frac{t_1-h}{2}$};
	\filldraw[purple] (1.5,3) circle (0pt) node [anchor=center]{\tiny $h/2$};
\end{tikzpicture},
\quad
	\cF_{-\tth}=
\begin{tikzpicture}[>=stealth,thick,baseline]
	\matrix [matrix of math nodes,left delimiter=(,right delimiter=)](A){ 
    0&e_3^{\ad}&e_2^{\ad}&e_1^{\ad}&0\\
    0&Z_{11}&Z_{12}&Z_{13}&0\\
    0&Z_{21}&Z_{22}&Z_{23}&0\\
    0&Z_{31}&Z_{32}&Z_{33}&0\\
    0&-f_3^{\ad}&-f_2^{\ad}&-f_1^{\ad}&0\\
    I&&&&\\
    &I&&&\\
    &&I&&\\
    &&&I&\\
    &&&&I\\
				};
	\filldraw[purple] (3.4,2.45) circle (0pt) node [anchor=east]{\tiny $h/2$};
    \filldraw[purple] (3.5,1.85) circle (0pt) node [anchor=east]{\tiny $\frac{t_1-h}{2}$};
	\filldraw[purple] (3.6,1.25) circle (0pt) node [anchor=east]{\tiny $n-t_1$};
    \filldraw[purple] (3.5,0.7) circle (0pt) node [anchor=east]{\tiny $\frac{t_1-h}{2}$};
	\filldraw[purple] (3.4,0) circle (0pt) node [anchor=east]{\tiny $h/2$};
	\filldraw[purple] (-1.6,3) circle (0pt) node [anchor=east]{\tiny $h/2$};
    \filldraw[purple] (-0.7,3) circle (0pt) node [anchor=east]{\tiny $\frac{t_1-h}{2}$};
	\filldraw[purple] (0.55,3) circle (0pt) node [anchor=east]{\tiny $n-t_1$};
    \filldraw[purple] (1.5,3) circle (0pt) node [anchor=east]{\tiny $\frac{t_1-h}{2}$};
	\filldraw[purple] (1.8,3) circle (0pt) node [anchor=center]{\tiny $h/2$};
\end{tikzpicture}.
\end{equation}
With respect to this partition, we can express the transition matrices as:
\begin{equation*}
\lambda_1=\left(\begin{matrix}
    I_{n-\ttt_1}&0&0&0&0&0\\
    0&0&0&0&0&0\\
    0&0&I_h&0&0&0\\
    0&0&0&I_{n-\ttt_1}&0&0\\
    0&I_{\ttt_1-\tth}&0&0&0&0\\
    0&0&0&0&0&I_\tth
\end{matrix}\right),
\quad
\lambda_2=\left(\begin{matrix}
    I_{\tth}&0&0&0&0&0\\
    0&0&0&0&0&0\\
    0&0&I_{n-\ttt_1}&0&0&0\\
    0&0&0&I_h&0&0\\
    0&I_{\ttt_1-\tth}&0&0&0&0\\
    0&0&0&0&0&I_{n-\ttt_1}
\end{matrix}\right).
\end{equation*}
Therefore, $\lambda_1(\Pi\bfL_{-\ttt_1})\subset\cF_{-\tth}$ implies that
\begin{align*}
    &Z_{11}=0,Z_{12}=0,Z_{21}=0,Z_{22}=0,Z_{31}=0,Z_{32}=0,\\
    &e_{3}^{\ad}=0,e_{2}^{\ad}=0,-f_{3}^{\ad}=0,-f_2^{\ad}=0.
\end{align*}
The inclusion $\lambda_2(\cF_\tth)\subset\Pi\bfL_{\ttt_1}$ implies that
\begin{align*}
    &Z_{21}=0,Z_{22}=0,Z_{23}=0,Z_{31}=0,Z_{32}=0,Z_{33}=0,\\
    &f_2=0,f_3=0,e_2=0,e_3=0.
\end{align*}
Therefore, the only nonzero matrices are $Z_{13},f_1$ and $e_1$.
By the local model axioms, we have $X_4=X_4^{\ad},\tr(X_4)=0$, and $\wedge^2X=0$ (\cite[Prop. 4.1.1]{Luo}). After simplification, the affine coordinate ring of the open affine chart $U$ is isomorphic to:
\begin{equation*}
    \frac{\F[Z_{13},e_1,f_1]}{Z_{13}-Z_{13}^{\ad},\wedge^2(e_1,f_1,Z_{13})}.
\end{equation*}
This ring defines a symmetric determinantal variety corresponding to a $(\ttt_1+\tth)\times (\ttt_1-\tth)$ matrix.
By \cite[Thm. 4.2.2]{Luo}, this variety is normal and Cohen--Macaulay of dimension $\ttt_1+\tth$. Moreover, it is smooth when  $\ttt_1-\tth=1$ and singular otherwise.

Recall the following result of Conca:
\begin{proposition}[\cite{Conca94}]
Let  $X=(Y,Z)$ be a matrix of indeterminates, where $Y$ is an $m \times m$ matrix and $Z$ is an $m \times (n-m)$ matrix. The ring
\begin{equation*}
    \frac{k[X]}{\bigwedge^2 X,Y-Y^t}
\end{equation*}
is Gorenstein if and only if $2m=n+2$.
\end{proposition}
As a consequence, our stratum local model is Gorenstein if and only if  $2(\ttt_1-\tth)=(\ttt_1+\tth)+2$, i.e. $\ttt_1=3\tth+2$.

\subsubsection{Strata local model of $\cY$-strata}
We exclude the case where $n$ is even and $h=n$, as in this situation the local model does not have a worst point. The strata local model then parameterizes lattice chains
\begin{equation*}
\begin{aligned}
\xymatrix{
\bfL_{\ttt_2,R}\ar[r]^{\lambda_2}&\bfL_{\tth,R}\ar[r]^{\lambda^\vee}&\bfL_{n-\tth,R}\ar[r]^{\lambda_1}&\bfL_{n-\ttt_2,R}\\
\Pi\bfL_{\ttt_2,R}\ar[r]\ar@{^(->}[u]&\cF_{\tth}\ar[r]\ar@{^(->}[u]&\cF_{n-\tth}\ar[r]\ar@{^(->}[u]&\Pi\bfL_{n-\ttt_2,R}\ar@{^(->}[u].
}	
\end{aligned}
\end{equation*}
We choose the same affine charts as in the previous case. In particular, the formulas for  $\cF_\tth$ and $\cF_{n-\tth}$ remain as in equation \eqref{equ:ramified-Z-coordinate} (note that $\cF_{-\tth}$ has the same form as $\cF_{\tth}$ by the local model axiom).

By an argument similar to that leading to equation \eqref{equ:Z-tran-zero}, the transition map $\lambda$ carries $\cF_{-\tth}$ to $\Pi\Lambda_{\tth}$. By \cite[Thm. 6.3.2]{Luo}, this implies $X_4=0$. We can further refine the coordinates as follows
\begin{equation}
\cF_\tth=
\begin{tikzpicture}[>=stealth,thick,baseline]
	\matrix [matrix of math nodes,left delimiter=(,right delimiter=)](A){ 
    D_1&D_2&0&C_1&C_2\\
    D_3&D_4&0&C_3&C_4\\
    F_1&F_2&0&E_1&E_2\\
    B_1&B_2&0&A_1&A_2\\
    B_3&B_4&0&A_3&A_4\\
    I&&&&\\
    &I&&&\\
    &&I&&\\
    &&&I&\\
    &&&&I\\
				};
	\filldraw[purple] (3,2.45) circle (0pt) node [anchor=east]{\tiny $t_2/2$};
    \filldraw[purple] (3.15,1.85) circle (0pt) node [anchor=east]{\tiny $\frac{h-t_2}{2}$};
	\filldraw[purple] (3.15,1.25) circle (0pt) node [anchor=east]{\tiny $n-h$};
    \filldraw[purple] (3.15,0.7) circle (0pt) node [anchor=east]{\tiny $\frac{h-t_2}{2}$};
	\filldraw[purple] (3,0.1) circle (0pt) node [anchor=east]{\tiny $t_2/2$};
	\filldraw[purple] (-1.1,3) circle (0pt) node [anchor=east]{\scriptsize $\frac{t_2}{2}$};
    \filldraw[purple] (-0.3,3) circle (0pt) node [anchor=east]{\scriptsize $\frac{h-t_2}{2}$};
	\filldraw[purple] (0.45,3) circle (0pt) node [anchor=east]{\tiny $n-h$};
    \filldraw[purple] (1.2,3) circle (0pt) node [anchor=east]{\scriptsize $\frac{h-t_2}{2}$};
	\filldraw[purple] (1.3,3) circle (0pt) node [anchor=center]{\scriptsize $\frac{t_2}{2}$};
\end{tikzpicture},
\quad
	\cF_{n-\tth}=
\begin{tikzpicture}[>=stealth,thick,baseline]
	\matrix [matrix of math nodes,left delimiter=(,right delimiter=)](A){ 
    -D_1&-D_2&E_2^{\ad}&-C_1&-C_2\\
    -D_3&-D_4&E_1^{\ad}&-C_3&-C_4\\
    0&0&0&0&0\\
    -B_1&-B_2&-F_2^{\ad}&-A_1&-A_2\\
    -B_3&-B_4&-F_1^{\ad}&-A_3&-A_4\\
    I&&&&\\
    &I&&&\\
    &&I&&\\
    &&&I&\\
    &&&&I\\
				};
	\filldraw[purple] (3.9,2.45) circle (0pt) node [anchor=east]{\tiny $t_2/2$};
    \filldraw[purple] (4.0,1.85) circle (0pt) node [anchor=east]{\tiny $\frac{h-t_2}{2}$};
	\filldraw[purple] (4,1.25) circle (0pt) node [anchor=east]{\tiny $n-h$};
    \filldraw[purple] (4.0,0.7) circle (0pt) node [anchor=east]{\tiny $\frac{h-t_2}{2}$};
	\filldraw[purple] (3.9,0.2) circle (0pt) node [anchor=east]{\tiny $t_2/2$};
	\filldraw[purple] (-1.7,3) circle (0pt) node [anchor=east]{\tiny $t_2/2$};
    \filldraw[purple] (-0.5,3) circle (0pt) node [anchor=east]{\tiny $\frac{h-t_2}{2}$};
	\filldraw[purple] (0.6,3) circle (0pt) node [anchor=east]{\tiny $n-h$};
    \filldraw[purple] (1.6,3) circle (0pt) node [anchor=east]{\tiny $\frac{h-t_2}{2}$};
	\filldraw[purple] (2,3) circle (0pt) node [anchor=center]{\tiny $t_2/2$};
\end{tikzpicture}.
\end{equation}
With respect to this partition, the transition matrices can be represented as
\begin{equation*}
\lambda_1=\left(\begin{matrix}
    I_{n-\tth}&0&0&0&0&0\\
    0&0&0&0&0&0\\
    0&0&I_{\ttt_2} &0&0&0\\
    0&0&0&I_{n-h}&0&0\\
    0&I_{\tth-\ttt_2}&0&0&0&0\\
    0&0&0&0&0&I_{\ttt_2}
\end{matrix}\right),
\quad
\lambda_2=\left(\begin{matrix}
    I_{\ttt_2}&0&0&0&0&0\\
    0&0&0&0&0&0\\
    0&0&I_{n-\tth}&0&0&0\\
    0&0&0&I_{\ttt_2}&0&0\\
    0&I_{\tth-\ttt_2}&0&0&0&0\\
    0&0&0&0&0&I_{n-\tth}
\end{matrix}\right)
\end{equation*}
Now $\lambda_1(\cF_{n-\tth})\subset\Pi\Lambda_{n-\ttt_2}$ implies that
\begin{equation*}
    C=0,D=0,E=0,A_3=0,A_4=0,B_3=0,B_4=0,-F_1^{\ad}=0.
\end{equation*}
The inclusion $\lambda_2(\Pi\Lambda_{\ell})\subset\cF_h$ implies that
\begin{equation*}
    A=0,C=0,E=0,B_1=0,B_3=0,D_1=0,D_3=0,F_1=0.
\end{equation*}
The local model relations \cite[Prop. 4.1.5]{Luo} imply the following relations
\begin{equation*}
    \bigwedge^2 (B_2,F_2)=0, B=-\frac{1}{2}F^{\ad}F.
\end{equation*}
By the proof of Theorem 4.1.9 in loc.cit., we can further simplify the affine ring to
\begin{equation*}
    \frac{\F[F_2]}{\bigwedge^2 F_2}.
\end{equation*}
This defines a determinantal variety corresponding to a $(\tth-\ttt_2)\times (n-2\tth)$ matrix.
The variety is normal and Cohen--Macaulay of dimension $n-\tth-\ttt_2-1$. 
In the almost $\pi$-modular case, where $n-2\tth = 1$, the matrix $F_2$ reduces to a column vector, and the affine ring is smooth. Similarly, when $h-\ttt_2=1$, $F_2$ becomes a row vector, and the variety is again smooth.
In all other cases, the variety is singular. Moreover, it is Gorenstein if and only if  $(\tth-\ttt_2)=(n-2\tth)$, i.e. $\ttt_2=3\tth-n$.

\subsubsection{Intersection of $\cY$-strata and $\cZ$-strata}
We consider the intersection $\cZ(\Lambda_1)\cap \cY(\Lambda_2^\sharp)$ where $t(\Lambda_1)=t_1=2\ttt_1$ and $t(\Lambda_2)=t_2=2\ttt_2$.
This intersection is related to the strata local model $\ov{\cM}^{[h]}_{n}(t_1,t_2)$.
Maintaining the notation from equation \eqref{equ:ramified-Z-coordinate}, we further decompose:
\begin{equation*}
    f_1=(f_{11},f_{12}),
    \quad
    f_{11}:(\ttt_1-\tth)\times \tth,
    \quad
    f_{12}:(\ttt_1-\tth)\times (\tth-\ttt_2).
\end{equation*}
Combining the computations from the previous two subsections, we find that the affine chart of the strata local model is isomorphic to the spectrum of the ring:
\begin{equation*}
    \frac{\F[f_{12}]}{\bigwedge^2 f_{12}}.
\end{equation*}
This ring defines a variety that is normal and Cohen--Macaulay of dimension $\ttt_1-\ttt_2-1$. The variety is smooth if and only if $\ttt_1-\tth=1$ or $\tth-\ttt_2=1$. In all other cases, it is singular and is Gorenstein if and only if $\ttt_1-\tth=\tth-\ttt_2$.

\subsubsection{Strata local model of $\cY$-strata: $\pi$-modular case}
We now compute the affine charts of the $\pi$-modular local model. 
In this case, the local model does not contain a worst point, cf.\ \cite[Rem. 5.3]{Pappas_Rapoport_2009}.
We note that the affine chart chosen in \cite{Pappas_Rapoport_2009} does not apply to our situation.
The strata local model in this case parameterizes lattice chains
\begin{equation*}
\begin{aligned}
\xymatrix{
\bfL_{\ttt_2,R}\ar[r]^{\lambda_2}&\bfL_{m,R}\ar[r]^{\lambda_1}&\bfL_{n-\ttt_2,R}\\
\Pi\bfL_{\ttt_2,R}\ar[r]\ar@{^(->}[u]&\cF_{m}\ar[r]\ar@{^(->}[u]&\Pi\bfL_{n-\ttt_2,R}\ar@{^(->}[u].
}	
\end{aligned}
\end{equation*}
We choose the standard basis as given in \cite[(3.1.1)]{Luo}. 
Over the special fiber, the two transition maps are represented by the following matrices:
\begin{equation*}
    \lambda_2=\left(\begin{matrix}
        I_{\ttt_2}&0&0&0&0&0\\
        0&0&0&0&0&0\\
        0&0&I_m&0&0&0\\
        0&0&0&I_{\ttt_2}&0&0\\
        0&I_{m-\ttt_2}&0&0&0&0\\
        0&0&0&0&0&I_m
    \end{matrix}\right),
    \quad
    \lambda_1=\left(\begin{matrix}
        I_m&0&0&0&0&0\\
        0&0&0&0&0&0\\
        0&0&I_{\ttt_2}&0&0&0\\
        0&0&0&I_m&0&0\\
        0&I_{m-\ttt_2}&0&0&0&0\\
        0&0&0&0&0&I_{\ttt_2}
    \end{matrix}\right).
\end{equation*}
Consider the $\F$-span of the following vectors
\begin{equation}\label{equ:pi-modular chart}
    e_1,\cdots,e_{m-1},e_{m+1},\pi e_{m+1},\cdots,\pi e_n,
\end{equation}
It defines an $\F$-point of $\ov{\cM}^{[n]}_{n}$, as can be checked using the spin condition in \cite[Def. 3.9]{RSZ}.
We can embed the special fiber $\ov{\cM}^{[n]}_{n}$ into the Grassmannian $\Gr(n,\Lambda_{m,\F})$. 
Choose the affine chart $U$ of the point \eqref{equ:pi-modular chart} in the Grassmannian as
\begin{equation*}
\cF=
\begin{tikzpicture}[>=stealth,thick,baseline]
	\matrix [matrix of math nodes,left delimiter=(,right delimiter=)](A){ 
    X_{11}&X_{12}&X_{13}&X_{14}&X_{15}&X_{16}\\
    X_{21}&X_{22}&X_{23}&X_{24}&X_{25}&X_{26}\\
    Y_1&Y_2&Y_3&Y_4&Y_5&Y_6\\
    0&0&1&0&0&0\\
    X_{31}&X_{32}&X_{33}&X_{34}&X_{35}&X_{36}\\
    X_{41}&X_{42}&X_{43}&X_{44}&X_{45}&X_{46}\\
    I_{\ttt_2}&&&&&\\
    &I_{m-\ttt_2-1}&&&&\\
    Z_1&Z_2&Z_3&Z_4&Z_5&Z_6\\
    &&&1&&\\
    &&&&I_{m-\ttt_2-1}&\\
    &&&&&I_{\ttt_2}\\
				};
    \filldraw[purple] (4.7,3.2) circle (0pt) node [anchor=east]{\tiny $t_2/2$};
	\filldraw[purple] (5,2.55) circle (0pt) node [anchor=east]{\tiny $\frac{n-t_2-2}{2}$};
    \filldraw[purple] (4.5,2) circle (0pt) node [anchor=east]{\tiny $1$};
	\filldraw[purple] (4.5,1.35) circle (0pt) node [anchor=east]{\tiny $1$};
    \filldraw[purple] (5,0.9) circle (0pt) node [anchor=east]{\tiny $\frac{n-t_2-2}{2}$};
	\filldraw[purple] (4.7,0.3) circle (0pt) node [anchor=east]{\tiny $t_2/2$};
	\filldraw[purple] (-2.6,3.8) circle (0pt) node [anchor=east]{\tiny $t_2/2$};
    \filldraw[purple] (-1,3.8) circle (0pt) node [anchor=east]{\tiny $\frac{n-t_2-2}{2}$};
	\filldraw[purple] (-0.3,3.8) circle (0pt) node [anchor=east]{\tiny $1$};
    \filldraw[purple] (0.6,3.8) circle (0pt) node [anchor=east]{\tiny $1$};
	\filldraw[purple] (1.7,3.8) circle (0pt) node [anchor=center]{\tiny $\frac{n-t_2-2}{2}$};
    \filldraw[purple] (2.95,3.8) circle (0pt) node [anchor=center]{\tiny $t_2/2$};
\end{tikzpicture}.
\end{equation*}
Now the two subspaces $\lambda_2(\Pi\bfL_{\ell,R})$ and $\lambda_1(\cF)$ are spanned by the following matrices respectively:
\begin{equation*}
    \left(\begin{matrix}
        0&0&0&0&0&0\\
        0&0&0&0&0&0\\
        0&0&0&0&0&0\\
        0&0&0&0&0&0\\
        0&0&0&0&0&0\\
        0&0&0&0&0&0\\
        I_{\ttt_2}&0&0&0&0&0\\
        0&0&0&0&0&0\\
        0&0&0&0&0&0\\
        0&0&0&1&0&0\\
        0&&0&0&I_{m-\ttt_2-1}&0\\
        0&0&0&0&0&I_{\ttt_2}
    \end{matrix}\right),
    \quad
    \left(\begin{matrix}
        X_{11}&X_{12}&X_{13}&X_{14}&X_{15}&X_{16}\\
        X_{21}&X_{22}&X_{23}&X_{24}&X_{25}&X_{26}\\
        Y_1&Y_2&Y_3&Y_4&Y_5&Y_6\\
        0&0&0&0&0&0\\
        0&0&0&0&0&0\\
        X_{41}&X_{42}&X_{43}&X_{44}&X_{45}&X_{46}\\
        I_{\ttt_2}&0&0&0&0&0\\
        0&I_{m-\ttt_2-1}&0&0&0&0\\
        0&0&0&0&0&0\\
        0&0&0&1&0&0\\
        0&0&0&0&I_{m-\ttt_2-1}&0\\
        0&0&0&0&0&I_{\ttt_2}
    \end{matrix}\right).
\end{equation*}
Now $\lambda_2(\Pi\bfL_{\ttt_2})\subset\cF_m$ implies
\begin{align*}
    &X_{11}=0,X_{21}=0,X_{31}=0,X_{41}=0,Y_1=0,Z_1=0,\\
    &X_{14}=0,X_{24}=0,X_{34}=0,X_{44}=0,Y_4=0,Z_4=0,\\
    &X_{15}=0,X_{25}=0,X_{35}=0,X_{45}=0,Y_5=0,Z_5=0,\\
    &X_{16}=0,X_{26}=0,X_{36}=0,X_{46}=0,Y_6=0,Z_6=0.
\end{align*}
The inclusion $\lambda_1(\cF_m)\subset\Pi\bfL_{n-\ttt_2}$ implies
\begin{align*}
    &X_{11}=0,X_{12}=0,X_{13}=0,X_{14}=0,X_{15}=0,X_{16}=0,\\
    &X_{21}=0,X_{22}=0,X_{23}=0,X_{24}=0,X_{25}=0,X_{26}=0,\\
    &Y_1=0,Y_2=0,Y_3=0,Y_4=0,Y_5=0,Y_6=0,\\
    &X_{41}=0,X_{42}=0,X_{43}=0,X_{44}=0,X_{45}=0,X_{46}=0.
\end{align*}
By simplification, the subspace $\cF_m$ is spanned by the matrix of the form
\begin{equation*}
\cF_m=
\begin{tikzpicture}[>=stealth,thick,baseline]
	\matrix [matrix of math nodes,left delimiter=(,right delimiter=)](A){ 
    0&0&0&0&0&0\\
    0&0&0&0&0&0\\
    0&0&0&0&0&0\\
    0&0&1&0&0&0\\
    0&X_{32}&X_{33}&0&0&0\\
    0&0&0&0&0&0\\
    I_{\ttt_2}&&&&&\\
    &I_{m-\ttt_2-1}&&&&\\
    0&Z_2&Z_3&0&0&0\\
    &&&1&&\\
    &&&&I_{m-\ttt_2-1}&\\
    &&&&&I_{\ttt_2}\\
				};
\end{tikzpicture}.
\end{equation*}
Now we apply the remaining local model axioms. Due to our choice of an affine chart, the spin condition is automatically satisfied. The Kottwitz condition holds since the upper half of the matrix is nilpotent. The wedge condition implies that  $X_{32}=0$. 
It remains to interpret the isotropic condition. With respect to our chosen standard basis, the symmetric form  $(\,,\,)$ is represented by the matrix
\begin{equation*}
    M=\left(\begin{matrix}
        &&&-H\\
        &&H&\\
        &H&&\\
        -H&&&
    \end{matrix}\right),
\end{equation*}
where $H$ is the $m\times m$ anti-diagonal unit matrix.
The isotropic condition implies that $\cF_m^tM\cF_m=0$, which reduces to
\begin{equation*}
    A+A^t=0,
    \quad\text{where}\quad
    A=
    \left(\begin{matrix}
        I_{\ttt_2}&&\\
        &I_{m-\ttt_2-1}&Z_2^t\\
        &&Z_3
    \end{matrix}\right)
    H
    \left(\begin{matrix}
        0&0&1\\
        0&0&X_{33}\\
        0&0&0
    \end{matrix}\right).
\end{equation*}
This implies that $Z_3=0$ and $HX_{33}+Z_2^t=0$. Therefore, the affine chart of the strata local model is isomorphic to:
\begin{equation*}
    \Spec \F[X_{33}].
\end{equation*}
This defines a smooth variety of dimension  $m-\ttt_2-1$.
Since $\ov{\cM}^{[n]}_{n}(t)$ is contained in the preimage of the worst point $(\Pi\bfL_{\ttt_2,\F}\subset\bfL_{\ttt_2,\F},\Pi\bfL_{n-\ttt_2,\F}\subset\bfL_{n-\ttt_2,\F})$, the affine chart $U$ 
intersects all irreducible components of the strata local model. Therefore, the strata local model is smooth and irreducible of dimension $m-\ttt_2-1$.

\section{Bruhat--Tits stratification}\label{sec:BT}
In this section, we begin by proving a crucial lemma, and then proceed to define the Bruhat--Tits stratification of the reduced locus $\cN^{[h]}_{n,\ep,\red}$. In the final subsection, we establish its relationship with the Kottwitz--Rapoport strata.

\subsection{The crucial lemma}
In this subsection, we prove the crucial lemma in our case (Proposition \ref{prop:crutial-lem}), extending results in \cite[Lem. 2.1]{Vollaard} and \cite[Prop. 4.1]{RTW}. 
In our setting, the statement of the crucial lemma is more complicated and the proof is more involved. 
A key distinction in our case is that it is possible that only one of $T_c(M)$ or $T_d(M^\sharp)$ (to be defined below) is a vertex lattice.

Let $\kappa$ be any perfect field over $\F$. 
We denote by $M=M(X)$ the relative Dieudonn\'e module of $(X,\iota_X,\lambda_X,\rho_X)\in \cN^{[h]}_{n,\ep}(\kappa)$. By Proposition \ref{prop:RZ-lattice}, we have
\begin{equation}\label{equ:increasing sequence}
\pi M^\sharp \stackrel{n-h}{\subseteq} M\stackrel{h}{\subseteq} M^\sharp,\quad\text{and}\quad
M\stackrel{\leq 1}{\subseteq} M+\tau(M).
\end{equation}

Recall that we have $(M^\sharp)^\sharp=M$ and $\tau(M^\sharp)=(\tau(M))^\sharp$.   We denote by $T_i(M)$ the sum $M+\tau(M)+\cdots + \tau^i(M)$ and define $T_i(M^\sharp)$ similarly by replacing $M$ with $M^\sharp$. This gives two increasing chains:
\begin{equation*}
M\stackrel{\leq 1}{\subseteq}T_1(M) \stackrel{\leq 1}{\subseteq} T_2(M) \stackrel{\leq 1}{\subseteq}\cdots \quad\text{and}\quad 
M^\sharp\stackrel{\leq 1}{\subseteq}T_1(M^\sharp) \stackrel{\leq 1}{\subseteq} T_2(M^\sharp) \stackrel{\leq 1}{\subseteq}\cdots 
\end{equation*}
By Zink's lemma \cite[Prop. 2.17]{RZ}, both chains eventually stabilize, i.e., for sufficiently large $i$, we have $T_i(M)=T_{i+1}(M)$ and $T_i(M^\sharp)=T_{i+1}(M^\sharp)$, or equivalently $T_i(M)$ and $T_i(M^\sharp)$ are $\tau$-stable.
\textbf{From now on, we  always define $c$ (resp. $d$) to be the smallest non-negative integer such that $T_c(M)$ (resp. $T_d(M^\sharp)$) is $\tau$-invariant.}
This implies the existence of the following strictly increasing chains:
\begin{equation*}
M\stackrel{1}{\subset}T_1(M) \stackrel{1}{\subset} T_2(M) \cdots \stackrel{1}{\subset} T_c(M),\quad\text{and}\quad M^\sharp\stackrel{1}{\subset}T_1(M^\sharp) \stackrel{1}{\subset} T_2(M^\sharp) \cdots \stackrel{1}{\subset} T_d(M^\sharp).
\end{equation*}

\begin{lemma}\label{lem: same ind}
$[M+\tau (M):M]=1$ implies that $[M^\sharp +\tau(M^\sharp) : M^\sharp]=1$.
\end{lemma}
\begin{proof}
By taking the dual, we have $[M^\sharp +\tau(M^\sharp) : M^\sharp]=[M: M\cap \tau(M)]$. Moreover,  $[M: M\cap \tau(M)]=[M+\tau (M):M]=1$.  
\end{proof}

\begin{lemma}[\protect{\cite[Lem. 4.7]{HZBasic}}]\label{lem: inf inter}
We  have
\begin{altenumerate}
\item $T_c(M)\subseteq   \bigcap_{i\in \Z_{\ge 0}}\tau^i(\pi^{-1} M)$ and $     T_d(M^\sharp)\subseteq   \bigcap_{i\in \Z_{\ge 0}}\tau^i(\pi^{-1} M^\sharp). $
		 
\item If $\tau(M)\subseteq M^\sharp$ and $c\le d$, then  
\begin{align*}
T_c(M)\subseteq   \bigcap_{i\in \Z_{\ge 0}}\tau^i(M^\sharp ).
\end{align*}
\item If $\tau(\pi M^\sharp)\subseteq M$ and $d\le c$, then  
\begin{align*}
T_d(\pi M^\sharp)\subseteq   \bigcap_{i\in \Z_{\ge 0}}\tau^i(M ).
\end{align*}
\end{altenumerate}
\end{lemma}
\begin{proof}
For (1), we only prove the statement for $T_c(M)$ since the proof for $T_d(M^\sharp)$ is exactly the same.
First, note that
\begin{align*}
\tau(M)\subseteq \pi^{-1} M.
\end{align*}
Hence
\begin{align*}
T_c(M)\subseteq \pi^{-1} T_{c-1}(M).
\end{align*}
Since $T_c(M)$ is $\tau$-invariant, we have 
\begin{align*}
T_c(M)\subseteq \pi^{-1}  \tau(T_{c-1}(M))\cap \pi^{-1} T_{c-1}(M)=\pi^{-1}\tau(T_{c-2}(M)).
\end{align*}
Here, for the last equality,  we use the facts that:
\begin{altenumerate2}
\item $\tau(T_{c-1}(M))\neq T_{c-1}(M)$,
\item $\tau(T_{c-2}(M))\subseteq \tau(T_{c-1}(M))\cap T_{c-1}(M)$, and 
\item $[T_{c-1}(M):T_{c-2}(M)]=[\tau(T_{c-1}(M)):\tau(T_{c-2}(M))]=1$.
\end{altenumerate2}
  
Since $T_c(M)$ is $\tau$-invariant, we get
\begin{align*}
T_c(M)\subseteq \pi^{-1}T_{c-2}(M).
\end{align*}
Inductively, we have $T_c(M)\subseteq \pi^{-1} M$. Since $T_c(M)$ is $\tau$-invariant, we have
\begin{align*}
T_c(M)\subseteq  \pi^{-1} \bigcap_{i\in \Z_{\ge 0}}\tau^i(M)\subseteq  \pi^{-1} \bigcap_{0\le i\le f}\tau^i(M)
\end{align*}
for any $f\in \Z_{\ge 0}$.

Parts (2) and (3) are essentially the same as (1), once we observe that the assumption $c\le d$ implies that
$\tau(T_{c-1}(M^\sharp))\neq T_{c-1}(M^\sharp)$ and $\tau(T_{c-1}(M^\sharp))\cap T_{c-1}(M^\sharp)=\tau(T_{c-2}(M^\sharp))$ hold. Similarly, $d\le c$ implies that $\tau(T_{d-1}(M))\cap T_{d-1}(M)=\tau(T_{d-2}(M))$.
\end{proof}

\begin{proposition}[\protect{\cite[Prop. 4.8]{HZBasic}}]\label{prop:crutial-lem}
Assume $X\in \cN^{[h]}_{n,\ep}(\kappa)$ and $M=M(X)$. Then one of the following holds: 
\begin{altenumerate}
\item[(Case $\cY$)\,]: $\pi (T_c(M))^\sharp \subseteq \pi M^\sharp \subseteq M\subseteq   T_c(M) \subseteq (T_c(M))^\sharp \subseteq M^\sharp.$ In particular, $T_c(M)$ is a vertex lattice of type $t\le h$.
\item[(Case $\cZ$)\,]: $\pi M^\sharp \subseteq \pi T_d(M^\sharp)\subseteq  T_d(M^\sharp)^\sharp  \subseteq M\subseteq   M^\sharp\subseteq T_d(M^\sharp)$. In particular, $T_d(M^\sharp)^\sharp$ is a vertex lattice of type $t\ge h$.

\end{altenumerate}
More precisely,  we have:
\begin{enumerate}
\item   If $\tau(\pi M^\sharp)\not\subset   M$, then $M$ satisfies (Case $\cY$).
\item  If $\tau(M)\not\subset M^\sharp$, then $M$ satisfies (Case $\cZ$).
\item If $\tau(\pi M^\sharp) \subseteq M$ and $\tau(M) \subseteq M^\sharp$, then $M$ satisfies (Case $\cY$)   if $c\le d$.

\item If $\tau(\pi M^\sharp) \subseteq M$ and $\tau(M) \subseteq M^\sharp$, then $M$ satisfies (Case $\cZ$)   if $d\le c$.
        
\end{enumerate}
\end{proposition}
\begin{proof}
\begin{altenumerate}
\item Assume $\tau(  M^\sharp)\not\subset \pi^{-1} M$. Since $M\subseteq   T_c(M)$, it suffices to show $  T_c(M)\subseteq T_c(M)^\sharp$. 
Since $ M \stackrel{\le 1}{\subset}  M + \tau(M)$, $\tau(\pi M^\sharp)\not\subset  M $ and $\tau(\pi M^\sharp)\subset  \tau(M)$, we have  $M+\tau(M)= M+\tau(\pi M^\sharp)$. In fact, an inductive argument on $T_i(M)$ shows that 
\begin{align}\label{eq: 33} T_c(M)=M+\tau(\pi M^\sharp)+\cdots +\tau^c(\pi M^\sharp).
\end{align}
Equivalently, we have
\begin{align}\label{eq: 66}
T_c(M)^\sharp=M^\sharp\cap  \bigl(\bigcap_{1\le i\le c} \tau^i(\pi^{-1} M)\bigr) .
\end{align}
According to Lemma \ref{lem: inf inter}, we have
\begin{align}\label{eq: 11}
\begin{split}
T_c(\pi M^\sharp)&\subseteq     \bigcap_{0\le i\le c}\tau^i(M^\sharp) \subseteq   M^\sharp\cap  \bigl(\bigcap_{1\le i\le c}\tau^i(\pi^{-1}M)\bigr),\\
M&\subseteq M^\sharp\cap T_c(M)\subseteq M^\sharp\cap \bigl(  \bigcap_{1\le i\le c}\tau^i(\pi^{-1}M)\bigr).
\end{split}
\end{align}
Therefore, we have
\begin{align*}
T_c(M)\stackrel{\eqref{eq: 33}}{\subseteq} M+T_c(\pi M^\sharp) \stackrel{\eqref{eq: 11}}{\subseteq}M^\sharp\cap (  \bigcap_{1\le i\le c}\tau^i(\pi^{-1}M))\stackrel{\eqref{eq: 66}}{=}T_c(M)^\sharp.
\end{align*}

\item Now we assume  $\tau(M)\not\subset M^\sharp$. By construction, we have $T_d(M^\sharp)^\sharp\subseteq M\subseteq   M^\sharp\subseteq T_d(M^\sharp)$. Therefore, we only need to show $\pi T_d(M^\sharp)\subseteq T_d(M^\sharp)^\sharp$. 		
Since $M^\sharp \stackrel{\le 1}{\subset} M^\sharp +\tau(M^\sharp)$ and $\tau(M)\not\subset M^\sharp$, we have  $M^\sharp +\tau(M^\sharp)=M^\sharp +\tau(M)$. In fact, an inductive argument shows that 
\begin{align}\label{eq: 3}
T_d(M^\sharp)=M^\sharp+\tau(M)+\cdots +\tau^d(M).
\end{align}
Equivalently, this implies
\begin{align}\label{eq: 6}
T_d(M^\sharp)^\sharp=M\cap \bigl(\bigcap_{1\le i\le d}\tau^i(M^\sharp)\bigr).
\end{align}
According to Lemma \ref{lem: inf inter}, we have
\begin{align}\label{eq: 1}
\begin{split}
T_d(\pi M)&\subseteq  \bigcap_{0\le i\le d}\tau^i(M) \subseteq   M \cap  \bigl(\bigcap_{1\le i\le d}\tau^i( M^\sharp)\bigr),\\
\pi M^\sharp &\subseteq  M\cap T_d(\pi M^\sharp)\subseteq M \cap \bigl(  \bigcap_{1\le i\le d}\tau^i( M^\sharp)\bigr).
\end{split}
\end{align}
Therefore, we have
\begin{align*}
   \pi  T_d(M^\sharp)\stackrel{\eqref{eq: 3}}{\subseteq} \pi M^\sharp+  T_d(\pi  M) \stackrel{\eqref{eq: 1}}{\subseteq}M \cap \bigl(  \bigcap_{1\le i\le d}\tau^i( M^\sharp)\bigr)\stackrel{\eqref{eq: 6}}{=}T_d(M^\sharp)^\sharp.
\end{align*}

\item Assuming $\tau(M) \subseteq M^\sharp$ and $c\le d$, we will show that $M$ satisfies (Case $\cY$) (which is a stronger statement than $(3)$ in the assertion).  Note that $M+\tau(M)\subseteq M^\sharp$. Inductively, we have $T_c(M)\subseteq T_{c-1}(M^\sharp)$. Since $T_c(M)$ is $\tau$-invariant,  we have $T_c(M)\subseteq \bigcap_i \tau^i(M^\sharp)\subseteq \bigcap_{0\le i\le c} \tau^i(M^\sharp) \subseteq (T_c(M))^\sharp $ by   Lemma \ref{lem: inf inter}. Hence, we have
\begin{align*}
M\subseteq T_c(M) \subseteq (T_c(M))^\sharp \subseteq M^\sharp.
\end{align*}
Since $\pi M^\sharp \subseteq M \subseteq M^\sharp$, we have 
\begin{align*}
\pi (T_c(M))^\sharp \subseteq \pi M^\sharp \subseteq M\subseteq   T_c(M) \subseteq (T_c(M))^\sharp \subseteq M^\sharp.
\end{align*}
\end{altenumerate}
The proof of $(4)$ follows from the same argument as the proof of $(3)$.
\end{proof}

\subsection{Bruhat--Tits stratification}
In this subsection, we define the Bruhat--Tits stratification. Let $\Lambda\subset \mbV$ be a vertex lattice of type $t$. 
Since $t$ is an even integer, we write $t=2\ttt$.
Recall from \S \ref{sec:notation} that for any lattice $\Lambda$, we let $\breve{\Lambda}:=\Lambda\otimes_{O_F}O_{\breve{F}}$. 
In \S \ref{sec:notation}, we also defined two important quotient spaces: $V_{\Lambda}:=\Lambda^\sharp/\Lambda$ equipped with an alternating form and $V_{\Lambda^\sharp}:=\Lambda/\pi\Lambda^\sharp$ equipped with a symmetric form.
By extending scalars to $\F$, we obtain $\Omega_{\Lambda}:=V_\Lambda\otimes_{\F_q}\F$ and $\Omega_{\Lambda^\sharp}:=V_{\Lambda^\sharp}\otimes_{\F_q}\F$.
The alternating form on $\Omega_\Lambda$ and the symmetric form on $\Omega_{\Lambda^\sharp}$ are defined by extending the corresponding forms on $V_\Lambda$ and $V_{\Lambda^\sharp}$ via scalar extension.

From Proposition \ref{prop:RZ-lattice}, we have a lattice-theoretic description of $\cN_{n,\ep,\red}^{[h]}(\kappa)$, where $\kappa$ is a perfect field over $\F$.
These descriptions naturally extend to a lattice-theoretic characterization of the $\kappa$-point of the Bruhat--Tits strata.
\begin{proposition}\label{prop:BT_strata}
Let $\kappa$ be any perfect field over $\F$.
The $\kappa$-points of the Bruhat--Tits strata can be described as follows:
\begin{altenumerate}
\item Assume $\Lambda$ is a vertex lattice of type $t\ge h$. Then 
\begin{align*}
    \cZ(\Lambda)(\kappa)=\{(X,\iota_X,\lambda_X,\rho_X)\in \cN^{[h]}_{n,\ep}(\kappa)\mid    \Lambda\otimes_O W_O(\kappa) \subseteq M(X)\subseteq M(X)^\sharp\subseteq \Lambda^\sharp\otimes_O W_O(\kappa)  \}.
\end{align*}
\item Assume $\Lambda$ is a vertex lattice of type $t\le h$. Then 
\begin{align*}
    \cY(\Lambda^\sharp)(\kappa)=\{(X,\iota_X,\lambda_X,\rho_X)\in \cN^{[h]}_{n,\ep}(\kappa)\mid \pi \Lambda^\sharp\otimes W_O(\kappa) \subseteq \pi M(X)^\sharp \subseteq M(X)\subseteq   \Lambda\otimes W_O(\kappa)\}.
\end{align*}
\end{altenumerate}
\end{proposition}
\begin{proof}
We prove part (1), part (2) follows form the same arguments. 
By Proposition \ref{prop:rz-top flat} and Definition \ref{def:YZ-cycles}, the set of $\kappa$-points in $\cZ(\Lambda)$ is identical to the $\kappa$-points of the corresponding subfunctor defined in $\ov{\cN}_{n,\ep}^{[h],\wedge}$.
By composing with the framing map, the Dieudonn\'e module $M(X)$ is identified  as a lattice in $\bV\otimes_O W_O(\kappa)=\bV\otimes_{F_0}W_O(\kappa)[1/\pi_0]$. 
Since the quasi-isogeny $\rho_{\Lambda,X}$ is an isogeny, it ensures the inclusion $\Lambda\otimes_OW_O(\kappa)=M(X_{\Lambda,\kappa})\subset M(X)$. The other inclusion follows by taking the dual.
\end{proof}

\begin{theorem} \label{thm:BT-decomp}
Recall that $\cL_{\cZ}$ (resp. $\cL_{\cY}$) denotes the set of all vertex lattices in $\bV$ of type $\geq h$ (resp. $\leq h$).
Let $\kappa$ be any perfect field over $\F$, then we have the following:
\begin{altenumerate}
\item 
The reduced locus $\cN_{n,\ep,\red}^{[h]}$ is the union of closed subvarieties:
$$
\cN_{n,\ep}^{[h]}(\kappa)=\Bigl(\bigcup_{\Lambda_1\in\cL_\cZ}\cZ(\Lambda_1)(\kappa)\Bigr)\cup \Bigl(\bigcup_{\Lambda_2\in\cL_\cY}\cY(\Lambda_2^\sharp)(\kappa)\Bigr),
$$
Moreover, these strata satisfy the following inclusion relations:
\begin{altenumerate2}
\item For any $\Lambda_1$ and $\Lambda_1'$ in $\cL_\cZ$, $\cZ(\Lambda_1)(\kappa)\subseteq\cZ(\Lambda_1')(\kappa)
   $ if and only if $\Lambda_1\supseteq\Lambda_1'$,
\item For any $\Lambda_2$ and $\Lambda_2'$ in $\cL_\cY$, $ \cY(\Lambda_2^\sharp)(\kappa)\subseteq\cY(\Lambda_2'^\sharp)(\kappa)$ if and only if $\Lambda_2\subseteq\Lambda_2'$.
\end{altenumerate2}
\item For any $\Lambda_1\in\cL_{\cZ}$ and $\Lambda_2\in\cL_{\cY}$, the intersection $\cZ(\Lambda_1)(\kappa)\cap \cY(\Lambda_2^\sharp)(\kappa)$ is non-empty if and only if $\Lambda_1\subseteq \Lambda_2$. 
\item For $\Lambda\in\cL_{\cZ}\cap\cL_{\cY}$ (i.e., $\Lambda$ is a vertex lattice of type $h$), the set $\cZ(\Lambda)(\kappa)=\cY(\Lambda^{\sharp})(\kappa)$ is a singleton, corresponding to a discrete point in the RZ space called the \emph{worst point}.
\end{altenumerate}
\end{theorem}
\begin{proof}
\begin{altenumerate}
\item The decomposition follows from the crucial lemma (Proposition \ref{prop:crutial-lem}), combined with the characterization of $\kappa$-points given in Propositions \ref{prop:RZ-lattice} and \ref{prop:BT_strata}.

\item By Proposition \ref{prop:BT_strata}, a $\kappa$-point in $\cY(\Lambda_2^\sharp)$ corresponds to a lattice $M$ such that:
\begin{equation*}
    \pi\Lambda_2^\sharp\otimes W_O(\kappa)\subseteq \pi M^\sharp\subseteq M\subseteq \Lambda_2\otimes W_O(\kappa)
    \subseteq\Lambda_2^\sharp\otimes W_O(\kappa)\subseteq M^\sharp.
\end{equation*}
If $\cZ(\Lambda_1)\cap \cY(\Lambda_2^\sharp)\neq \emptyset$, then there exists a lattice $M\subset \bV\otimes W_O(\kappa)$ of type $h$ such that
\begin{equation*}
    \Lambda_1\otimes W_O(\kappa)\subseteq M\subseteq \Lambda_2\otimes W_O(\kappa).
\end{equation*}
This implies $\Lambda_1\subseteq \Lambda_2$.
Conversely, if $\Lambda_1\subseteq\Lambda_2$, then we can find a vertex lattice of type $h$ such that $\Lambda_1\subseteq\Lambda\subseteq\Lambda_2$.
Let $M:=\Lambda\otimes W_O(\kappa)\subset \Lambda^\sharp\otimes W_O(\kappa)$. By Lemma \ref{lem:cycles_stable}, $M$ is stable under $\uF,\uV$ and $\Pi$. Since $\Lambda\subset \bV$, we have $\tau(M)=M$, hence $M\in\cZ(\Lambda_1)\cap \cY(\Lambda_2^\sharp)(\kappa)$.
\item For any vertex lattice $\Lambda\subset \bV$ of type $h$, we have by definition $\cZ(\Lambda)(\kappa)=\cY(\Lambda)(\kappa)=\{\Lambda\otimes W_O(\kappa)\}$ which is a single point (which corresponds to the worst point in the local model).\qedhere
\end{altenumerate}
\end{proof}

\subsection{Kottwitz--Rapoport strata}\label{sec:KR-strata}
In this subsection, we examine the relationship between two fundamental types of strata: the (closed) BT strata and the (closed) Kottwitz--Rapoport (KR) strata. Our discussion relies on notation and concepts from the local model theory, which is developed in Section \ref{sec:LM}.

\begin{proposition}\label{prop:KR-strata}
Let $M\subset N\otimes W_{O}(\kappa)$ represent a point in $\cN_{n,\ep}^{[h]}(\kappa)$. Then one of the following inclusions must hold:
\begin{equation*}
    \tau(M)\subseteq M^{\sharp},
    \quad\text{or}\quad
    \tau(\Pi M^\sharp)\subseteq M.
\end{equation*}
\end{proposition}
We present two independent proofs of this result: the first proof uses the crucial lemma, the second proof uses local model theory.
\begin{proof}[First proof]
Assume $(X,\iota_X,\lambda_X,\rho_X)\in \cZ(\Lambda)(\kappa)$ corresponds to the lattices chain $\Lambda\otimes W_O(\kappa)\subseteq M\subseteq M^{\sharp}\subseteq \Lambda^\sharp\otimes W_O(\kappa)$. Then we have
\begin{equation*}
    \pi M^\sharp\subseteq\pi\Lambda^\sharp\otimes W_O(\kappa)\subseteq \Lambda\otimes W_O(\kappa)\subseteq M.
\end{equation*}
Then we have
\begin{equation*}
    \tau(\pi M^\sharp)\subseteq \tau(\Lambda\otimes W_O(\kappa))=\Lambda\otimes W_O(\kappa)\subseteq M.
\end{equation*}
By definition, we have $(X,\iota_X,\lambda_X,\rho_X)\in \cZ$. The same argument holds for $\cY$-strata. 
\end{proof}
\begin{proof}[Second proof]
For a strict $O_{F_0}$-module $X$ over a perfect field $\kappa$, recall its associated Dieudonn\'e module
\begin{equation*}
    \pi_0M\subset \uV M\subset M.
\end{equation*} 
There exists an identification between short exact sequences:
\begin{multline*}
\hspace{2cm}0\to \Fil(X)\to D(X)\to \Lie(X)\to 0\\
\cong\quad 0\to \uV M/\pi_0 M\to M/\pi_0 M\to M/\uV M\to 0.\hspace{2cm}
\end{multline*}
We have two transition maps $\lambda:X\to X^\vee$ and $\lambda^\vee:X^\vee\to X$.
By \cite[Thm. 6.2.2]{Luo}, for the Hodge filtrations 
\begin{equation*}
    \Fil(X)\subset D(X),\quad \Fil(X^\vee)\subset D(X^\vee),
\end{equation*}
we have either 
\begin{equation*}
    \lambda(\Fil(X))\subseteq \iota(\pi)D(X^\vee),\quad\text{or}\quad
    \lambda^\vee(\Fil(X^\vee))\subseteq \iota(\pi)D(X).
\end{equation*}
Translating these conditions into the Dieudonn\'e modules, this means:
\begin{equation}\label{equ:lattice inclusions in KR strata}
\lambda(VM/\pi_0M)\subseteq \Pi^{-1}(\Pi M^\sharp/\pi_0 M^\sharp),
\quad\text{or}\quad
\lambda^\vee(V M^\sharp/\pi_0M^\sharp)\subseteq \Pi M/\pi_0 M.
\end{equation}
Note that $M(X^\vee)=M^\sharp$, hence $D(X^\vee)\cong M^\sharp/\pi_0M^\sharp$ and $\Fil(X^\vee)\cong VM^\sharp/\pi_0M^\sharp$.

The equation \eqref{equ:lattice inclusions in KR strata} is equivalent to having either
\begin{equation*}
    VM\subseteq \Pi M^\sharp,
    \quad\text{or}\quad
    V M^{\sharp}\subseteq  M.
\end{equation*}
The assertion now follows from the definition of $\tau$.
\end{proof}

We define (closed) Kottwitz--Rapoport strata in the reduced locus of the RZ space $\ov{\mcN}_{n,\ep}^{[h]}$.
\begin{definition}
\begin{altenumerate}
\item The $\F$-scheme $\cZ$ is defined as the reduced locus of the subfunctor of $\ov{\cN}_{n,\ep}^{[h]}$ consisting of tuples $(X,\iota,\lambda,\rho)$ over an $\F$-scheme $S$ such that $\lambda^\vee(\Fil(X^\vee))\subseteq\iota(\pi) D(X)$.
\item The $\F$-scheme $\cY$ is defined as the reduced locus of the subfunctor of $\ov{\cN}_{n,\ep}^{[h]}$ consisting of tuples $(X,\iota,\lambda,\rho)$ over an $\F$-scheme $S$ such that $\lambda(\Fil(X))\subseteq\iota(\pi)D(X^\vee)$.
\end{altenumerate}
\end{definition}
The inclusions $\lambda^\vee(\Fil(X^\vee))\subseteq\iota(\pi) D(X)$ and $\lambda(\Fil(X))\subseteq\iota(\pi)D(X^\vee)$ are closed conditions, thanks to \cite[Thm. 6.2.2]{Luo}. Therefore, the corresponding subfunctors are representable as formal schemes, ensuring that the $\F$-schemes $\cZ$ and $\cY$ are well-defined.

\begin{proposition}\label{prop:KR}
\begin{altenumerate}
\item We have the decomposition of the reduced locus:
\begin{equation*}
\cN_{n,\ep,\red}^{[h]}=\cZ\cup \cY.
\end{equation*}
Furthermore, for any perfect field $\kappa$ over $\F$, we have lattice descriptions:
\begin{equation*}
\cZ(\kappa)=\left\{M\in\cN_{n,\ep,\red}^{[h]}(\kappa)\mid \tau(\Pi M^\sharp)\subseteq M\right\},
\quad\text{and}\quad\cY(\kappa)=\left\{M\in\cN_{n,\ep,\red}^{[h]}(\kappa)\mid \tau(M)\subseteq M^\sharp\right\}.
\end{equation*}
\item For any $\Lambda_1\in\cL_{\cZ}$ and $\Lambda_2\in\cL_{\cY}$, we have $\cZ(\Lambda_1)\subset \cZ$ and $\cY(\Lambda_2^\sharp)\subset \cY$.
\end{altenumerate}
\end{proposition}

\begin{proof}
Part (1) follows directly from Proposition \ref{prop:KR-strata} and the definition.
For part (2), assume $(X,\iota_X,\lambda_X,\rho_X)\in \cZ(\Lambda)(\kappa)$ corresponds to the lattice chain $\Lambda\otimes W_O(\kappa)\subseteq M\subseteq M^{\sharp}\subseteq \Lambda^\sharp\otimes W_O(\kappa)$. 
Therefore, we have
\begin{equation*}
    \pi M^\sharp\subseteq\pi\Lambda^\sharp\otimes W_O(\kappa)\subseteq \Lambda\otimes W_O(\kappa)\subseteq M.
\end{equation*}
Then we have
\begin{equation*}
    \tau(\pi M^\sharp)\subseteq \tau(\Lambda\otimes W_O(\kappa))=\Lambda\otimes W_O(\kappa)\subseteq M.
\end{equation*}
By definition, we have $(X,\iota_X,\lambda_X,\rho_X)\in \cZ$. The same argument holds for $\cY$-strata. 
\end{proof}

\begin{remark}
\begin{altenumerate}
\item Proposition \ref{prop:BT_strata}(2) can also be derived from the local model computation in \S \ref{sec:LM}.
\item For the converse of Proposition \ref{prop:KR}(2), we have the following proper inclusions:
\begin{equation*}
\bigcup_{\Lambda_1\in\cL_\cZ}\cZ(\Lambda_1)\subsetneq \cZ
\quad\text{and}\quad
\bigcup_{\Lambda_2\in\cL_\cY}\cY(\Lambda_2^\sharp)\subsetneq \cY.
\end{equation*}
This follows from Proposition \ref{prop:crutial-lem}.
\end{altenumerate}
\end{remark}
\section{Deligne--Lusztig varieties}\label{sec: DL var}
In this section, we study a class of generalized Deligne--Lusztig varieties associated with symplectic, orthogonal, and general linear groups. These algebraic varieties are then used to derive the global geometric properties of the Bruhat--Tits strata of Rapoport--Zink spaces.

\subsection{Deligne--Lusztig varieties}
Let $G_0$ be a reductive group over $\F_q$. We fix a maximal torus $T_0$ and Borel subgroup $B_0$ over $\F_q$. Let $G$ be the reductive group $G_0\otimes_{\F_q}\F$ over $\F$ and let $T:=T_0\otimes_{\F_q}\F$ and $B:=B_0\otimes_{\F_q}\F$ be the corresponding subgroups equipped with a Frobenius action $\Phi$. Let $S$ be the set of simple reflections of $W$ corresponding to $(T,B)$.

Let $W=W_G$ be the Weyl group of $G$. For any subset $I\subseteq S$, let $W_I$ be the subgroup of $W$ generated by simple reflections in $I$. Let $P_I=BW_IB$ be the associated parabolic subgroup.

Let $I, J\subset S$ be two non-empty subsets.
Every double coset in $W_I\backslash W\slash W_J$ contains a unique element of minimal length. Let $\prescript{I}{}{W}^J\subset W$ be the set of such elements. Then the map $\prescript{I}{}{W}^J\to W_I\backslash W\slash W_J$ is a bijection, which we regard as an identification.

\begin{definition}
For each $w\in \prescript{I}{}{W}^{\Phi(I)}$, the \emph{(generalized) Deligne--Lusztig variety} $X_{P_I}(w)$ is defined as
\begin{equation*}
    X_{P_I}(w):=\{g\in G/P_I:g^{-1}\Phi(g)\in P_IwP_{\Phi(I)}\}.
\end{equation*}
In the remaining part of the paper, we will write DL variety for Deligne--Lusztig variety.
\end{definition}
We recall the \emph{G\"ortz local model diagram} introduced in \cite[\S 5.2]{Gortz_Yu_2010}:
\begin{equation}\label{equ:gortz-lm}
\begin{aligned}
\xymatrix{
&G\ar[dl]_{\pi}\ar[dr]^{L}&\\
G/P_I&&G/P_{\Phi(I)}
}.
\end{aligned}
\end{equation}
Here, $\pi$ denotes the natural projection and $L$ is the Lang map, which sends an element $g$ to $g^{-1}\Phi(g)P_{\Phi(I)}$. Both maps $\pi$ and $L$ are smooth of relative dimension $\dim P_I$, and we have the equality
\begin{equation*}
    \pi^{-1}(X_{P_I}(w))=L^{-1}(P_IwP_{\Phi(I)}/P_{\Phi(I)}),
\end{equation*}
where $P_IwP_{\Phi(I)}/P_{\Phi(I)}$ is the (generalized) Schubert cell in the partial flag variety $G/P_{\Phi(I)}$.
Since both $\pi$ and $L$ in the diagram above are smooth, we can use this diagram to deduce properties of DL varieties from their corresponding Schubert cells.
\begin{proposition}\label{prop:DL-prop}
Let $I\subset S$ be a non-empty subset and let $w\in \prescript{I}{}{W}^{\Phi(I)}$. 
\begin{altenumerate}
\item The DL variety $X_{P_I}(w)$ is smooth of dimension $\ell(w)+\ell(W_{\Phi(I)})-\ell(W_{I\cap\prescript{w}{}{\Phi(I)}})$. Here $\ell(W_I)$ is the length of the longest element in the Weyl group $W_I$, and we define $\prescript{w}{}{I}:=wIw^{-1}$.
\item The DL variety $X_{P_I}(w)$ is irreducible if and only if $W_I w$ is not contained in a proper $\Phi$-stable standard parabolic subgroup of $W$.
\end{altenumerate}
\end{proposition}
\begin{proof}
Part $(1)$ follows from \cite[Lem. 2.1.3]{Hoeve} and part $(2)$ follows from \cite{BONNAFE200637}.
\end{proof}

\subsection{Symplectic case}\label{sec:DL_symp}
In this subsection, we study DL varieties for symplectic groups. In contrast to \cite{RTW}, the spaces we consider here are not of Coxeter type.
Recall the notation convention that $h=2\tth$. Let $\Lambda\subset \bV$ be a vertex lattice of type $t(\Lambda)=t>h$, where $\bV$ is the space defined in \S \ref{sec:framing}. Suppose from now on that $t\neq 0$.
We set
\begin{equation*}
    V=V_{\Lambda}:=\Lambda^{\sharp}/\Lambda
\end{equation*}
with induced symplectic form $\langle\,,\,\rangle$. 
We fix a basis $V=\mathrm{Span}_{\F_q}(\sfe_1,\cdots,\sfe_\ttt,\sff_1,\cdots,\sff_\ttt)$ such that $\langle \sfe_i,\sff_j\rangle=\delta_{ij}$, and $\langle \sfe_i,\sfe_j\rangle=\langle \sff_i,\sff_j\rangle=0$ for any $1\leq i,j\leq \ttt$. 

Let $\F$ be the algebraic closure of $\F_q$ with Frobenius $\Phi$.
We denote by $\Omega=\Omega_\Lambda=V_\Lambda\otimes_{\F_q}\F$ the symplectic space over $\F$.
Consider the standard isotropic flags $\sF_\bullet$ in $\Omega$ defined by
\begin{equation}\label{equ:sym:stand-flag}
    \sF_i=\mathrm{Span}_\F\{\sfe_1,\cdots,\sfe_i\}\quad\text{for}\quad 1\leq i\leq \ttt.
\end{equation}
This choice pins down the maximal torus and Borel subgroup $T\subset B\subset G:=\Sp(\Omega_{\Lambda})$, which are stable under the $\Phi$-action.
We use $\Delta^*=\{\sfs_1,\cdots,\sfs_\ttt\}$ to denote the set of corresponding simple reflections in the Weyl group $W=N(T)/T$, where:
\begin{itemize}
\item for $1\leq i\leq \ttt-1$, the reflection $\sfs_i$ interchanges $\sfe_i\leftrightarrow \sfe_{i+1}$ and $\sff_i\leftrightarrow \sff_{i+1}$, while fixing the other basis elements;
\item The element $\sfs_\ttt$ interchanges $\sfe_\ttt\leftrightarrow \sff_\ttt$.
\end{itemize}
For each $0\leq \tts\leq \tth\leq \ttr \leq \ttt-1$, we denote the following subset of $S$ by:
\begin{align}\label{eq: I_{s,r}}
I_{\ttr\tts}:=\{\sfs_1,\cdots,\sfs_{\ttt-\ttr-1},\sfs_{\ttt-\tts+1},\cdots,\sfs_\ttt\}=\{\sfs_1,\cdots,\sfs_\ttt\}\setminus \{\sfs_{\ttt-\ttr},\cdots,\sfs_{\ttt-\tts}\}.
\end{align}
Note that many of the objects we consider in the remainder of this section depend on $\tth$ (e.g. $I_{\ttr\tts}$). 
To simplify the notation, we will omit $\tth$ from the notations. 
We denote by $W_{\ttr\tts}$ the subgroup of the Weyl group generated by the elements in $I_{\ttr\tts}$ and denote by $P^{}_{\ttr\tts}$ the corresponding standard parabolic subgroup.

Consider now the parabolic subgroup $P_{\tth\tth}$. The space $G/P^{}_{\tth\tth}$ parameterizes the isotropic subspaces in $\Omega$ of dimension $\ttt-\tth$.
We consider the subvariety $S_{\Lambda}^{}$ of $G/P^{}_{\tth\tth}$ such that for any field $k$ over $\F$, its $k$-points are given by
\begin{equation}\label{equ:S-Lambda-defn}
    S_{\Lambda}^{}(k)=\{\cV\subset \Omega_{\Lambda,k}\mid \cV\text{ is isotropic},\,\text{and}\, \dim \cV=\ttt-\tth,\,\text{and}\, \dim(\cV\cap \Phi(\cV))\geq \ttt-\tth-1\}.
\end{equation}

\begin{example}
In the case where $\ttt-\tth=1$, the variety $S_{\Lambda}$ parameterizes all isotropic lines in the vector space $\Omega_{\Lambda,k}$, hence is isomorphic to the projective space $\mathbb{P}(\Omega_{\Lambda,k})$.
\end{example}

For any $0\leq \tts\leq \tth< \ttr\leq \ttt$, and $1\le i\le \ttt$, we let:
\begin{itemize}
\item $\sfg_{i}:=\sfs_i \sfs_{i+1} \cdots \sfs_{\ttt-1} \sfs_\ttt \sfs_{\ttt-1} \cdots \sfs_{i+1} \sfs_i.$
\item $\sfw_{\ttr\tts}:=(\sfs_{\ttt-\tth}\sfs_{\ttt-\tth+1}\cdots \sfs_{\ttt-\tts-1}\sfg_{\ttt-\tts})\cdot(\sfs_{\ttt-\tth-1}\sfs_{\ttt-\tth-2}\cdots\sfs_{\ttt-\ttr+1})$.
\item 
 $\sfw'_{\ttr\tts}:=
   (\sfs_{\ttt-\tth}\sfs_{\ttt-\tth-1}\cdots \sfs_{\ttt-\ttr+1})\cdot (\sfs_{\ttt-\tth+1} \sfs_{\ttt-\tth+2}\cdots \sfs_{\ttt-\tts-1})$ if $\tts<\tth$.
\end{itemize}
Here,  we have $\sfw_{\ttr\tts}=\sfs_{\ttt-\tth}\sfs_{\ttt-\tth+1}\cdots \sfs_{\ttt-\tts-1}\sfg_{\ttt-\tts}$ if $\ttt=\tth+1$ or $\ttr=\tth+1$.  Also, $\sfw_{\ttr\tth}=\sfg_{\ttt-\tth}\sfs_{\ttt-\tth-1}\sfs_{\ttt-\tth-2}\cdots \sfs_{\ttt-\ttr+1}$, and $\sfw'_{\ttr\tth-1}=\sfs_{\ttt-\tth}\sfs_{\ttt-\tth-1}\cdots \sfs_{\ttt-\ttr+1}$.

\begin{example}
Assume $\ttt=4$ and $\tth=2$. Then
    \begin{align*}
    w_{40}  &= s_2s_3s_4s_1, & w_{41}  &= s_2s_3s_4s_3 s_1, & w_{42} &= s_2s_3s_4s_3s_2 s_1, \\
    w_{30}  &= s_2s_3s_4,    & w_{31}  &= s_2s_3s_4s_3,     & w_{32} &= s_2s_3s_4s_3s_2, \\
    w_{40}' &= s_2s_1s_3,    & w_{41}' &= s_2s_1,           &        & \\
    w_{30}' &= s_2s_3,       & w_{31}' &= s_2.              &        &
\end{align*}
\end{example}

Note that the element $\sfg_i$ interchanges $\sfe_i\leftrightarrow \sff_i$ and fixes all other elements in the basis. The actions of $\sfw^{}_{\ttr\tts}$ and $\sfw'_{\ttr\tts}$ on the basis are described in the following lemma.
\begin{lemma}\label{lem:rel-pos}
\begin{altenumerate}
\item The element $\sfw^{}_{\ttr\tth}$ stabilizes $\sfe_i$ and $\sff_i$ for $1\leq i\leq \ttt-\ttr$ and $\ttt-\tth+1\leq i\leq \ttt$.
It acts on the remaining elements in the basis as follows:
\begin{equation*}
\begin{aligned}
\xymatrix{
\sfe_{\ttt-\ttr+1}\ar@/^1pc/[rrrrr]
&\cdots \ar@/^/[l]&
\sfe_{\ttt-\tth}\ar@/^/[l]&
\sff_{\ttt-\ttr+1}\ar@/^/[l]
&\cdots\ar@/^/[l]&
\sff_{\ttt-\tth}.\ar@/^/[l]
}
\end{aligned}
\end{equation*}
Moreover, $\sfw^{}_{\ttr\tth}$ is the minimal representative element in the double coset $W_{\ttr}\sfw^{}_{\ttr\tth} W_{\ttr}$.
\item 
If $\tts<\tth$, then the element $\sfw_{\ttr\tts}$ stabilizes $\sfe_i$ and $\sff_i$ for $1\leq i\leq \ttt-\ttr$ and $\ttt-\tts+1\leq i\leq \ttt$.
It acts on the remaining $\sfe_i$'s as follows:
\begin{equation*}
\begin{aligned}
\xymatrix{
\sfe_{\ttt-\ttr+1}\ar@/^1pc/[rrr]&
\cdots\ar@/^/[l]&
\sfe_{\ttt-\tth}\ar@/^/[l]&
\sfe_{\ttt-\tth+1}\ar@/^/[r]&
\cdots\ar@/^/[r]&
\sfe_{\ttt-\tts}\ar[dlll]\\
\sff_{\ttt-\ttr+1}\ar@/_1pc/[rrr] &
\cdots \ar@/_/[l]&
\sff_{\ttt-\tth}\ar@/_/[l] &
\sff_{\ttt-\tth+1}\ar@/_/[r]&
\cdots\ar@/_/[r]&
\sff_{\ttt-\tts}\ar[ulll]|!{[u];[lll]}\hole
}
\end{aligned}
\end{equation*}
Moreover, $\sfw_{\ttr\tts}$ is the minimal representative element in the double coset $W_{\ttr\tts}\sfw_{\ttr\tts}W_{\ttr\tts}$.
\item The element $\sfw'_{\ttr\tts}$ stabilizes $\sfe_i$ and $\sff_i$ for $1\leq i\leq \ttt-\ttr$ and $\ttt-\tts+1\leq i\leq \ttt$.
It acts on the remaining $\sfe_i$'s as follows:
\begin{equation*}
\begin{aligned}
\xymatrix{
\sfe_{\ttt-\ttr+1}\ar@/^1pc/[rrr]&
\cdots\ar@/^/[l]&
\sfe_{\ttt-\tth}\ar@/^/[l]&
\sfe_{\ttt-\tth+1}\ar@/^/[r]&
\sfe_{\ttt-\tth+2}\ar@/^/[r]&
\cdots\ar@/^/[r]&
\sfe_{\ttt-\tts}\ar@/^1pc/[llll]
}
\end{aligned}
\end{equation*}
We have the same action on the remaining $\sff_i$'s.
Moreover, $\sfw'_{\ttr\tts}$ is the minimal representative element in the double coset $W_{\ttr\tts}\sfw'_{\ttr\tts} W_{\ttr\tts}$.
\end{altenumerate}
\end{lemma}
\begin{proof}
The actions of $w_{\ttr\tts}$ and $w_{\ttr\tts}'$ can be checked via direct and explicit computation so we leave them to the reader.
Now we sketch the proof that
$w_{\ttr\tts}$ has minimal length among the elements of $W_{\ttr\tts}w_{\ttr\tts}W_{\ttr\tts}$ (see \cite[Lemma 5.15]{HZBasic} for more details). By \cite[Proposition 5.13]{HZBasic}, the length $\ell(w)$ for $w\in W$ is equal to $\mathrm{inv}(w)$, the number of inversions of $w$. To explain this, we use the bijection 
$$
\begin{aligned}
\left\{e_1, \ldots, e_d\right\} & \rightarrow\{1, \ldots, d\}, \\
e_i & \mapsto d+1-i,
\end{aligned}
$$
and 
$$
\begin{aligned}
\left\{f_{1}, \ldots, f_{d}\right\} & \rightarrow\{-d, \ldots,-1\}, \\
f_{i} & \mapsto-d-1+i
\end{aligned}
$$
to identify $\{f_1,\ldots,f_d,e_d,\ldots,e_1\}$ with $\{-d,\ldots,-1,1,\ldots,d\}$ so that $\{-d,\ldots,-1,1,\ldots,d\}$ has a corresponding action of $W$. Now by \cite[\S 5.3.2]{HZBasic}, we have 
\begin{align*}
    \mathrm{inv}(w)=\mathrm{inv}^+(w)+\mathrm{inv}^-(w),
\end{align*}
where we define
\begin{align*}
    \mathrm{inv}^+(w)= \sharp 
\left\{(i, j) \in\{1, \ldots, \ttt\}^2: i<j, w(i)>w(j)\right\},
\end{align*}
and
\begin{align*}
    \mathrm{inv}^-(w)=\sharp 
\left\{(i, j) \in\{1, \ldots, \ttt\}^2: i\le j, w(-i)>w(j)\right\}.
\end{align*}

On the one hand, $\sfw_{\ttr\tts}$ stabilizes $\sfe_i$ and $\sff_i$ for $1\leq i\leq \ttt-\ttr$ and $\ttt-\tts+1\leq i\leq \ttt$. On the other hand, $W_{\ttr\tts}$ is generated by $s_{i}$ in 
\begin{align*}
    I_{\ttr\tts}:=\{\sfs_1,\cdots,\sfs_{\ttt-\ttr-1},\sfs_{\ttt-\tts+1},\cdots,\sfs_\ttt\}.
\end{align*}
So $W_{\ttr\tts}$ only acts on $\sfe_i$ and $\sff_i$ for $1\leq i\leq \ttt-\ttr$ and $\ttt-\tts+1\leq i\leq \ttt$. Therefore, for any $w,w'\in W_{\ttr\tts}$, we have the inequality
\begin{align*}
    \mathrm{inv}(w w_{\ttr\tts}w')\ge  \mathrm{inv}( w_{\ttr\tts})
\end{align*}
which completes the proof.

The statements about the minimal length in (1) and (3) can be proved similarly, so we omit the details.
\end{proof}

The main result of this subsection is the following:
\begin{theorem}\label{thm:symplectic main}
We have the following stratification: 
\begin{equation}\label{eq:symplectic stratification}
     S_\Lambda=
    \Bigl(\coprod_{0\leq j\leq \tth< i\leq \ttt}X_{P_{ij}}(\sfw_{ij})\Bigr)\amalg
    \Bigl(\coprod_{0\leq j<\tth<i\leq \ttt}X_{P_{ij}}(\sfw'_{ij})\Bigr)
    \amalg X_{P_{\tth\tth}}(\id),
\end{equation}
such that for each  $0\leq \tts \leq \tth\leq \ttr\leq \ttt$, we have the following closure relations:
\begin{align}
\ov{X_{P_{\ttr\tts}}(\sfw_{\ttr\tts})}={}&\Bigl(\coprod_{0\leq j\leq \tts\leq \tth< i\leq \ttr}X_{P_{ij}}(\sfw_{ij})\Bigr)\amalg
    \Bigl(\coprod_{\tts\leq j <\tth<i\leq \ttr}X_{P_{ij}}(\sfw'_{ij})\Bigr)
    \amalg X_{P_{\tth\tth}}(\id);\label{eq:symplectic closure relation 1}\\
\ov{X_{P_{\ttr\tts}}(\sfw'_{\ttr\tts})}={}&\Bigl(\coprod_{\tts\leq j<\tth<i\leq \ttr}X_{P_{ij}}(\sfw'_{ij})\Bigr)
    \amalg X_{P_{\tth\tth}}(\id).\label{eq:symplectic closure relation 2}
\end{align}
Moreover, $S_\Lambda$ is irreducible and normal of dimension $\ttt+\tth$. 
\end{theorem}

\begin{proof}
For $\tts\leq \tth< \ttr$, define a locally closed subvariety $S_{\ttr\tts}$ of $G/P_{\ttr\tts}$ by specifying its $k$-points for any field $k$ over $\F$
\begin{equation*}
S_{\ttr\tts}(k):=\left\{\begin{array}{r}
(\cF_{\ttt-\ttr}\subset \ldots \subset \cF_{\ttt-\tts})\\
\in  G/P_{\ttr\tts}(k)
\end{array}\left\vert
\begin{array}{l}\cF_{i}=\cF_{i+1}\cap \Phi(\cF_{i+1})\text{ for }\ttt-\ttr\leq i\leq \ttt-\tth-1;\\
\cF_{j+1}=\cF_{j}+\Phi(\cF_{j})\text{ for }\ttt-\tth \leq j \leq \ttt-\tts-1;\\
\cF_{\ttt-\ttr}=\Phi(\cF_{\ttt-\ttr})\text{ and }\cF_{\ttt-\tts}+\Phi(\cF_{\ttt-\tts})\text{ is non-isotropic}.
\end{array}\right\}\right..
\end{equation*}
For $\tts< \tth< \ttr$, define a locally closed subvariety $S'_{\ttr\tts}$ of $G/P_{\ttr\tts}$ by specifying its $k$-points for any field $k$ over $\F$
\begin{equation*}
S'_{\ttr\tts}(k):=\left\{\begin{array}{r}
(\cF_{\ttt-\ttr}\subset \ldots \subset \cF_{\ttt-\tts})\\
\in  G/P_{\ttr\tts}(k)
\end{array}\left\vert
\begin{array}{l}\cF_{i}=\cF_{i+1}\cap \Phi(\cF_{i+1})\text{ for }\ttt-\ttr\leq i\leq \ttt-\tth-1;\\
\cF_{j+1}=\cF_{j}+\Phi(\cF_{j})\text{ for }\ttt-\tth \leq j \leq \ttt-\tts-1;\\
\cF_{\ttt-\ttr}=\Phi(\cF_{\ttt-\ttr})\text{ and }\cF_{\ttt-\tts}=\Phi(\cF_{\ttt-\tts}).
\end{array}\right\}\right..
\end{equation*}
Finally it is clear that $X_{P_{\tth\tth}}(\id)$ is the subvariety of $G/P_{\tth\tth}$ whose $k$-points are the finite set 
\[X_{P_{\tth\tth}}(\id)(k)=\{\cF_{\ttt-\tth}\in G/P_{\tth\tth}(k)\mid \cF_{\ttt-\tth}= \Phi(\cF_{\ttt-\tth})
\}. \]
We claim that:
\begin{equation}\label{eq:fake-symplectic stratification}
     S_\Lambda=
    \Bigl(\coprod_{0\leq j\leq \tth< i\leq \ttt}S_{\ttr\tts}\Bigr)\amalg
    \Bigl(\coprod_{0\leq j<\tth<i\leq \ttt}S'_{\ttr\tts}\Bigr)
    \amalg X_{P_{\tth}}(\id).
\end{equation}
Assuming \eqref{eq:fake-symplectic stratification}, the decomposition \eqref{eq:symplectic stratification} now follows from Proposition \ref{prop:moduli interpretation w_r,s}, and Proposition \ref{prop:moduli interpretation w'_r,s} below. 
The closure relations \eqref{eq:symplectic closure relation 1} and \eqref{eq:symplectic closure relation 2} now follow from the moduli interpretations of $S_{\ttr\tts}$ and $S'_{\ttr\tts}$. In particular, we see that $X_{P_{\ttt\tth}}(\sfw_{\ttt\tth})$ is an open dense subvariety of $S_\Lambda$, hence the statements about the irreducibility and dimension of $S_\Lambda$ follow from the corresponding statements for $X_{P_{\ttt\tth}}(\sfw_{\ttt\tth})$. Then Proposition \ref{prop:DL-prop} implies the irreducibility of $X_{P_{\ttt\tth}}(\sfw_{\ttt\tth})$ since $\sfw_{\ttt\tth}$ is a long word that contains every $s_i$ so that condition (2) of Proposition \ref{prop:DL-prop} is satisfied. 
The dimension of $X_{P_{\ttt\tth}}(\sfw_{\ttt\tth})$ follows from Proposition \ref{prop:moduli interpretation w_r,s}. The normality of $S_\Lambda$ follows from the normality of Schubert varieties and the G\"ortz local model diagram \eqref{equ:gortz-lm}.

In the remaining part of the proof, we prove \eqref{eq:fake-symplectic stratification}.
For any point $z=(\cF_{\ttt-\tth})\in S_\Lambda(k)$, consider the flag 
\[(\cF_i)_{i\in \Z}:=
\bigl(\ldots \subseteq \cF_{\ttt-\tth-1}\subseteq\cF_{\ttt-\tth}\subseteq \cF_{\ttt-\tth+1}\subseteq \ldots\bigr), \]
defined such that for $i<\ttt-\tth$, $\cF_i=\cF_{i+1}\cap \Phi(\cF_{i+1})$ and for $i>\ttt-\tth$, $\cF_i=\cF_{i-1}+ \Phi(\cF_{i-1})$.
Since $\Omega_\Lambda$ is a finite dimensional vector space, the flag $(\cF_i)_{i\in \Z}$ stabilizes at both ends. 

Let $\ttr\geq \tth$ be the unique integer such that $\cF_{\ttt-\ttr-1}=\cF_{\ttt-\ttr}\subsetneq \cF_{\ttt-\ttr+1}$. Similarly, let $\tts\leq \tth$ be the unique integer such that one of the following two situations occurs:
\begin{enumerate}
    \item[(a)] $\cF_{\ttt-\tts-1}\subsetneq \cF_{\ttt-\tts}$, $\cF_{\ttt-\tts}$ is isotropic, and $\cF_{\ttt-\tts+1}$ is anisotropic.
    \item[(b)] $\cF_{\ttt-\tts-1}\subsetneq \cF_{\ttt-\tts}=\cF_{\ttt-\tts+1}$ and $\cF_{\ttt-\tts}$ is isotropic.
\end{enumerate}
Depending on whether the flag $(\cF_i)_{i\in \Z}$ stabilizes first or becomes anisotropic first at the right end, exactly one of the above situations will occur. 

Consider the isotropic flag $(\cF_{\ttt-\ttr}\subset \ldots \subset \cF_{\ttt-\tts})$, we claim that
$$
\dim_k \cF_i=i\quad\text{for}\quad\ttt-\ttr\leq i \leq \ttt-\tts.
$$
Indeed, for any $\ttt-\ttr<i<\ttt-\tth$, we have 
\[ [\cF_i:\cF_{i-1}]=[\cF_i:\cF_{i}\cap\Phi(\cF_i)]=[\Phi(\cF_i)+\cF_i:\Phi(\cF_i)]=[\cF_i+\Phi^{-1}(\cF_i):\cF_i],  \]
where for any $\cF\subset\cG$, we denote by $[\cF:\cG]:=\dim_k \cF/\cG$.
Since $\cF_i=\cF_{i+1}\cap\Phi(\cF_{i+1})$, we have $\cF_i+\Phi^{-1}(\cF_i)\subseteq\cF_{i+1}$. Hence we obtain 
\[[\cF_i:\cF_{i-1}] \leq [\cF_{i+1}:\cF_{i}]\leq \ldots \leq [\cF_{\ttt-\tth}:\cF_{\ttt-\tth-1}]\leq 1, \]
where the last inequality follows from the definition of $S_\Lambda$. 

Since by assumption $\cF_i\neq \cF_{i-1}$ for any $\ttt-\ttr< i\leq \ttt-\tth$, it follows that $[\cF_i:\cF_{i-1}]=1$, and hence $\dim_k \cF_{i-1}=i-1$. 
Similarly $\dim_k \cF_i=i$ for $\ttt-\tth< i \leq \ttt-\tts$. In other words the isotropic flag $(\cF_{\ttt-\ttr}\subset \ldots \subset \cF_{\ttt-\tts})$ defines a point in $G/P_{\ttr\tts}(k)$. 

Now if situation (a) occurs for $\cF_{\ttt-\tts}$, then $(\cF_{\ttt-\ttr}\subset \ldots \subset \cF_{\ttt-\tts})$ lies in $S_{\ttr\tts}(k)$. If situation (b) occurs, then $(\cF_{\ttt-\ttr}\subset \ldots \subset \cF_{\ttt-\tts})$ lies in $S'_{\ttr\tts}(k)$, resp. in $X_{P_{\tth}}(\id)(k)$, if in addition $\ttr=\tts=\tth$. Conversely, any point $z=(\cF_{\ttt-\ttr}\subset \ldots \subset \cF_{\ttt-\tts})$ in $S_{\ttr\tts}(k)$ or $S'_{\ttr\tts}(k)$ or $X_{P_{\tth}}(\id)(k)$ gives rise to a point $\cF_{\ttt-\tth}\in S_\Lambda$, from which the point $z$ can be recovered from $\cF_{\ttt-\tth}$ by the above procedure.
This proves \eqref{eq:fake-symplectic stratification}.
\end{proof}

\begin{remark}\label{rmk:first decomp}
Let
$$\sfw_{\Lambda}:=\sfs_{\ttt-\tth} \sfs_{\ttt-\tth+1} \cdots \sfs_{\ttt-1} \sfs_\ttt \sfs_{\ttt-1} \cdots \sfs_{\ttt-\tth+1} \sfs_{\ttt-\tth}
\quad\text{and}\quad
\sfw_{\Lambda}':=\sfs_{\ttt-\tth}.$$
One can verify that (cf. the proof of Proposition \ref{prop:moduli interpretation w_r,s})
\begin{equation}\label{equ:symplectic-first-step}
    S_\Lambda=X_{P_{\tth\tth}}(\sfw_{\Lambda})\amalg X_{P_{\tth\tth}}(\sfw'_{\Lambda}) \amalg X_{P_{\tth\tth}}(\id),
\end{equation}
where we can describe these DL varieties as follows:
\begin{align*}
X_{P_{\tth\tth}}(\sfw_{\Lambda})={}&\{
\cV\in S_\Lambda\mid \cV\neq \Phi(\cV)\text{ and }\cV+\Phi(\cV)\text{ is not isotropic}\};\\
X_{P_{\tth\tth}}(\sfw'_{\Lambda})={}&\{
\cV\in S_\Lambda\mid \cV\neq \Phi(\cV)\text{ and }\cV+\Phi(\cV)\text{ is isotropic}
\};\\
X_{P_{\tth\tth}}(\id)={}&\{
\cV\in S_\Lambda\mid \cV= \Phi(\cV)
\}.
\end{align*}
We have the following decomposition, which refines the stratification \eqref{eq:symplectic stratification}:
\begin{itemize}
\item $\displaystyle X_{P_{\tth\tth}}(\sfw_{\Lambda})=\coprod_{i=h+1}^{\ttt}X_{P_{i\tth}}(\sfw_{i\tth})$; 
\item $\displaystyle X_{P_{\tth\tth}}(\sfw'_{\Lambda})= \Bigl(\coprod_{0\leq j<\tth< i\leq \ttt}X_{P_{ij}}(\sfw_{ij})\Bigr)\amalg
    \Bigl(\coprod_{0\leq j<\tth<i\leq \ttt}X_{P_{ij}}(\sfw'_{ij})\Bigr)$.
\end{itemize}
This result can be established through two approaches: either via a group-theoretical argument as shown in \cite[Thm. 5.4]{HZBasic}, or through the moduli descriptions given in Propositions \ref{prop:moduli interpretation w_r,s} and \ref{prop:moduli interpretation w'_r,s}.

This decomposition corresponds to the stratification of \emph{fine DL varieties}, as discussed in \cite[\S 2.3]{He-Li-Zhu}. In our context, these strata themselves are Deligne--Lusztig varieties, as demonstrated in \cite[Thm. 2.1.7(1)]{Hoeve}. This provides an alternative proof of Theorem \ref{thm:symplectic main}, though we will not elaborate on this approach here.
When $\tth=0$, the component $X_{P_{\tth\tth}}(\sfw'_{\Lambda})$ is absent, which is evident both from the indices in \eqref{eq:symplectic closure relation 1} and from the moduli description in Proposition \ref{prop:moduli interpretation w'_r,s}.
\end{remark}

To complete the proof of Theorem \ref{thm:symplectic main}, it remains to state and prove Proposition \ref{prop:moduli interpretation w_r,s} and Proposition \ref{prop:moduli interpretation w'_r,s} below.

\begin{proposition}\label{prop:moduli interpretation w_r,s}
For any integers $0\leq \tts\leq \tth<\ttr\leq \ttt$, the DL variety $X_{P_{\ttr\tts}}(\sfw_{\ttr\tts})$ is the subvariety of $G/P_{\ttr\tts}$ whose $k$-points for any field $k$ over $\F$ are characterized by:
\begin{equation*}
X_{P_{\ttr\tts}}(\sfw_{\ttr\tts})(k):=\left\{\begin{array}{r}
(\cF_{\ttt-\ttr}\subset \ldots \subset \cF_{\ttt-\tts})\\
\in  G/P_{\ttr\tts}(k)
\end{array}\left\vert
\begin{array}{l}\cF_{i}=\cF_{i+1}\cap \Phi(\cF_{i+1})\text{ for }\ttt-\ttr\leq i\leq \ttt-\tth-1;\\
\cF_{j+1}=\cF_{j}+\Phi(\cF_{j})\text{ for }\ttt-\tth \leq j \leq \ttt-\tts-1;\\
\cF_{\ttt-\ttr}=\Phi(\cF_{\ttt-\ttr})\text{ and }\cF_{\ttt-\tts}+\Phi(\cF_{\ttt-\tts})\text{ is non-isotropic}.
\end{array}\right\}\right..
\end{equation*}
Its dimension is $\ttr+\tts$.
\end{proposition}
\begin{proof}
We start with the following general observation.
Let $gP\in X_P(w)$ be any point in the DL variety, we have $g^{-1}\Phi(g)\in Pw\Phi(P)$, in particular, we can find $p_1,p_2\in P$ such that $\Phi(g)=gp_1w\Phi(p_2)$.
Let $g_0=gp_1$, then we have 
\begin{equation*}
    \Phi(g_0)=g_0w\Phi(p_2)\Phi(p_1)=g_0 w\Phi(p_2p_1).
\end{equation*}
Since $g_0P=gP$, we conclude that:
\begin{equation}\label{equ:moduli}
gP\in X_P(w)\quad\text{if and only if}\quad
\exists g_0\in gP\text{ such that  }\Phi(g_0)\Phi(P)=g_0w\Phi(P).
\end{equation}

In particular, by identifying the right coset $gP$ with partial flags $(\cF_{\ttt-\ttr}\subset \ldots \subset \cF_{\ttt-\tts})$, we can use Lemma \ref{lem:rel-pos}(1)(2) and equation \eqref{equ:moduli} to show that any partial flag $(\sfF_i)=g_0P_{\ttr\tts}\in X_{P_{\ttr\tts}}(\sfw_{\ttr\tts})(k)$ belongs to the moduli functor defined in the assertion.
Conversely, consider a partial flag $(\sfF_i)=g_0P_{\ttr\tts}$ in $\Sp(\Omega_\Lambda)/P_{\ttr\tts}(k)$ that defines a point in the moduli functor. We will prove that this flag lies in $X_{P_{\ttr\tts}}(\sfw_{\ttr\tts})(k)$.

We find a common basis for all $\cF_i$ and $\Phi(\cF_j)$ in the following three steps:
\begin{altenumerate}
\item[\emph{Step 1.}\,] We choose vectors $u_i\in \sfF_{\ttt-\tth}$ for $1\leq i\leq \ttt-\tth$ such that $\sfF_{\ttt-\ttr}=\mathrm{Span}_{k}(u_1,\cdots,u_{\ttt-\ttr})$ and that
\begin{equation*}
\sfF_{i}=\sfF_{i-1}\oplus \langle u_i\rangle\quad\text{for all}\quad\ttt-\ttr\leq i\leq \ttt-\tth.
\end{equation*}  

\item[\emph{Step 2.}\,] For each $\ttt-\ttr+1\leq i\leq \ttt-\tth-1$, we have $\Phi(\sfF_i)\stackrel{1}{\subset}\Phi(\sfF_{i+1}).$
For each $\ttt-\ttr\leq i\leq \ttt-\tth-1$, the equality $\sfF_i=\sfF_{i+1}\cap \Phi(\sfF_{i+1})$ implies that 
\begin{equation}\label{equ:vector-f-correct}
    \sfF_i=\sfF_{i+1}\cap\Phi(\sfF_{i+1})\stackrel{1}{\subset}\Phi(\sfF_{i+1}).
\end{equation}
These inclusions yield:
\begin{equation}\label{equ:vector-f}
    \Phi(\sfF_{i+1})=\sfF_i+\Phi(\sfF_i),\quad\text{for all}\quad\ttt-\ttr+1\leq i\leq \ttt-\tth-1.
\end{equation}
Therefore, we can find a vector $u_{\ttt-\tth+1}\in\Phi(\sfF_{\ttt-\ttr})$ such that:
\begin{equation*}
    \Phi(\sfF_{\ttt-\ttr})=\sfF_{\ttt-\ttr-1}\oplus\langle u_{\ttt-\tth+1}\rangle.
\end{equation*}
We have $u_{\ttt-\tth+1}\in \Phi(\sfF_{i})$ for all $\ttt-\ttr\leq i\leq \ttt-\tth$. By \eqref{equ:vector-f-correct} and induction, we have $u_{\ttt-\tth+1}\notin \sfF_i$ for any $\ttt-\ttr\leq i\leq \ttt-\tth$, and that: 
\begin{equation*}
    \Phi(\sfF_{i+1})=\sfF_{i}\oplus\langle u_{\ttt-\tth+1}\rangle,\quad\text{for all}\quad\ttt-\ttr\leq i\leq \ttt-\tth-1.
\end{equation*}

\item[\emph{Step 3.}\,] Finally, for each $\ttt-\tth+2\leq j\leq \ttt-\tts-1$ (if $\tts=\tth$, we skip this step), we denote by 
\begin{equation*}
    u_j:=\Phi(u_{j-1})=\Phi^{j-(\ttt-\tth+1)}(u_{\ttt-\tth+1}).
\end{equation*}
Since $\cF_{j}=\cF_{j-1}+\Phi(\cF_{j-1})$, by induction, we have
\begin{equation*}
\cF_j=\mathrm{Span}_{k}(u_1,\cdots,u_j)\quad\text{and}\quad
    \Phi(\cF_j)=\mathrm{Span}_{k}(u_1,\cdots,u_{\ttt-\tth-1},u_{\ttt-\tth+1},\cdots,u_{j+1}).
\end{equation*}
\end{altenumerate}
To summarize, we have formed a set of vectors $u_i$ such that
$$\cF_i=\mathrm{Span}_{k}(u_1,\cdots,u_i), \quad\text{for all}\quad \ttt-\ttr\leq i\leq \ttt-\tts,$$ 
and that for any $\ttt-\ttr\leq j\leq \ttt-\tts$, we have
\begin{equation*}
\Phi(\cF_j)=\left\{\begin{array}{ll}
\mathrm{Span}_{k}(u_1,\cdots,u_{j},u_{\ttt-\tth+1}),& \ttt-\ttr\leq j\leq \ttt-\tth;\\
\mathrm{Span}_{k}(u_1,\cdots,u_{\ttt-\tth-1},u_{\ttt-\tth+1},\cdots,u_{j+1}),&\ttt-\tth+1\leq j\leq \ttt-\tts.\\
\end{array}\right.
\end{equation*}

Notice that:
\begin{altitemize}
\item Since the subspace $\cF_{\ttt-\tts}$ is isotropic, we have
\begin{equation*}
    \langle u_i,u_j\rangle=0\quad\text{for all}\quad 1\leq i,j\leq \ttt-\tts.
\end{equation*}
\item Since the subspace $\Phi(\cF_{\ttt-\tts})$ is isotropic, we have
\begin{equation*}
    \langle u_i,u_{\ttt-\tts+1}\rangle=0\quad\text{for all}\quad 1\leq i\neq \ttt-\tth\leq \ttt-\tts.
\end{equation*}
\item Since $\cF_{\ttt-\tts}+\Phi(\cF_{\ttt-\tts})$ is non-isotropic, we have
\begin{equation}\label{equ:F+phiF-anisotropic}
    \langle u_{\ttt-\tth},u_{\ttt-\tts+1}\rangle\neq 0.
\end{equation}
By normalizing $u_{\ttt-\tts+1}$, we may assume that the pairing equals $1$.
\end{altitemize}

Define a linear map between two subspaces of $\Omega_\Lambda$ by sending
\begin{equation*}
    \sfe_i\mapsto u_i, \quad\text{for }1\leq i\leq \ttt-\tts;
    \quad\text{and}\quad \sff_{\ttt-\tth}\mapsto u_{\ttt-\tts+1}.
\end{equation*}
This is an isometry, and by Witt's theorem, we can find a $g\in G$ that extends this isometry to an isometry of $\Omega_\Lambda$.
For such $g$, we have:\begin{equation*}
\begin{aligned}
\xymatrix{
\Bigl[
\sF_{\ttt-\ttr}\subset\cdots\subset\sF_{\ttt-\tts}
\Bigr]\ar@{~>}[r]^-{g}&\Bigl[
\sfF_{\ttt-\ttr}\subset\cdots\subset \sfF_{\ttt-\tts}
\Bigr],
}
\end{aligned}
\end{equation*}
and
\begin{equation*}
\begin{aligned}
\xymatrix{
\Bigl[
\sfw_{\ttr\tth}(\sF_{\ttt-\ttr})\subset\cdots\subset\sfw_{\ttr\tth}(\sF_{\ttt-\tts})
\Bigr]\ar@{~>}[r]^-{g}&\Bigl[
\Phi(\sfF_{\ttt-\ttr})\subset\cdots\subset \Phi(\sfF_{\ttt-\tts})
\Bigr].
}
\end{aligned}
\end{equation*}
Recall that $\sF_i$ are standard flags defined in \eqref{equ:sym:stand-flag}.
Equivalently, we have $gP_{\ttr\tts}=g_0P_{\ttr\tts}$ and that $\Phi(g)P_{\ttr\tts}=g\sfw_{\ttr\tts}P_{\ttr\tts}$.
By \eqref{equ:moduli}, we have $g_0P_{\ttr\tts}=gP_{\ttr\tts}\in X_{P_{\ttr\tts}}(\sfw_{\ttr\tts})$.

Finally, by Lemma \ref{lem:rel-pos}(1)(2), we see that $\sfw_{\ttr\tts}$ acts trivially on $\sfe_i$ and $\sff_i$ for $1\leq i\leq \ttt-\ttr$ and $\ttt-\tts+1\leq i\leq \ttt$, hence $\sfw_{\ttr\tts} I_{\ttr\tts}=I_{\ttr\tts}\sfw_{\ttr\tts},$. By Proposition \ref{prop:DL-prop}(1), the dimension of the DL variety $X_{P_{\ttr\tts}}(\sfw_{\ttr\tts})$ equals
$\ell(\sfw_{\ttr\tts})=\ttr+\tts$, as desired.
\end{proof}

We have similar statements for the other class of DL varieties:
\begin{proposition}\label{prop:moduli interpretation w'_r,s}
For any integers $0\leq \tts < \tth<\ttr\leq \ttt$, the DL variety $X_{P_{\ttr\tts}}(\sfw'_{\ttr\tts})$ is the subvariety of $G/P_{\ttr\tts}$ whose $k$-points for any field $k$ over $\F$ are characterized by:
\begin{equation*}
X_{P_{\ttr\tts}}(\sfw'_{\ttr\tts})(k):=\left\{\begin{array}{r}
(\cF_{\ttt-\ttr}\subset \ldots \subset \cF_{\ttt-\tts})\\
\in  G/P_{\ttr\tts}(k)
\end{array}\left\vert
\begin{array}{l}\cF_{i}=\cF_{i+1}\cap \Phi(\cF_{i+1})\text{ for }\ttt-\ttr\leq i\leq \ttt-\tth-1;\\
\cF_{j+1}=\cF_{j}+\Phi(\cF_{j})\text{ for }\ttt-\tth \leq j \leq \ttt-\tts-1;\\
\cF_{\ttt-\ttr}=\Phi(\cF_{\ttt-\ttr})\text{ and }\cF_{\ttt-\tts}=\Phi(\cF_{\ttt-\tts}).
\end{array}\right\}\right..
\end{equation*}
Its dimension is $\ttr-\tts-1$.
\end{proposition}
\begin{proof}
It is straightforward to verify that $\sfw'_{\ttr\tts} I_{\ttr\tts}=I_{\ttr\tts}\sfw'_{\ttr\tts}$, which concludes the proof of the dimension formula.
By Lemma \ref{lem:rel-pos}(iii) and \eqref{equ:moduli}, any partial flag $(\sfF_i)=g_0P_{\ttr\tts}\in X_{P_{\ttr\tts}}(\sfw'_{\ttr\tts})(k)$ lies in the moduli functor.
Conversely, suppose $(\sfF_i)=g_0P_{\ttr\tts}$ is a partial flag in $G/P_{\ttr\tts}(k)$ that defines a point in the moduli functor, we show that this flag lies in $X_{P_{\ttr\tts}}(\sfw'_{\ttr\tts})(k)$.

Following with the same procedure as Proposition \ref{prop:moduli interpretation w_r,s}, we can find vectors $u_i\in \Omega_{\Lambda}(k)$ such that
$$\cF_i=\mathrm{Span}_{k}(u_1,\cdots,u_i),\quad\text{for all}\quad\ttt-\ttr\leq i\leq \ttt-\tts,$$ 
and that for any $\ttt-\ttr\leq j\leq \ttt-\tts-1$, we have
\begin{equation*}
\Phi(\cF_j)=\left\{\begin{array}{ll}
\mathrm{Span}_{k}(u_1,\cdots,u_{j},u_{\ttt-\tth+1}),& \ttt-\ttr\leq j\leq \ttt-\tth;\\
\mathrm{Span}_{k}(u_1,\cdots,u_{\ttt-\tth-1},u_{\ttt-\tth+1},\cdots,u_{j+1}),&\ttt-\tth+1\leq j\leq \ttt-\tts-1.\\
\end{array}\right.
\end{equation*}
Since $\cF_{\ttt-\tts}=\cF_{\ttt-\tts-1}+\Phi(\cF_{\ttt-\tts-1})$ is $\Phi$-stable now, we have 
\begin{equation*}
    \Phi(\cF_{\ttt-\tts})=\cF_{\ttt-\tts}=\mathrm{Span}_{k}(u_1,\cdots,u_{\ttt-\tts}).
\end{equation*}
Since $\cF_{\ttt-\tts}$ is isotropic, we conclude that 
$$
\langle u_i,u_j\rangle=0,\quad\text{for all}\quad 1\leq i,j\leq \ttt-\tts.
$$
Define a linear map between two subspaces of $\Omega_\Lambda$ by sending $\sfe_i\mapsto u_i, \text{ for }1\leq i\leq \ttt-\tts.$
This is an isometry and, by Witt's theorem, we can find a $g\in G$ that extends this isometry to an isometry of $\Omega_\Lambda$. By the same argument as in the proof of Proposition \ref{prop:moduli interpretation w_r,s}, we have $g_0P_{\ttr\tts}=gP_{\ttr\tts}\in X_{P_{\ttr\tts}}(\sfw_{\ttr\tts}')$.
\end{proof}

\begin{remark}
The DL variety
$X_{P_{\ttr\tts}}(\sfw'_{\ttr\tts})\subset \Sp(\Omega_{\Lambda})/P_{\ttr\tts}$ in Proposition \ref{prop:moduli interpretation w'_r,s} is not irreducible in general. Its irreducible components are all isomorphic to the DL variety associated with the general linear group, as described in Proposition \ref{prop:GL-DLV}. These components are indexed by pairs of $\F_q$-rational subspaces $\cF^0_{\ttt-\ttr}\subset\cF^0_{\ttt-\tts}\subset V_{\Lambda}$ of dimensions $\ttt-\ttr$ and $\ttt-\tts$, resp. For further details on this matter, we refer the reader to \cite[Prop. 5.7]{RTW}.
\end{remark}

\subsection{The orthogonal case}
First, we introduce the relevant group-theoretic notation in the orthogonal case. Let $\Lambda\subset \mbV$ be a vertex lattice of  type $t< h$. Since $h$ and $t$ are always even, we still let $\tth=\frac{h}{2}$ and $\ttt=\frac{t}{2}$. Let $\ttn:=\lfloor n/2\rfloor$ and $\tth'=\ttn-\tth$.
\subsubsection{The even dimensional orthogonal case} 
Assume $n$ is even so that $n-t=2\tt(\ttn-\ttt)$ is even. 
Let $\Omega_{\Lambda^{\sharp}}=V_{\Lambda^\sharp}\otimes_{\F_\p} \F$  be the quadratic space over $\F$ with basis $\left\{\sfe_1, \ldots,  \sfe_{\ttn-\ttt}, \sff_1, \ldots, \sff_{\ttn-\ttt}\right\}$ such that $\operatorname{Span}_\F\left\{\sfe_1, \ldots, \sfe_{\ttn-\ttt}\right\}$ and $\operatorname{Span}_\F\left\{\sff_1, \ldots, \sff_{\ttn-\ttt}\right\}$ are totally isotropic and $\left(\sfe_i, \sff_j\right)=\delta_{i, j}$. If $V_{\Lambda^\sharp}$ is split, then the Frobenius $\Phi$ fixes $\sfe_1, \ldots, \sfe_{\ttn-\ttt}, \sff_1, \ldots, \sff_{\ttn-\ttt}$. If $V_{\Lambda^\sharp}$ is non-split, then $\Phi$ fixes $\sfe_1, \ldots, \sfe_{\ttn-\ttt-1}, \sff_1, \ldots, \sff_{\ttn-\ttt-1}$ but interchanges $\sfe_{\ttn-\ttt} \leftrightarrow \sff_{\ttn-\ttt}$.

Consider the isotropic flags $\mathscr{F}_{\bullet}^{+}$and $\mathscr{F}_{\bullet}^{-}$in $\Omega_{\Lambda^{\sharp}}$ defined by
$$
\begin{aligned}
\mathscr{F}_i^{ \pm} & =\operatorname{Span}_\F\left\{\sfe_1, \ldots, \sfe_i\right\} \quad\text{for}\quad 1 \leqslant i \leqslant \ttn-\ttt-1 \\
\mathscr{F}_{\ttn-\ttt}^{+} & =\operatorname{Span}_\F\left\{\sfe_1, \ldots, \sfe_{\ttn-\ttt-1}, \sfe_{\ttn-\ttt}\right\} \\
\mathscr{F}_{\ttn-\ttt}^{-} & =\operatorname{Span}_\F\left\{\sfe_1, \ldots, \sfe_{\ttn-\ttt-1}, \sff_{\ttn-\ttt}\right\} .
\end{aligned}
$$
Hence when $V_{\Lambda^\sharp}$ is split, we have $\Phi\mathscr{F}_{\bullet}^{ \pm}= \mathscr{F}_{\bullet}^{\pm }$. When $V_{\Lambda^\sharp}$ is non-split, we have $\Phi\mathscr{F}_{\bullet}^{ \pm}= \mathscr{F}_{\bullet}^{\mp }$. 
The stabilizers of  $\mathscr{F}_{\bullet}^{\pm}$ are the same, which we denote by $B \subset H=\mathrm{SO}(\Omega_{\Lambda^{\sharp}})$. Hence $B$ is a $\Phi$-stable Borel subgroup containing $T$ where $T$ is a maximal $\Phi$-stable torus of $\mathrm{SO}(\Omega_{\Lambda^{\sharp}})$. 
We use $\Delta^*=\left\{\sfs_1, \ldots,  s_{\ttn-\ttt}\right\}$ to denote the  set of corresponding simple reflections in the Weyl group $W=N(T) / T$ where
\begin{itemize}
    \item $\sfs_i$ interchanges $\sfe_i \leftrightarrow \sfe_{i+1}$ and $\sff_i \leftrightarrow \sff_{i+1}$, and fixes the other basis elements if $i<\ttn-\ttt$.
    \item $s_{\ttn-\ttt}$ interchanges $\sfe_{\ttn-\ttt-1} \leftrightarrow \sff_{\ttn-\ttt}$ and $\sff_{\ttn-\ttt-1} \leftrightarrow \sfe_{\ttn-\ttt}$, and fixes the other basis elements.
\end{itemize}
We also use $t^+$ and $t^-$ to denote $s_{\ttn-\ttt-1}$ and  $s_{\ttn-\ttt}$ respectively.

For each $0\leq \tts< \tth'\leq \ttr \leq \ttn-\ttt-1$, we let
\begin{align*}
    I_{\ttr\tts}^{}\coloneqq
    \begin{cases}
       \{\sfs_1,\ldots,\sfs_{\ttn-\ttt-\ttr-1}\}   & \text{ if $\tts=0$},\\
      \{\sfs_1,\ldots,\sfs_{\ttn-\ttt-\ttr-1}, \sfs_{\ttn-\ttt-\tts},\ldots,\sfs_{\ttn-\ttt-2},s_{\ttn-\ttt-1},s_{\ttn-\ttt}\}     & \text{ otherwise},
    \end{cases}
\end{align*}
and let $P_{\ttr\tts}^{}$ be the corresponding parabolic subgroup.  
Next, we define Weyl group elements for DL varieties.

For each $0\leq \tts<\tth'< \ttr \leq \ttn-\ttt$ (we set $\tts=0$ when $\tth'=0$),  let
\begin{itemize}
\item $
\sfg_{i}^\pm:=\begin{cases}
t^\pm & \text{if $i=\ttn-\ttt$ and $\ttn=\tth$},\\
 t^-t^+& \text{if $i=\ttn-\ttt-1$},\\
  \sfs_i \sfs_{i+1} \cdots \sfs_{\ttn-\ttt-2} t^- t^+ \sfs_{\ttn-\ttt-2} \cdots \sfs_{i+1} \sfs_i & \text{otherwise}.
\end{cases}$
\item $
\sfw_{\ttr\tts}^\pm:=\begin{cases}
g_{\ttn-\ttt-\tts}^\pm\sfs_{\tth-\ttt-2}\sfs_{\tth-\ttt-3}\cdots\sfs_{\ttn-\ttt-\ttr+1} & \text{if $\ttn=\tth$},\\
(\sfs_{\tth-\ttt}\sfs_{\tth-\ttt+1}\cdots \sfs_{\ttn-\ttt-\tts-2}\sfg_{\ttn-\ttt-1-\tts})\cdot(\sfs_{\tth-\ttt-1}\sfs_{\tth-\ttt-2}\cdots\sfs_{\ttn-\ttt-\ttr+1}) & \text{if $\ttn\neq \tth$}.
\end{cases}$
\end{itemize}
Here, when $h=n$, we define $w_{10}^\pm =\mathrm{Id}$. Note that when $h\neq n$, then $\sfw_{\ttr\tts}^\pm$ (resp. $g_i^\pm$) is independent of $\pm$ sign and we will simply denote it as  $\sfw_{\ttr\tts}$ (resp. $g_i$) in this case.  

When $\tth=\ttn-1$, for each $0\leq \tts <\tth'< \ttr \leq \ttn-\ttt$,  let
\begin{itemize}
\item $w_{\ttr0}^{\prime,+}= \sfs_{\tth-\ttt}\cdots \sfs_{\ttn-\ttt-\ttr+1}= \sfs_{\ttn-\ttt-1}\cdots \sfs_{\ttn-\ttt-\ttr+1}=t^+ \sfs_{\ttn-\ttt-2}\cdots \sfs_{\ttn-\ttt-\ttr+1}$
\item $\sfw^{\prime,-}_{\ttr0}:=
  t^-\sfs_{\ttn-\ttt-2}\cdots \sfs_{\ttn-\ttt-\ttr+1}$.
\end{itemize}

When $\tth<\ttn-1$, for each $0\leq \tts <\tth'< \ttr \leq \ttn-\ttt$,  let
\begin{itemize}
\item $\sfw'_{\ttr\tts}=\sfw_{\ttr\tts}^{\prime,+}:=
   (\sfs_{\tth-\ttt}\sfs_{\tth-\ttt-1}\cdots \sfs_{\ttn-\ttt-\ttr+1})\cdot (\sfs_{\tth-\ttt+1} \sfs_{\tth-\ttt+2}\cdots \sfs_{\ttn-\ttt-\tts-1})$.
\end{itemize}
In particular, $\sfw_{\ttr\tts}=\sfs_{\tth-\ttt}\sfs_{\tth-\ttt+1}\cdots \sfs_{\ttn-\ttt-\tts-2}\sfg_{\ttn-\ttt-\tts-1}$ if $\tth=\ttt+1$ or $\ttr=\ttn -\tth+1$, $\sfw_{\ttr(\tth'-1)}=\sfg_{\tth-\ttt}\sfs_{\tth-\ttt-1}\sfs_{\tth-\ttt-2} \cdots \sfs_{\ttn-\ttt-\ttr+1}$ if $\ttn\neq \tth$.
Note that the element $\sfg_i$ interchanges $\sfe_i$ with  $\sff_i$,  $\sfe_{\ttn-\ttt}$ with $ \sff_{\ttn-\ttt}$ and fixes all the other  basis vectors when $i<\ttn-\ttt$ and $g_{\ttn-\ttt}^{\pm}=t^\pm$.

\begin{example}
\begin{altenumerate}
\item   Assume $\ttn-\ttt=3$ and $\tth'=1$. Hence $\tth-\ttt=2$. Then
\begin{equation*}
w_{30}=t^+t^-s_1,\quad w_{20}=t^+t^-,\quad 
w_{30}^{\prime,\pm}= t^\pm s_1,\quad
w_{20}^{\prime,\pm}= t^\pm.
\end{equation*}
\item Assume $\ttn-\ttt=3$ and $\tth'=2$. Hence $\tth-\ttt=1$. Then
\begin{equation*}
 w_{30}= s_1t^-t^+,\quad
 w_{31}= s_1t^-t^+s_1,\quad
w_{30}^{\prime}=  s_1 t^+,\quad
w_{31}^{\prime}= s_1.
\end{equation*}
\end{altenumerate}
\end{example}

Since the case $\ttn=\tth$ is treated in detail in \cite{HP2}, we assume $\ttn\neq \tth$  until the end of this section unless explicitly mentioned.
\begin{lemma}\label{lem:rel-pos even ortho}
\begin{altenumerate}
\item The element $\sfw_{\ttr(\tth'-1)}$ stabilizes $\sfe_i$ and $\sff_i$ for $1\leq i\leq \ttn-\ttt-\ttr$ and $\tth-\ttt+1\leq i\leq \ttn-\ttt-1$, and switches $e_{\ttn-\ttt}$ and $f_{\ttn-\ttt}$.
It acts on the remaining elements in the basis as follows:
\begin{equation*}
\begin{aligned}
\xymatrix{
\sfe_{\ttn-\ttt-\ttr+1}\ar@/^1pc/[rrrrr]
&\cdots \ar@/^/[l]&
\sfe_{\tth-\ttt}\ar@/^/[l]&
\sff_{n-\ttt-\ttr+1}\ar@/^/[l]
&\cdots\ar@/^/[l]&
\sff_{\tth-\ttt}.\ar@/^/[l]
}
\end{aligned}
\end{equation*}
Moreover, $\sfw_{\ttr(\tth'-1)}$ is the minimal representative element in the double coset $W_{\ttr}\sfw_{\ttr(\tth'-1)}W_{\ttr}$. 
\item 
If $\tts<\tth'-1$, then the element $\sfw_{\ttr\tts}$  stabilizes $\sfe_i$ and $\sff_i$ for $1\leq i\leq \ttn-\ttt-\ttr$ and $\ttn-\ttt-\tts\leq i\leq \ttn-\ttt-1$, and switches $e_{\ttn-\ttt}$ and $f_{\ttn-\ttt}$.
It acts on the remaining $\sfe_i$'s as follows:
\begin{equation*}
\begin{aligned}
\xymatrix{
\sfe_{\ttn-\ttt-\ttr+1}\ar@/^1pc/[rrr]&
\cdots\ar@/^/[l]&
\sfe_{\tth-\ttt}\ar@/^/[l]&
\sfe_{\tth-\ttt+1}\ar@/^/[r]&
\cdots\ar@/^/[r]&
\sfe_{\ttn-\ttt-\tts-1}\ar[dlll]\\
\sff_{\ttn-\ttt-\ttr+1}\ar@/_1pc/[rrr] &
\cdots \ar@/_/[l]&
\sff_{\tth-\ttt}\ar@/_/[l] &
\sff_{\tth-\ttt+1}\ar@/_/[r]&
\cdots\ar@/_/[r]&
\sff_{\ttn-\ttt-\tts-1}\ar[ulll]|!{[u];[lll]}\hole
}.
\end{aligned}
\end{equation*}
Moreover, $\sfw_{\ttr\tts}$ is the minimal representative element in the double coset $W_{\ttr\tts}\sfw_{\ttr\tts}W_{\ttr\tts}$.
\item The element $\sfw'_{\ttr\tts}$ stabilizes $\sfe_i$ and $\sff_i$ for $1\leq i\leq \ttn-\ttt-\ttr$ and $\ttn-\ttt-\tts+1\leq i\leq \ttn-\ttt$.
It acts on the remaining $\sfe_i$'s as follows:
\begin{equation*}
\begin{aligned}
\xymatrix{
\sfe_{\ttn-\ttt-\ttr+1}\ar@/^1pc/[rrr]&
\cdots\ar@/^/[l]&
\sfe_{\tth-\ttt}\ar@/^/[l]&
\sfe_{\tth-\ttt+1}\ar@/^/[r]&
\sfe_{\tth-\ttt+2}\ar@/^/[r]&
\cdots\ar@/^/[r]&
\sfe_{\ttn-\ttt-\tts}\ar@/^1pc/[llll]
}
\end{aligned}
\end{equation*}
We have the same action on the remaining $\sff_i$'s.

If $\tth=\ttn-1$, then the element $\sfw_{\ttr0}^{\prime,-}$ stabilizes $\sfe_i$ and $\sff_i$ for $1\leq i\leq \ttt-\ttr$. 
It acts on the remaining $\sfe_i$'s as follows:
\begin{equation*}
\begin{aligned}
\xymatrix{
\sfe_{\ttn-\ttt-\ttr+1}\ar@/^1pc/[rrrrr]
&\cdots \ar@/^/[l]&
\cdots\ar@/^/[l]&
\sfe_{\ttn-\ttt-2}\ar@/^/[l]
&\sfe_{\ttn-\ttt-1}\ar@/^/[l]&
\sff_{\ttn-\ttt}.\ar@/^/[l]
}
\end{aligned}
\end{equation*}
We have the same action on the remaining $\sff_i$'s.

Moreover, $\sfw'_{\ttr\tts}$ is the minimal representative element in the double coset $W_{\ttr\tts}\sfw'_{\ttr\tts} W_{\ttr\tts}$.\qed
\end{altenumerate}
\end{lemma}

\subsubsection{The odd dimensional orthogonal case}
In this subsection, we assume $n=2\ttn+1$ is odd. 
Let $\Omega_{\Lambda^{\sharp}}=V_{\Lambda^\sharp}\otimes_{\F_\p} \F$ be a quadratic space over $\F$. Since $V_{\Lambda^\sharp}$ is of odd dimension,  we can choose a basis $\left\{\sfe_1, \ldots, \sfe_{\ttn-\ttt}, \sff_1, \ldots, \sff_{\ttn-\ttt}, \sfe_{2(\ttn-\ttt)+1}\right\}$  such that $\operatorname{Span}_\F\left\{\sfe_1, \ldots, \sfe_{\ttn-\ttt}\right\}$ and $\operatorname{Span}_\F\left\{\sff_1, \ldots, \sff_{\ttn-\ttt}\right\}$ are totally isotropic and $\left(\sfe_i, \sff_j\right)=\delta_{i, j}$ and $(\sfe_{2(\ttn-\ttt)+1},\sfe_{2(\ttn-\ttt)+1})=1$ and $\sfe_{2(\ttn-\ttt)+1}$ is orthogonal to all the other basis vectors and $\Phi$ fixes all the basis vectors.

Consider the isotropic flags $\mathscr{F}_{\bullet}$ in $\Omega_{\Lambda^{\sharp}}$ defined by
$$
\begin{aligned}
\mathscr{F}_i & =\operatorname{Span}_\F\left\{\sfe_1, \ldots, \sfe_i\right\} \text { for } 1 \leqslant i \leqslant \ttn-\ttt.
\end{aligned}
$$
The stabilizer of  $\mathscr{F}_{\bullet}$ is denoted by $B \subset H=\mathrm{SO}(\Omega_{\Lambda^\sharp})$. Hence $B$ is a $\Phi$-stable Borel subgroup containing $T$ where $T$ is a maximal $\Phi$-stable torus of $\mathrm{SO}(\Omega_{\Lambda^\sharp})$. 

We use $\Delta^*=\left\{\sfs_1, \ldots, \sfs_{\ttn-\ttt-2}, s_{\ttn-\ttt-1}, s_{\ttn-\ttt}\right\}$ to denote the  set of corresponding simple reflections in the Weyl group $W=N(T) / T$, where
\begin{itemize}
    \item $\sfs_i$ interchanges $\sfe_i \leftrightarrow \sfe_{i+1}$ and $\sff_i \leftrightarrow \sff_{i+1}$, and fixes the other basis elements if $i<\ttn-\ttt$.
    \item $s_{\ttn-\ttt}$ interchanges $\sfe_{\ttn-\ttt} \leftrightarrow \sff_{\ttn-\ttt}$ and $e_{2(\ttn-\ttt)+1} \leftrightarrow -\sfe_{2(\ttn-\ttt)+1}$, and fixes the other basis elements.
\end{itemize}
For each $0\leq \tts\leq \tth'\leq \ttr \leq \ttn-\ttt-1$, we let
\begin{align*}
    I_{\ttr\tts}^{}\coloneqq
    \begin{cases}
       \{\sfs_1,\ldots,\sfs_{\ttn-\ttt-\ttr-1}\}   & \text{ if $h\ge n-2$},\\
      \{\sfs_1,\ldots,\sfs_{\ttn-\ttt-\ttr-1}, \sfs_{\ttn-\ttt-\tts+1},\ldots,\sfs_{\ttn-\ttt-2},s_{\ttn-\ttt-1},s_{\ttn-\ttt}\}     & \text{ if $h<n-2$},
    \end{cases}
\end{align*}
and let $P_{\ttr\tts}^{}$ be the corresponding parabolic subgroup.  

For each $0\leq \tts\leq \ttn-\tth< \ttr \leq \ttn-\ttt$,  let
\begin{itemize}
\item $\sfg_{i}:=\sfs_i \sfs_{i+1} \cdots \sfs_{\ttn-\ttt-1} \sfs_{\ttn-\ttt} \sfs_{\ttn-\ttt-1} \cdots \sfs_{i+1} \sfs_i.$
\item $\sfw_{\ttr\tts}:=(\sfs_{\tth-\ttt}\sfs_{\tth-\ttt+1}\cdots \sfs_{\ttn-\ttt-\tts-1}\sfg_{\ttn-\ttt-\tts})\cdot(\sfs_{\tth-\ttt-1}\sfs_{\tth-\ttt-2}\cdots\sfs_{\ttn-\ttt-\ttr+1})$.
\item $\sfw'_{\ttr\tts}:=
   (\sfs_{\tth-\ttt}\sfs_{\tth-\ttt-1}\cdots \sfs_{\ttn-\ttt-\ttr+1})\cdot (\sfs_{\tth-\ttt+1} \sfs_{\tth-\ttt+2}\cdots \sfs_{\ttn-\ttt-\tts-1})$ if $\tts<\ttn-\tth$.
\end{itemize}

In particular, $\sfw_{\ttr\tts}=\sfs_{\tth-\ttt}\sfs_{\tth-\ttt+1}\cdots \sfs_{\ttn-\ttt-\tts-1}\sfg_{\ttn-\ttt-\tts}$ if $\ttt=\tth-1$ or $\ttr=\ttn-\tth+1$, and $\sfw_{\ttr,\ttn-\tth}=\sfg_{\tth-\ttt}\sfs_{\tth-\ttt-1}\sfs_{\tth-\ttt-2} \cdots \sfs_{\ttn-\ttt-\ttr+1}$, and $\sfw'_{\ttr,\ttn-\tth-1}=\sfs_{\tth-\ttt}\sfs_{\tth-\ttt-1}\cdots \sfs_{\ttn-\ttt-\ttr+1}$.
Note that the element $\sfg_i$ interchanges $\sfe_i\leftrightarrow \sff_i$, $\sfe_{2d+1}\leftrightarrow -\sfe_{2d+1}$, and fixes all other elements in the basis.

\subsubsection{Orthogonal Deligne--Lusztig varieties}
For a quadratic space $\Omega$ over $\F$, let $\operatorname{OGr}(\tth-\ttt,\Omega)$ be the orthogonal Grassmannian that parameterizes the isotropic $(\tth-\ttt)$-dimensional subspaces of $\Omega$. Let $R_{\Lambda^\sharp}^{} \subset \operatorname{OGr}(\tth-\ttt,\Omega_{\Lambda^{\sharp}})$  be the reduced closed subscheme such that for any field $k$ over $\F$, we have
\begin{equation}\label{equ:R-Lambda-defn}
    R_{\Lambda^\sharp}^{}(k)  = \{\cV\subset\Omega_{\Lambda^{\sharp},k}\mid\cV\text{ is isotropic},\,\text{and}\,
  \mathrm{dim}_k \cV=\tth-\ttt, \,\text{and}\,\operatorname{dim}_k(\cV+\Phi(\cV))\le \tth-\ttt+1 \}.
\end{equation}
\begin{example}
\begin{altenumerate}
\item When $\tth-\ttt=1$, $R_{\Lambda^\sharp}$ parameterizes lines in $\Omega$ that are isotropic. Hence, it is a quadratic hypersurface in $\mbP(\Omega)$. 
In particular, when $\dim \Omega=n-t=3$, this is isomorphic to a projective line.
When $\dim\Omega=n-t=4$, this is isomorphic to a ruled surface.

\item When $\dim\Omega=n-t=4$ and $\tth-\ttt=2$, $R_{\Lambda^\sharp}$ parameterizes $2$-dimensional isotropic subspaces in $\Omega$, which are isomorphic to the disjoint union of two projective lines.
\end{altenumerate}
\end{example}

The following result can be proved in essentially the same way as Theorem \ref{thm:symplectic main}, and we will not repeat it here (the case $h=n$ is a bit different and is treated in detail in \cite{HP2}). We also refer the readers to \cite{HZBasic} for a group-theoretic approach.
\begin{theorem}\label{thm:orthogonal main}
Let $\tth':=\ttn-\tth$ and $\ttt':=\ttn-\ttt$, where $\ttn=\lfloor n/2\rfloor$ and $\tth=h/2$ and $\ttt=t/2$.
We have the following stratification:
\begin{align*}\label{eq:orthogonal stratification}
    R_{\Lambda^\sharp}=
    \begin{cases}
     \displaystyle  \coprod_{\delta\in \{\pm\}}R_{\Lambda^\sharp}^{\delta} =\coprod_{\delta\in\{\pm\}}\big(\coprod_{\substack{0\leq j\le \tth'< i\leq \ttt' }}X_{P_{ij}}(\sfw_{ij}^{\delta})\big) 
     & \text{ if $h=n$},\\
     \bigl(\displaystyle\coprod_{0\leq j< \tth'< i\leq \ttt'}X_{P_{ij}}(\sfw_{ij})\bigr)\amalg
    \bigl(\coprod_{\substack{0\leq j<\tth'< i\leq \ttt',\\
    \delta\in \{\pm\}}} X_{P_{ij}}(\sfw^{\prime,\delta}_{ij}  )\bigr)
    \amalg X_{P_{\tth'\tth'}}(\id) & \text{ if $h=n-2$},\\
         \bigl(\displaystyle\coprod_{0\leq j< \tth'< i\leq \ttt'}X_{P_{ij}}(\sfw_{ij})\bigr)\amalg
    \bigl(\coprod_{0\leq j<\tth'< i\leq \ttn-\ttt}X_{P_{ij}}(\sfw'_{ij})\bigr)
    \amalg X_{P_{\tth'\tth'}}(\id) & \text{ otherwise },
    \end{cases}
\end{align*}
such that for each  $0\leq \tts <\tth'\leq \ttr\leq \ttt'$, we have the following closure relations: 
\begin{align*}
\overline{X_{P_{\ttr\tts}}(\sfw^{\pm}_{\ttr\tts})}
={}&\begin{cases}
\displaystyle\coprod_{0\leq j\le \tts<\tth'< i\leq \ttr}X_{P_{ij}}(\sfw_{ij}^{\pm}) &   \text{ if $h=n$},\\
    \displaystyle\Bigl(\coprod_{0\leq j\le \tts< \tth'< i\leq \ttr}X_{P_{ij}}(\sfw_{ij})\Bigr)\amalg
    \Bigl(\coprod_{\substack{\tts\leq j<\tth'<i\leq \ttr,\\ \delta\in\{\pm\}}} X_{P_{ij}}(\sfw^{\prime,\delta}_{ij})\Bigr)
    \amalg X_{P_{\tth'\tth'}}(\id) & \text{ if $h=n-2$,}\\
     \displaystyle\Bigl(\coprod_{0\leq j\le \tts<\tth'< i\leq \ttr}X_{P_{ij}}(\sfw_{ij})\Bigr)\amalg
    \Bigl(\coprod_{\tts\leq j<\tth'<i\leq \ttr}X_{P_{ij}}(\sfw^{\prime}_{ij})\Bigr)
    \amalg X_{P_{\tth'\tth'}}(\id) & \text{ otherwise.}\\
\end{cases} \\
\overline{X_{P_{\ttr\tts}}(\sfw^{\prime,\pm}_{\ttr\tts})}={}&\Bigl(\coprod_{\tts\leq j<\tth'<i\leq \ttr}X_{P_{ij}}(\sfw^{\prime,\pm}_{ij})\Bigr)
    \amalg X_{P_{\tth'\tth'}}(\id). 
\end{align*}
The dimension of $R_{\Lambda^\sharp}$ is  $n-(\ttt+\tth)-1$.  Moreover, $R_{\Lambda^\sharp}$ is irreducible and normal if $h\neq n$. If $h=n$,  $R_{\Lambda^\sharp}^{\delta}$ is irreducible and normal.
\end{theorem}

\subsection{General linear case}\label{sec:DL_GL}
In this subsection, we present the necessary results concerning generalized DL varieties associated with the general linear group.
Let $\Lambda_1\subset \Lambda_2$ be vertex lattices of type $t_1$ and $t_2$, resp., such that $t_2\leq h\leq t_1$. 
Recall that we use the notation $t_i=2\ttt_i$.
Let $V=V_{[\Lambda_1,\Lambda_2]}=\Lambda_2/\Lambda_1$ denote the vector space of dimension $t=\ttt_1-\ttt_2$ over $\F_q$ and define $\Omega=\Omega_{[\Lambda_1,\Lambda_2]}:=V\otimes_{\F_q}\F$, with Frobenius $\Phi$. 
We choose a basis $V=\mathrm{Span}(\sfe_{\ttt_2+1},\cdots,\sfe_{\ttt_1})$. Note that the notation used here is not standard.
Consider the standard flags $\sF_\bullet$ defined by
\begin{equation*}
    \sF_i=\mathrm{Span}_\F\{\sfe_{\ttt_2+1},\cdots,\sfe_i\},\quad \ttt_2\leq i\leq \ttt_1.
\end{equation*}
In particular, we have $\sF_{\ttt_2}=(0)$ and $\dim_{\F}\sF_i=i-\ttt_2$. This pins the choice of the maximal torus and the Borel $T\subset B\subset G:=\GL(\Omega_{[\Lambda_1,\Lambda_2]})$.
We use $\Delta^*=\{\sfs_{\ttt_2+1},\cdots,\sfs_{\ttt_1-1}\}$ to denote the set of corresponding simple reflections in the Weyl group $W$, where $\sfs_i$ is the permutation between $\sfe_i$ and $\sfe_{i+1}$. 
For each $\ttt_2\leq \tts\leq \tth\leq \ttr\leq \ttt_1$, let
\begin{equation*}
    I_{\ttr\tts}^{}:=\{\sfs_{\ttt_2},\cdots,\sfs_{\tts-1},\sfs_{\ttr+1},\cdots,\sfs_{\ttt_1-1}\}.
\end{equation*}
Denote by $W^{}_{\ttr\tts}$ the subgroup of the Weyl group generated by elements in $I_{\ttr\tts}^{}$ and by $P_{\ttr\tts}^{}$ the corresponding parabolic subgroup.

The partial flag variety $G/P_{\tth\tth}$ parametrizes the subspaces in $\Omega$ of dimension $\ttt_1-\tth$. We consider the subvariety $S_{[\Lambda_1,\Lambda_2]}$ of $G/P_{\tth\tth}$ whose $k$-points for any field $k$ over $\F$ are given by
\begin{equation}\label{equ:S Lambda12}
S_{[\Lambda_1,\Lambda_2]}(k)=\{\cV\subset\Omega_{[\Lambda_1,\Lambda_2],k}\mid \dim_k\cV=\ttt_1-\tth\quad\text{and}\quad \dim_k(\cV\cap\Phi(\cV))\geq \ttt_1-\tth-1\}.
\end{equation}

For each  $\ttt_2\leq \tts \leq \tth\leq \ttr\leq \ttt_1$, we let
\begin{equation*}
\sfw_{\ttr\tts}:=
\left(
\sfs_{\tth}\cdots\sfs_{\ttr-2}\sfs_{\ttr-1}
\right)\cdot\left(
\sfs_{\tth-1}\cdots \sfs_{\tts+2}\sfs_{\tts+1}
\right).
\end{equation*}
The main result is the following:
\begin{proposition}\label{Prop: decom SLambda1Lambda2 into DL var}
We have a stratification of
$S^{}_{[\Lambda_1,\Lambda_2]}$
\begin{equation*}
    S^{}_{[\Lambda_1,\Lambda_2]}=\coprod_{\ttt_2\leq j\leq \tth\leq i\leq \ttt_1}X_{P_{ij}}(\sfw^{}_{ij}),
\end{equation*}
such that for each $\ttt_2\leq \tts \leq \tth\leq \ttr\leq \ttt_1$, we have:
\begin{equation*}
    \ov{X_{\ttr\tts}(\sfw^{}_{\ttr\tts})}=\coprod_{\tts\leq j\leq \tth\leq i\leq \ttr}X_{P_{ij}}(\sfw^{}_{ij}).
\end{equation*}
In particular, $S^{}_{[\Lambda_1,\Lambda_2]}$ is irreducible and normal of dimension $\ttt_1-\ttt_2-1$.
\end{proposition}
\begin{proof}
This follows from a similar but much simpler argument than the one in the proof of Theorem \ref{thm:symplectic main}, combined with Proposition \ref{prop:GL-DLV}.
\end{proof}

\begin{proposition}\label{prop:GL-DLV}
For any integers $\ttt_2\leq \tts\leq \tth<\ttr\leq \ttt_1$, the DL variety $X_{P_{\ttr\tts}}(\sfw_{\ttr\tts}^{})$ is the subvariety of $G/P^{}_{\ttr\tts}$ whose $k$-points for any field $k$ over $\F$ are characterized by:
\begin{equation*}
X_{P_{\ttr\tts}}(\sfw_{\ttr\tts}^{})(k):=\left\{\begin{array}{r}
(\cF_{\ttt-\ttr}\subset \ldots \subset \cF_{\ttt-\tts})\\
\in  G/P_{\ttr\tts}(k)
\end{array}\left\vert
\begin{array}{l}\sfF_i=\sfF_{i+1}\cap \Phi(\sfF_{i+1})\text{ for }\tts\leq i\leq \tth-1;\\
\sfF_j=\sfF_{j-1}+\Phi(\sfF_{j-1})\text{ for }\tth+1\leq j\leq \ttr;\\
\sfF_{\ttt}=\Phi(\sfF_{\ttt})\text{ and }\sfF_{\tts}=\Phi(\sfF_{\tts}).
\end{array}\right\}\right..
\end{equation*}
It is smooth of dimension $\ttr-\tts-1$. Moreover, the DL variety $X_{P_{[\ttt_1,\ttt_2]}}(\sfw_{[\ttt_1,\ttt_2]}^{})$ is irreducible.
\end{proposition}
\begin{proof}
The element $\sfw_{\ttr\tts}$ acts on $V_{[\Lambda_1,\Lambda_2]}$ by
\begin{equation*}
\begin{aligned}
\xymatrix{
\sfe_{\tts+1}\ar@/^1pc/[rrrr]&
\cdots\ar@/^/[l]&
\sfe_{\tth-1}\ar@/^/[l]&
\sfe_{\tth}\ar@/^/[l]&
\sfe_{\tth+1}\ar@/^/[r]&
\sfe_{\tth+2}\ar@/^/[r]&
\cdots\ar@/^/[r]&
\sfe_{\ttr}\ar@/^1pc/[llll]
}
\end{aligned}
\end{equation*}
One can immediately check that:
\begin{align*}
    & \sF_i=\sF_{i+1}\cap \sfw_{\ttr\tts}(\sF_{i+1}),& i=\tts+1,\cdots,\tth-1;\\
    & \sF_j=\sF_{j-1}+\sfw_{\ttr\tts}(\sF_{j-1}),& j=\tth+1,\cdots,\ttr-1;\\
    &\sF_{k}=\sfw_{\ttr\tts}(\sF_{k}),&k=\ttr,\tts.
\end{align*}
The assertions now follow directly.
\end{proof}

\section{Relation between Bruhat--Tits strata and Deligne--Lusztig varieties}\label{sec:iden}
In this section, we establish the identification between the BT-strata defined in \S \ref{sec:BT} and certain spaces introduced in \S \ref{sec: DL var}. 
While one could construct this isomorphism directly using $O_{F_0}$-displays as indicated in \eqref{equ:display-U(X)}, we instead follow the $p$-divisible group construction to maintain consistency with the existing literature \cites{Vollaard,Vollaard-Wedhorn,RTW,Wu16,HLSY,KR-BG,Cho18}.

\subsection{$\cZ$-strata stratification}\label{sec:Z-strata-strata}
Throughout the remainder of this paper, we fix a framing object $\bX_{n,\ep}^{[h]}$. Recall from \S \ref{sec:framing} the hermitian space $\bV$ over $F$ of dimension $n$ with $\mathrm{Hasse}(\bV)=-\ep$.
Let $\Lambda\subset \bV$ be any vertex lattice of type $t(\Lambda)\geq h$. We will construct an isomorphism
$$
\Psi_{\cZ}: \cZ(\Lambda)\cong S_{\Lambda },
$$
where $S_{\Lambda}$ is defined in \eqref{equ:S-Lambda-defn}.
For any $\F$-algebra $R$ and any $R$-point $(X,\iota,\lambda,\rho)\in{\cZ}(\Lambda)(R)$ in the $\cZ$-stratum (see Definition \ref{def:YZ-cycles}), we have a chain of isogenies
\begin{equation*}
    \rho_{\Lambda,\Lambda^\sharp}:X_{\Lambda,R}\xrightarrow{\rho_{\Lambda,X}}X\xrightarrow{\lambda}X^\vee\xrightarrow{\rho_{X^\vee,\Lambda^\sharp}} X_{\Lambda^\sharp,R}.
\end{equation*}
Let $\rho_{X,\Lambda^\sharp}$ be the composition $\rho_{X^\vee,\Lambda^\sharp}\circ\lambda$.
Applying de Rham realization, we have  a sequence of $R$-modules:
\begin{equation*}
D(\rho_{\Lambda,\Lambda^\sharp}):D(X_{\Lambda,R})\xrightarrow{\rho_{\Lambda,X}}D(X)\xrightarrow{\lambda}D(X^\vee)\xrightarrow{\rho_{X^\vee,\Lambda^\sharp}} D(X_{\Lambda^\sharp,R}).
\end{equation*}
By definition, the image $\mathrm{Im}(D(\rho_{\Lambda,\Lambda^\sharp}))$ is a locally free direct summand of $D(X_{\Lambda^\sharp,R})=\Lambda^\sharp\otimes_{\F_q} R$ of corank $t=t(\Lambda)=2\ttt $, such that
\begin{equation*}
    D(X_{\Lambda^\sharp,R})\slash\mathrm{Im}(D(\rho_{\Lambda,\Lambda^\sharp}))\simeq \Lambda^\sharp\slash\Lambda\otimes_{\F_q} R=\Omega_{\Lambda,R}.
\end{equation*}
Since $\Lambda$ is a vertex lattice, we have $\ker(\rho_{\Lambda,\Lambda^\sharp})\subseteq X_{\Lambda}[\iota(\pi)]$. This implies that the kernel of the composition $$\rho_{X,\Lambda^\sharp}:=\rho_{X^\vee,\Lambda^\sharp}\circ\lambda:X\longrightarrow X_{\Lambda^\sharp,R}$$ 
lies in $X[\iota(\pi)]$. Therefore, there exists an isogeny $\widetilde{\rho}_{X,\Lambda^\sharp}:X_{\Lambda^\sharp}\rightarrow X$ 
such that $\widetilde{\rho}_{X,\Lambda^\sharp}\circ \rho_{X,\Lambda^\sharp}=\iota(\pi):X\to X$.

\begin{lemma}\label{BT_Z_isotropic-subspace}
For $(X,\iota,\lambda,\rho)\in\cZ(\Lambda)(R)$, we have induced filtrations
\begin{equation*}
\begin{aligned}
\xymatrix{
D(X_{\Lambda,R})\ar[r]^-{D(\rho_{\Lambda,\Lambda^\sharp})}&D(X_{\Lambda^\sharp,R})\ar[r]^-{D(\widetilde{\rho}_{X,\Lambda^\sharp})}&D(X) \\
\Pi D(X_{\Lambda,R})\ar[r]\ar@{^(->}[u]&\Pi D(X_{\Lambda^\sharp,R})\ar@{^(->}[u]\ar[r]&\Fil(X).\ar@{^(->}[u]
}
\end{aligned}
\end{equation*}
The preimage $D(\widetilde{\rho}_{X,\Lambda^\sharp})^{-1}(\Fil(X))\subseteq D(X_{\Lambda^\sharp,R})$ is a locally free direct summand that contains $\mathrm{Im}(D(\rho_{\Lambda,\Lambda^\sharp}))$. Moreover, the quotient
\begin{equation}
U(X):=D(\widetilde{\rho}_{X,\Lambda^\sharp})^{-1}(\Fil(X))\slash\mathrm{Im}(D(\rho_{\Lambda,\Lambda^\sharp}))\subset \Omega_{\Lambda,R}
\end{equation}
    is a locally free isotropic direct summand of rank $\frac{1}{2}(t-h)$.
\end{lemma}

\begin{proof}
By universality, it suffices to check the case where $\Spec R$ is an affine sub-formal scheme of $\cZ(\Lambda)$. In this case, by Nakayama's lemma, it suffices to check the condition on $\kappa$-points of $\cZ(\Lambda)$, where $\kappa$ is any algebraically closed field over $\F$. 
Recall that we use the notation $O=O_{F_0}$.
For the remainder of the proof, fix one such $\kappa$ and denote $\breve{\Lambda}:=\Lambda\otimes_O W_O(\kappa)$ and $\breve{\Lambda}^\sharp:=\Lambda^\sharp\otimes_O W_O(\kappa)$, where $W_O(\kappa)$ denotes the ring of $O$-Witt vectors (Note that elsewhere in the paper, we used the notation $\breve{\Lambda}:=\Lambda\otimes_O O_{\breve{F_0}}$).

By Proposition \ref{prop:BT_strata}, a point $(X,\iota,\lambda,\rho)\in\cZ(\Lambda)(\kappa)$ corresponds to a chain of Dieudonn\'e lattices 
\begin{equation*}
    \breve{\Lambda}\subseteq M(X)\subseteq M(X)^\sharp\subseteq \breve{\Lambda}^\sharp.
\end{equation*}
The isogeny $\widetilde{\rho}_{X,\Lambda^\sharp}$ induces a map between Dieudonn\'e lattices 
\begin{equation*}
\breve{\Lambda}^\sharp\longrightarrow M(X),\quad x\mapsto \Pi\cdot x.
\end{equation*}
Since $\Fil(X)=\uV M(X)/\pi_0 M(X)$, we have
    \begin{equation*}
        M(\widetilde{\rho}_{X,\Lambda^\sharp})^{-1}(\Fil(X))=\Pi^{-1}V M/\pi_0\breve{\Lambda}^\sharp=\tau^{-1}M/\pi_0\breve{\Lambda}^\sharp.
    \end{equation*}
Moreover, since $\breve{\Lambda}\subseteq M(X)$, it follows that $\breve{\Lambda}=\tau^{-1}(\breve{\Lambda})\subseteq\tau^{-1}(M(X))$. Therefore
$$\text{Im}(M((\rho_\Lambda)_R))=\breve{\Lambda}/\pi_0 
\breve{\Lambda}^\sharp\subseteq M(\widetilde{\rho}_{X,\Lambda^\sharp})^{-1}(\Fil(X)).$$
The quotient 
$$U(X)\cong \tau^{-1}(M)\slash\breve{\Lambda}=\Phi^{-1}(M\slash\breve{\Lambda})\subset \Lambda^{\sharp}\slash\Lambda\otimes \kappa =\Omega_{\Lambda,\kappa}$$ 
has dimension $\frac{1}{2}(t-h)$, where $\Phi$ is the Frobenius in $\Omega_{\Lambda,\kappa}$.

For any $x\in M$ and $y\in\breve{\Lambda}^\sharp$, we have $\langle \bar{x},\bar{y}\rangle_{\breve{\Lambda}^\sharp\slash\breve{\Lambda}}=0$ if and only if $\pi\langle x,y\rangle\in (\pi)\subset O_F\otimes_{O}W_O(\kappa)$. This is equivalent to $y\in M^\sharp$, which implies that $M^\sharp/\breve{\Lambda}$ is the orthogonal complement of $M/\breve{\Lambda}$ in $\Omega_{\Lambda,\kappa}$. In particular, $\Phi^{-1}(M/\breve{\Lambda})$ is an isotropic subspace.
\end{proof}

For the remainder of our discussion, recall that $\Omega_{\Lambda} = \Lambda^{\sharp}/\Lambda \otimes \F$, when $t(\Lambda) \neq h$. The partial flag variety $\Sp(\Omega_{\Lambda})/P_{\tth\tth}$ parameterizes isotropic subspaces of dimension $\ttt-\tth$, and $S_{\Lambda}$ in \S \ref{sec:DL_symp} is defined as a closed subvariety whose $k$-points for any field $k$ over $\F$ are given by
\begin{equation*}
    S_{\Lambda}^{}(k)=\{\cV\subset \Omega_{\Lambda,k}\mid \cV\text{ is isotropic},\,\text{and}\, \dim_k \cV=\ttt-\tth,\,\text{and}\, \dim_k(\cV\cap \Phi(\cV))\geq \ttt-\tth-1\}.
\end{equation*}
When $t(\Lambda)=h$, the space $S_\Lambda$ consists of the zero-dimensional subspace. Hence, the stratification in Theorem \ref{thm:symplectic main} degenerates into $S_\Lambda=X_{P_{\tth\tth}}(\id)$, which is irreducible of dimension $0$.
By Lemma \ref{BT_Z_isotropic-subspace}, we define the map
\begin{equation*}
    \Psi_{\cZ}:\cZ(\Lambda)\longrightarrow \Sp(\Omega_\Lambda)/P_{\tth\tth},\quad (X,\iota,\lambda,\rho)\mapsto U(X).
\end{equation*}

\begin{lemma}\label{lem:Z-strata_points}
Let $\kappa$ be any perfect field over $\F$. The map $\Psi_{\cZ}$ defines a bijection between $\cZ(\Lambda)(\kappa)$ and $S_{\Lambda}(\kappa)$, hence it factors through $S_\Lambda$, which we will still denote by $\Psi_\cZ$.
\end{lemma}
\begin{proof}
By Dieudonn\'e theory, a point $z\in \cZ(\Lambda)(\kappa)$ corresponds to a lattice $M$ satisfying
$$\breve{\Lambda}\stackrel{\ttn-\tth-\ttt}{\subset} M\stackrel{t}{\subset} M^\sharp\stackrel{\ttn-\tth-\ttt}{\subset}\breve{\Lambda}^\sharp,$$ see Proposition \ref{prop:BT_strata}. 
The wedge condition implies $M\cap \tau(M)\stackrel{\leq 1}{\subset} M$, which is equivalent to
\begin{equation*}
    \Phi^{-1}(M/\breve{\Lambda})\cap M/\breve{\Lambda}\stackrel{\leq 1}{\subset}\Phi^{-1}(M/\breve{\Lambda}).
\end{equation*}
Therefore, $\Psi_{\cZ}(z)\in S_\Lambda(\kappa)$.

Conversely, let $U\in S_\Lambda(\kappa)$ be any geometric point. The dual space $U^\perp\subset \Omega_\Lambda$ is a subspace of dimension $\ttn-\tth+\ttt$ containing $U$. By the definition of $S_\Lambda$, we have
\begin{equation}\label{equ:Z-strata_points}
    U\cap \Phi(U)\stackrel{\leq 1}{\subset} U.
\end{equation}
We denote by $M\subset \breve{\Lambda}^\sharp$  (resp. $M_{\Phi(U)^\sharp}\subset \breve{\Lambda}^\sharp$) the preimage of $\Phi(U)$ (resp. $\Phi(U)^\perp$) under the projection map $\breve{\Lambda}^\sharp\rightarrow\breve{\Lambda}^\sharp/\breve{\Lambda}$.

Let $M^\sharp$ be the hermitian dual of the lattice $M\subset M(\bX)\otimes W_O(\kappa)[1/\pi_0]$, which agrees with the symplectic dual in the same space, see \S \ref{sec:framing}. Since $\breve{\Lambda}\subseteq M$, we have $M^\sharp\subseteq \breve{\Lambda}^\sharp$.
Therefore, we have
\begin{align*}
    M^\sharp={}&\{x\in M(\bX)\otimes W_O(\kappa)[1/\pi_0]\mid \langle x,y\rangle\subset W_O(\kappa)\quad\text{for all } y\in M\},\\
    ={}&\{x\in \breve{\Lambda}^\sharp\mid \pi\langle x,y\rangle\equiv 0\mod \pi_0\quad\text{for all } y\in M\},\\
    ={}&\{x\in \breve{\Lambda}^\sharp\mid \langle \bar{x},\bar{y}\rangle=0, \quad\text{for all } y\in M\},\\
    ={}&\{x\in\breve{\Lambda}^\sharp\mid \langle \bar{x},\bar{y}\rangle=0\quad\text{for all } y\in \Phi(U)\},\\
    ={}&\{x\in\breve{\Lambda^\sharp}\mid \bar{x}\in \Phi(U)^\perp\subset \Omega_{\Lambda,\kappa}\}=M_{\Phi(U)^\sharp}.
\end{align*}
Since $U\stackrel{h}{\subset} U^\sharp$, we have $M\stackrel{h}{\subset} M^\sharp$.
The relation \eqref{equ:Z-strata_points} implies that $M\cap \tau(M)\stackrel{\leq 1}{\subset} M$.
Next, we have
\begin{equation*}
    \Pi M\subset \Pi\breve{\Lambda}^\sharp\subseteq\breve{\Lambda}=\tau^{-1}\breve{\Lambda}\subseteq\tau^{-1}M\subseteq \tau^{-1}\breve{\Lambda}^\sharp=\breve{\Lambda}^\sharp\subseteq\Pi^{-1}\Lambda\subseteq \Pi^{-1}M.
\end{equation*}
Hence the lattice $M$ defines a point in $\cN_{n,\ep}^{[h]}(\kappa)$. We conclude that $M\in \cZ(\Lambda)(\kappa)$.
\end{proof}

\begin{theorem}\label{prop:ZcycleDLV}
    Let $\Lambda$ be a vertex lattice of type $t\ge h$ in $\bV$. Then the morphism $\Psi_{\cZ}$ defines an isomorphism $\cZ(\Lambda)\rightarrow S_{\Lambda}$.
\end{theorem}
\begin{proof}
Recall that $\cZ(\Lambda)$ is a projective scheme (see the discussion after Definition \ref{def:YZ-cycles}) and is normal by Theorem \ref{thm:LM-results}.
By Lemma \ref{lem:Z-strata_points}, we know that $\Psi_{\cZ}$ is a bijection for any perfect field $\kappa$ over $\F$. We will show in Proposition \ref{thm:Z-strata_points} that this property extends to any field $k$ over $\F$. Applying this to the function field of $S_{\Lambda}$ shows that $\cZ(\Lambda)$ is irreducible and $\Psi_{\cZ}$ is birational. Given this result, the normality of $S_{\Lambda}$, and the properness of $\Psi_{\cZ}$ (as it is a morphism between projective varieties), it follows from Zariski's main theorem that $\Psi_{\cZ}$ is an isomorphism.
\end{proof}

\begin{proposition}\label{prop:comp-inclusion}
The isomorphism $\Psi_{\cZ}$ is compatible with inclusions of vertex lattices. In other words, given two vertex lattices $\Lambda_1\subseteq \Lambda_2$ in $\cL_\cZ$, we have the following commutative diagram
\begin{equation*}
\begin{aligned}
\xymatrix{
\cZ(\Lambda_2)\ar@{^(->}[r]\ar[d]^-{\Psi_\cZ}&\cZ(\Lambda_1)\ar[d]^-{\Psi_\cZ}\\
S_{\Lambda_2}\ar@{^(->}[r]&S_{\Lambda_1},
}
\end{aligned}
\end{equation*}
here, the inclusion in the top row is defined in a natural way. The inclusion in the bottom row is given by taking the image of $\cV\subset \Omega_{\Lambda_2}$ under the morphism 
$$
\Omega_{\Lambda_2}= \Lambda_2^\sharp/\Lambda_2\otimes_{\F_q}\F\subseteq \Lambda_1^\sharp/\Lambda_2\otimes_{\F_q}\F\to \Lambda_1^\sharp/\Lambda_1\otimes_{\F_q}\F=\Omega_{\Lambda_1}
$$
Equivalently, $\cV$ corresponds to a vertex lattice $M$ satisfying
$$\breve{\Lambda}_2\subseteq M\subseteq M^\sharp\subseteq \breve{\Lambda}_2^\sharp,$$
and we define the map via $M\mapsto M/\breve{\Lambda}_1\subset \breve{\Lambda}_1^\sharp/\breve{\Lambda}_1$. Since the spaces in the diagram are varieties, these descriptions uniquely characterize the morphisms.
\end{proposition}
\begin{proof}
We can check this on $k$-valued points for any field $k$ over $\F$, where it is obvious from the definition of the map $\Psi_\cZ$, as we explained in the statement.
\end{proof}

To extend Lemma \ref{lem:Z-strata_points} to an arbitrary field $k$ over $\F$, we use $O_{F_0}$-displays. We refer the reader to \S \ref{sec:reviewdisplays} for the details of notations.
For any $R\in\Nilp_{O_{\breve{F}}}$ and $R$-point $(X,\iota,\lambda,\rho)\in \cZ(\Lambda)(R)$, let $\cP(X)=\cP=(P,Q,\bF,\dF)$ be the associated $O_{F_0}$-display  (see Theorem \ref{thm:ACZ-main}). 
Let $\cP_{\Lambda}=(P_{\Lambda},Q_{\Lambda},\bF,\dF)$ and $\cP_{\Lambda^\sharp}=(P_{\Lambda^{\sharp}},Q_{\Lambda^{\sharp}},\bF,\dF)$ denote the displays $\cP(X_{\Lambda^{\sharp}})$ and $\cP(X_{\Lambda})$, resp.
The isogeny $X\to X_{\Lambda^{\sharp}}$ induces an $O_{F_0}$-display morphism $\rho_{X,\Lambda^{\sharp}}:\cP(X)\to \cP_{\Lambda^{\sharp}}$.
Define
\begin{equation*}
\cU(X):=\ker\bigl[\Pi:P_{\Lambda^{\sharp}}\to P_{\Lambda^{\sharp}}\twoheadrightarrow P_{\Lambda^{\sharp}}/\rho_{X,\Lambda^{\sharp}}(Q)\bigr]
\end{equation*}
This defines a $W_O(R)$-submodule of $P_{\Lambda^{\sharp}}$.

When $R=\kappa$ is an algebraically closed field, by Dieudonn\'e theory, we have $P=M(X)$ and $\uV M(X)=Q\cong \rho_{X,\Lambda^{\sharp}}(Q)\subseteq P_{\Lambda^{\sharp}}=\Lambda^{\sharp}\otimes W_O(\kappa)$. Hence
\begin{align*}
\cU(X)={}&\bigl\{
    v\in \Lambda^{\sharp}\otimes W_O(\kappa)\mid \Pi(v)\in \uV M(X)
    \bigr\}\\
    ={}&\bigl\{
    v\in \Lambda^{\sharp}\otimes W_O(\kappa)[1/\pi] \mid v\in \Pi^{-1}\uV M(X)\cap \Lambda^{\sharp}\otimes W_O(\kappa)
    \bigr\}\\
    ={}&\Pi^{-1}VM(X)=\tau^{-1}M(X)\subseteq \Lambda^{\sharp}\otimes W_O(\kappa).
\end{align*}
Therefore, we can re-define 
\begin{equation}\label{equ:display-U(X)}
    U(X):=\cU(X)/\rho_{\Lambda,\Lambda^{\sharp}}(P_{\Lambda})\subset P_{\Lambda^{\sharp}}/\rho_{\Lambda,\Lambda^{\sharp}}(P_{\Lambda})=\Lambda^{\sharp}/\Lambda\otimes W_O(R)=\Lambda^{\sharp}/\Lambda\otimes  R.
\end{equation}
This construction agrees with our previous one, as verified by checking geometric points.

For any field $k$ over $\F$, recall from Proposition \ref{prop:ramified Witt vector} that $W_O(k)$ is an integral domain.
Consequently, for any $X\in \cZ(\Lambda)(k)$, the induced maps $P_{\Lambda}\to P\to P_{\Lambda^{\sharp}}$ are injective.

\begin{proposition}\label{thm:Z-strata_points}
Let $k$ be any field over $\F$.
The map $\Psi_{\cZ}$ defines a bijection between $\cZ(\Lambda)(k)$ and $S_{\Lambda}(k)$.
\end{proposition}
\begin{proof}
To construct the inverse of $\Psi_{\cZ}(k)$, consider a $k$-point $U\in S_\Lambda(k)$ defining an isotropic subspace $U\stackrel{\ttt-\tth}{\subset}\Omega_{\Lambda,k}$. 
Via the isomorphism $\Omega_{\Lambda,k}= \Lambda^\sharp/\Lambda\otimes k\cong\Lambda^\sharp/\Lambda\otimes W_O(k)=P_{\Lambda^{\sharp}}/P_{\Lambda}$, the subspace $U$ lifts to a chain of $W_O(k)$-modules
\begin{equation}\label{equ:cU-lattice}
    P_\Lambda\stackrel{\ttt-\tth}{\subseteq} \cU\subseteq P_{\Lambda^{\sharp}}
\end{equation}
Let $Q:=\Pi \cU\subset \Pi\cP_{\Lambda^{\sharp}}\subset Q_{\Lambda^\sharp}$ and let $P$ be the $W_O(k)$-submodule of $\cP_{\Lambda^{\sharp}}$ generated by $\dF(Q)$, where $\dF$ is the restriction of $\dF:Q_{\Lambda^\sharp}\to P_{\Lambda^\sharp}$. Similarly, $\bF$ is the restriction from $P_{\Lambda^\sharp}$.

We claim that $\cP:=(P,Q,\bF,\dF)$ defines an $O$-display over $k$. By the equivalence of categories in Theorem \ref{thm:cohen-equiv}, it suffices to prove the assertion after replacing all constructions involving $O$-displays with those for the $O$-Cohen ring $C_O(k)$, which is a complete discrete valuation ring.
In particular, both $P$ and $Q$ are finite projective modules over the $O$-Cohen ring $C_O(k)$. The existence of a normal decomposition then follows from the classification of finitely generated modules over a discrete valuation ring, together with the fact that $(\varpi)P\subseteq Q\subseteq P$.

By \eqref{equ:cU-lattice}, we have
\begin{equation*}
    \Pi P_\Lambda=Q_\Lambda\subset Q\subset \Pi P_{\Lambda^{\sharp}}=Q_{\Lambda^{\sharp}},
\end{equation*}
which yields natural inclusions $\cP_{\Lambda}\hookrightarrow\cP\hookrightarrow \cP_{\Lambda^{\sharp}}$. Since $\mcP_{\Lambda}$ and $\mcP_{\Lambda^{\sharp}}$ are bi-nilpotent displays, it is clear that $\mcP$ is also a bi-nilpotent display. Then we obtain a biformal strict $O_{F_0}$-module $\mathrm{BT}(\mcP)$. The definition of $S_\Lambda$ ensures that $\mathrm{BT}(\cP)$ defines a point in $\cZ(\Lambda)(k)$. Conversely, given a point in $\cZ(\Lambda)(k)$, we obtain a filtration $\cP_\Lambda\subset\cP\subset\cP_{\Lambda^{\sharp}}$. By reversing the above procedure, we obtain a point in $S_{\Lambda}(k)$. Hence the above construction is the inverse of $\Psi_{\cZ}(k)$.
\end{proof}

\begin{remark}\label{rmk:correct-VW}
It is well-known among experts that the original construction of the morphism from Bruhat--Tits strata to Deligne--Lusztig varieties by Vollaard and Wedhorn \cite{Vollaard-Wedhorn} is not an isomorphism, but rather differs by a Frobenius twist. This is evident from the display construction and the proof of Proposition \ref{thm:Z-strata_points}.

We first remind the reader that for ($O$-) Witt vectors, the Frobenius morphism $\sigma$ is no longer an isomorphism. In particular, when generalizing Dieudonn\'e theory from perfect fields to arbitrary fields or rings, the Verschiebung cannot be generalized directly as it is $\sigma^{-1}$-linear. However, its ``inverse'', which is $\dot{\uF}$ in $O$-display, can be generalized.

In fact, by \cite[Lem. 2.2]{ACZ}, the Verschiebung can only be constructed up to a twist:
\begin{equation*}
    V^{\sharp}:P\to W_O(R)\otimes_{\sigma,W_O(R)}P.
\end{equation*}
Therefore, while $Q\subset P_{\Lambda}$ uniquely determines $P\subset P_{\Lambda}$ (since $P\text{``=''}V^{-1}Q$), the converse fails. 
In other words, $Q\text{``=''}\uV P$ can be recovered from $P$ only up to a Frobenius twist, which fails to be an isomorphism over non-perfect fields $k$.
Thus, the construction in \cite{Vollaard-Wedhorn}, using the language of $O$-displays, given by 
$$\cZ(\Lambda)\to S_\Lambda\quad X\mapsto \cU(X)^{\mathrm{fake}}:=\mathrm{Im}[P\to P_{\Lambda^\sharp}]$$
is not an isomorphism but differs by a Frobenius twist.
\end{remark}
\begin{remark}
While the proof of Proposition \ref{thm:Z-strata_points} is using the frame of a $O$-Cohen subring, we stick with the $O$-display construction for our definitions. The $O$-Witt vector construction is more natural here, since the construction of Cohen subrings is only valid for fields, and depends on a choice of a $p$-basis. For example, it is not clear how to handle functoriality when using the Cohen frame.
\end{remark}

\subsection{$\cY$-strata stratification}
Let $\Lambda\subset \bV$ be any vertex lattice of type $t(\Lambda)\leq h$. We sketch the construction (similar to \S \ref{sec:Z-strata-strata}) of the isomorphism
\begin{equation*}
    \Psi_{\cY}:\cY(\Lambda^\sharp)\simeq R_{\Lambda^{\sharp}},
\end{equation*}
where $R_{\Lambda^{\sharp}}$ is defined in \eqref{equ:R-Lambda-defn}. 
For any $\F$-algebra $R$ and any $R$-point $(X,\iota,\lambda,\rho)\in{\cY}(\Lambda^\sharp)(R)$ in the $\cY$-stratum (see Definition \ref{def:YZ-cycles}), we have an isogeny $\rho_{\Lambda^\sharp,X^\vee}: X_{\Lambda^\sharp}\to X^\vee$. This implies that the quasi-isogeny $\rho_{X^\vee,\pi^{-1}\Lambda}:X^\vee\to X_\Lambda$ defined by $\rho_{X^\vee,\pi^{-1}\Lambda}\circ\rho_{\Lambda^\sharp,X^\vee}=\rho_{\Lambda^\sharp,\pi^{-1}\Lambda}$ is an isogeny.

Since $\Lambda$ is a vertex lattice, we have $\ker[\rho_{X^\vee,\pi^{-1}\Lambda}]\subset X^\vee[\pi]$. Therefore, there exists an isogeny $\widetilde{\rho}_{X^\vee,\pi^{-1}\Lambda}: X_{\pi^{-1}\Lambda}\to X^\vee$ such that $\widetilde{\rho}_{X^\vee,\pi^{-1}\Lambda}\circ\rho_{X^\vee,\pi^{-1}\Lambda}=\iota(\pi):X^\vee\to X^\vee$.

The proof of the following lemma is the same as that of Lemma \ref{BT_Z_isotropic-subspace} and we leave the details to the reader.
\begin{lemma}\label{BT_Y_isotropic-subspace}
For $X\in\cY(\Lambda^\sharp)(R)$, we have induced filtrations
\begin{equation*}
\begin{aligned}
\xymatrix @C=4pc{
D(X_{\Lambda^\sharp})\ar[r]^-{D(\rho_{\Lambda^\sharp,\pi^{-1}\Lambda})}&D(X_\Lambda)\ar[r]^-{D(\widetilde{\rho}_{X^\vee,\pi^{-1}\Lambda})}&D(X^\vee)\\
\Pi D(X_{\Lambda^\sharp})\ar@{^(->}[u]\ar[r]&\Pi D(X_{\Lambda})\ar@{^(->}[u]\ar[r]&\Fil(X^\vee).\ar@{^(->}[u]
}
\end{aligned}
\end{equation*}
The preimage
$D(\widetilde{\rho}_{X^\vee,\pi^{-1}\Lambda})^{-1}(\Fil(X^\vee))\subseteq D(X_{\pi^{-1}\Lambda,R})$ is a locally free direct summand that contains $\mathrm{Im}(D(\rho_{\Lambda^\sharp,\pi^{-1}\Lambda}))$. Moreover, the quotient
    \begin{equation}
        U(X):=M(\widetilde{\rho}_{X^\vee,\pi^{-1}\Lambda})^{-1}(\Fil(X^\vee))\slash\mathrm{Im}M(\rho_{\Lambda^\sharp,\pi^{-1}\Lambda})\subset \Omega_{\Lambda^\sharp,R}
    \end{equation}
    is a locally free isotropic direct summand of rank $\frac{1}{2}(h-t)$.\qed
\end{lemma}

For the remainder of our discussion, recall that $\Omega_{\Lambda^\sharp}=\Lambda/\pi\Lambda^\sharp\otimes\F$ when $t(\Lambda)\neq h$. The variety $R_{\Lambda^\sharp}$, defined in \eqref{equ:R-Lambda-defn}, has $k$-points for any field $k$ over $\F$ are given by
\begin{equation}
    R_{\Lambda^\sharp}^{}(k)  = \{\cV\subset\Omega_{\Lambda^\sharp,k}\mid\cV\text{ is isotropic},\,\text{and}\,
  \mathrm{dim}_k \cV=\tth-\ttt, \,\text{and}\,\operatorname{dim}_k(\cV+\Phi(\cV))\le \tth-\ttt+1 \}.
\end{equation}
When $t(\Lambda)=h$, the space $R_{\Lambda^\sharp}$ consists of the zero-dimensional subspace, hence the stratification in Theorem \ref{thm:orthogonal main} degenerates into $R_{\Lambda^\sharp}=X_{P_{[\tth',\tth']}}(\id)$, which is irreducible of dimension $0$.

By Lemma \ref{BT_Y_isotropic-subspace}, we define the map $\Psi_\cY:\cY(\Lambda^\sharp)\rightarrow R_{\Lambda^\sharp}$.

\begin{lemma}\label{lem:Y-strata_points}
Let $\kappa$ be any perfect field over $\F$.
The map $\Psi_{\cY}$ defines a bijection between $\cY(\Lambda^\sharp)(\kappa)$ and $R_{\Lambda^\sharp}(\kappa)$.
\end{lemma}
\begin{proof}
By Dieudonn\'e theory, a point $y\in \cY(\Lambda^\sharp)(\kappa)$ corresponds to a lattice $M$ satisfying 
$$\pi\breve{\Lambda}^\sharp\stackrel{\tth-\ttt}{\subset} \pi M^\sharp \stackrel{n-h}{\subset} M\stackrel{\tth-\ttt}{\subset}\breve{\Lambda},$$
see Proposition \ref{prop:BT_strata}. The wedge condition implies $M\cap \tau(M)\stackrel{\leq 1}{\subseteq} M$, which is equivalent to  
$$M^\sharp \cap \tau(M^\sharp) \stackrel{\leq 1}{\subseteq} M^\sharp$$ by Lemma \ref{lem: same ind}. Hence, we have
\begin{equation*}
    \Phi^{-1}(\pi M^\sharp /\pi \breve{\Lambda}^\sharp)\cap \pi M^\sharp/\pi \breve{\Lambda}^\sharp \stackrel{\leq 1}{\subseteq}\Phi^{-1}(\pi M^\sharp /\pi \breve{\Lambda}^\sharp).
\end{equation*}
Therefore $\Psi_{\cY}(y)\in R_{\Lambda^\sharp}(\kappa)$.

Conversely, let $U\in R_{\Lambda^\sharp}(\kappa)$ be any closed point.
Let  $\pi_\Lambda: \Lambda \to \Lambda/\pi \Lambda^\sharp$ denote the natural projection map. Then  the preimage $\pi_{\Lambda}^{-1}(U)$ defines a point in $\cY(\Lambda^\sharp)$. Indeed, $U\cap \Phi(U) \stackrel{\le 1}{\subset} U$ corresponds to $M^\sharp \cap \tau(M^\sharp)\stackrel{\le 1}{\subseteq}{M^\sharp}$ and the condition that $U$ is isotropic corresponds to the condition $\pi M^\sharp \subset M$ by computations similar to those in the proof of Lemma \ref{lem:Z-strata_points}. It is straightforward to check that the two maps are inverses to each other.
\end{proof}

By the same argument as in Theorem \ref{prop:ZcycleDLV}, we obtain the following.
\begin{theorem}\label{prop:YcycleDLV}
Let $\Lambda\in \cL_{\cY}$.   Then the morphism $\Psi_{\cY}$ defines an isomorphism $\cY(\Lambda^\sharp)\rightarrow R_{\Lambda^\sharp}$.\qed
\end{theorem}

\subsection{Intersection between $\cY$-strata and $\cZ$-strata}
We now discuss the intersection of $\cY$-strata and $\cZ$-strata. Let $\Lambda_1\in\cL_\cZ$ and $\Lambda_2\in\cL_\cY$ be two vertex lattices such that $\Lambda_1\subseteq\Lambda_2$. Recall from \eqref{equ:S Lambda12} that we defined the variety $S_{[\Lambda_1,\Lambda_2]}$ whose $k$-points for any field $k$ over $\F$ are given by
\begin{equation*}
S_{[\Lambda_1,\Lambda_2]}(k)=\{\cV\subset\Omega_{[\Lambda_1,\Lambda_2],k}\mid \dim_k\cV=\ttt_1-\tth\quad\text{and}\quad \dim_k(\cV\cap\Phi(\cV))\geq \ttt_1-\tth-1\}.
\end{equation*}
The main result is:
\begin{proposition}\label{prop: Z int Y}
Let $\Lambda_1\in \cL_{\cZ}$ and let $\Lambda_2\in \cL_{\cY}$ be vertex lattices satisfying $\Lambda_1\subseteq\Lambda_2$. 
The restriction of $\Psi_\cZ$ to the intersection $\cZ(\Lambda_1)\cap\cY(\Lambda_2^\sharp)$ defines an isomorphism:
\begin{equation*}
\Psi_{\cZ\cap \cY}:\cZ(\Lambda_1)\cap \cY(\Lambda_2^\sharp)\cong S_{[\Lambda_1,\Lambda_2]}
\end{equation*}
\end{proposition}
\begin{proof}
By Proposition \ref{prop:ZcycleDLV}, we have a natural isomorphism:
\begin{equation*}
    \Phi_{\cZ}:\cZ(\Lambda_1)\overset{\sim}{\longrightarrow} S_{\Lambda_1},\quad X\mapsto U(X).
\end{equation*}
The inclusion $\Lambda_1\subseteq\Lambda_2\subseteq\Lambda_2^\sharp\subseteq\Lambda_1^\sharp$ implies that the subspace $(0)\subseteq \Lambda_2/\Lambda_1\subseteq V_{\Lambda_1}$ is isotropic. Let $\kappa$ be a fixed algebraically closed field over $\F$. Let $W:=\breve{\Lambda}_2/\breve{\Lambda}_1\subset \Omega_{\Lambda_1,\kappa}$.
Via the isomorphism $\Psi_{\cZ}$, the intersection $\cZ(\Lambda_1)\cap \cY(\Lambda_2^\sharp)$
corresponds to the subvariety of $S_{\Lambda_1}$ that parameterizes all subspaces 
\begin{equation*}
    \{\cV\subset \Omega_{\Lambda_1,\kappa}\mid \cV\subseteq W \text{ is isotropic},\, \dim \cV=\ttt_1-\tth\quad\text{and}\quad \dim(\cV\cap\Phi(\cV))\geq \ttt_1-\tth-1\}.
\end{equation*}
Since $W$ is isotropic, any subspace $U\subset W$ is automatically isotropic.
Moreover, since $\Lambda_2\subset \bV$ is closed under Frobenius, this subvariety is isomorphic to the subvariety of the partial flag variety that parameterizes
\begin{equation*}
    \{\cV\subset \Omega_{\Lambda_1,\kappa}\mid \cV\subseteq W\quad\text{with}\quad \dim \cV=\ttt_1-\tth\quad\text{and}\quad \dim(\cV\cap\Phi(\cV))\geq \ttt_1-\tth-1\}.
\end{equation*}
By definition, this is isomorphic to $S_{[\Lambda_1,\Lambda_2]}$.
\end{proof}

\subsection{Proof of Proposition \ref{prop:KR-DLV}}\label{sec:proof-KR-DLV}
We recall the statement:
\begin{proposition*}
\begin{altenumerate}
\item By restriction, the isomorphism $\Psi_{\cZ}$ induces isomorphisms:
\begin{altenumerate2}
\item $\displaystyle\cZ(\Lambda_1)\setminus \cY \cong\coprod_{i=0}^{\tth-1}X_{P_{[i,\tth]}}(\sfw_{[i,\tth]})$;
\item $\displaystyle\cZ(\Lambda_1)\cap \cY \cong \Bigl(\coprod_{0\leq j<\tth< i\leq \ttt}X_{P_{ij}}(\sfw_{ij})\Bigr)\amalg
    \Bigl(\coprod_{0\leq j<\tth<i\leq \ttt}X_{P_{ij}}(\sfw'_{ij})\Bigr)
    \amalg X_{P_{\tth\tth}}(\id)$.
\end{altenumerate2}
\item Similarly, by restriction, the isomorphism $\Psi_\cY$ induces isomorphisms:
\begin{altenumerate2}
\item $\displaystyle
\cY(\Lambda_2^\sharp)\setminus \cZ \cong\coprod_{i=0}^{\tth'-1}X_{P_{i,\tth'}}(\sfw_{i,\tth'})$;
\item  $\displaystyle\cY(\Lambda_2^\sharp)\cap \cZ \cong\left\{\begin{array}{ll} \emptyset & \text{ if $h=n$,}\\
    \displaystyle\Bigl(\coprod_{0\leq j<\tth'< i\leq \ttt'}X_{P_{ij}}(\sfw_{ij})\Bigr)\amalg
    \Bigl(\coprod_{\substack{0\leq j<\tth'<i\leq \ttt',\\ \delta\in\{\pm\}}}X_{P_{ij}}(\sfw^{\prime,\delta}_{ij})\Bigr)
    \amalg X_{P_{\tth'\tth'}}(\id) & \text{ if $h=n-2$,}\\
     \displaystyle\Bigl(\coprod_{0\leq j<\tth'< i\leq \ttt'}X_{P_{ij}}(\sfw_{ij})\Bigr)\amalg
    \Bigl(\coprod_{0\leq j<\tth'<i\leq \ttt'}X_{P_{ij}}(\sfw'_{ij})\Bigr)
    \amalg X_{P_{\tth'\tth'}}(\id) & \text{otherwise}.\end{array}\right.$
\end{altenumerate2}    
\end{altenumerate}    
\end{proposition*}
\begin{proof}[Proof of Proposition \ref{prop:KR-DLV}]
Let $\kappa$ be a perfect field over $\F$ and let $X \in \cZ(\Lambda_1)(\kappa)$. Denote by $M = M(X)$ its associated Dieudonn\'e module.

Via the map $\Psi_\cZ$, the sum $M + \tau(M)$ induces a subspace $U(X) + \Phi(U(X))$ in $\Omega_{\Lambda} = \breve{\Lambda}^\sharp/\breve{\Lambda}$.
This subspace is isotropic if and only if $\Phi(U(X))\subseteq U(X)^\perp$. Equivalently, this occurs if and only if $\tau(M)\subseteq M^\sharp$, which in turn holds precisely when $X \in \cY(\kappa)$.
The claimed decomposition then follows from the moduli description, see also Remark \ref{rmk:first decomp}, particularly \eqref{equ:symplectic-first-step}. Part $(2)$ follows by the same argument as part $(1)$.
\end{proof}

\begin{remark}\label{rmk:relation-to-KR}
Recall from Remark \ref{equ:symplectic-first-step} that we have the following first-step decomposition:
\begin{equation*}
    S_\Lambda=X_{P_{\tth\tth}}(\sfw_{\Lambda})\amalg X_{P_{\tth\tth}}(\sfw'_{\Lambda}) \amalg X_{P_{\tth\tth}}(\id),
\end{equation*}
where we can describe these DL varieties as follows:
\begin{align*}
X_{P_{\tth\tth}}(\sfw_{\Lambda})={}&\{
\cV\in S_\Lambda\mid \cV\neq \Phi(\cV)\text{ and }\cV+\Phi(\cV)\text{ is not isotropic}\};\\
X_{P_{\tth\tth}}(\sfw'_{\Lambda})={}&\{
\cV\in S_\Lambda\mid \cV\neq \Phi(\cV)\text{ and }\cV+\Phi(\cV)\text{ is isotropic}
\};\\
X_{P_{\tth\tth}}(\id)={}&\{
\cV\in S_\Lambda\mid \cV= \Phi(\cV)
\}.
\end{align*}
As established in the proof of Proposition \ref{prop:KR-DLV}, for any  $M(X)\in \cZ(\Lambda_1)$ with corresponding isotropic subspace $U(X)\in S_\Lambda$, we have $\tau(M)\subseteq M^\sharp$ if and only if $U(X)+\Phi(U(X))$ is a totally isotropic subspace.
Therefore, via the isomorphism $\Psi_\cZ$, we have:
\begin{equation*}
    \cZ(\Lambda_1)\setminus \cY\cong X_{P_{\tth\tth}}(w_\Lambda),\quad\text{and}\quad
    \cZ(\Lambda_1)\cap \cY\cong X_{P_{\tth\tth}}(w_{\Lambda}')\amalg X_{P_{\tth\tth}}(\id).
\end{equation*}
Similar results exist for $\cY$-strata.
\end{remark}

\begin{remark}\label{rmk:Yu-case}
We remark that when $h=n-2$, part (2) of Proposition \ref{prop:KR-DLV} agrees with the computation in the local model. More precisely, the special fiber of the local model $\cM^{[n-2]}_{n}$ (see Definition \ref{def:local-model} for our notation) consists of three irreducible components:
\begin{equation*}
    \ov{\cM}_{n}^{[n-2]}=\cY\cup \cZ^+\cup \cZ^{-}.
\end{equation*}
See \cite[\S 1.5]{Yu2019} for the relevant computations. Following the notation in loc.cit., $\cY$ is defined as the vanishing locus of  $x_1,x_2$, while $\cZ^+$ (resp. $\cZ^-$) is defined as the vanishing locus of the first (resp. second) row of $Y$, where the matrix $Y$ is defined in p.22 of loc. cit. 
In particular, the intersection $\cZ^+\cap\cZ^-$ is the vanishing locus of the entire matrix $Y$, which represents the worst point.

The closure $\overline{X_{P_{\ttt'0}}(\sfw^{\prime,\pm}_{\ttt'0})}$     can be regarded as $\cY(\Lambda_2^\sharp)\cap \cZ^\pm$ (where we use $\cZ^\pm$ to denote both the KR strata and their corresponding components in the local model). Moreover, we have: $\overline{X_{P_{\ttt'0}}(\sfw^{\prime,+}_{\ttt'0})}\cap \overline{X_{P_{\ttt'0}}(\sfw^{\prime,-}_{\ttt'0})}=X_{P_{\tth'\tth'}}(\id)$, 
which aligns with the fact that $\cZ^+\cap\cZ^-$ is the worst point.

When $\mbV$ is split, let $\Lambda\subset \mbV$ be a vertex lattice of type $n-2$, which corresponds to a worst point in the RZ space. 
The quadratic space $\Lambda^{\sharp}$ is of dimension $2$, and consists of exactly two isotropic lines. This corresponds to two vertex lattices of type $n$, meaning that there are exactly two $\mcZ$-strata containing a given worst point. This matches of the decomposition of $\mcZ^{\pm}$ that we discussed above.
\end{remark}

\appendix
\section{Ramified Cohen frame}
In this appendix, we construct a Cohen subring within the $O$-Witt vectors, and show that it forms an $O$-frame. Throughout this section, we allow $p=2$.

\subsection{Relative theory of displays}\label{sec:reviewdisplays}
In this section, we review the relative theory of displays. For a comprehensive treatment, we refer the reader to \cites{ACZ,Mihatsch_2022,kudla2023padicuniformizationunitaryshimura,LMZ}.
Let $K/\Q_p$ be a finite extension of $p$-adic fields with ring of integers $O$. Let $\varpi\in O$ be a fixed uniformizer. We assume that the residue field is finite of order $q$.

\subsubsection{$O$-Witt vectors}\label{sec:OWittvector}
We recall basic facts about $O$-Witt vectors.
For any $n\geq 0$, we define the \emph{$n$-th Witt polynomial}
\begin{equation*}
\bfw_n:=X_0^{q^n}+\varpi X_1^{q^{n-1}}+\cdots+\varpi^n X_n
\in O[X_0,\cdots,X_n]
\end{equation*}
Consider the functor
\begin{equation*}
\mcF:O\mathrm{-algebras}\to O\mathrm{-algebras},\quad R\mapsto R^{\mbN},
\end{equation*}
we write $[r_i]_{i\geq 0}$ for an element in $\mcF(R)$, where $r_i\in R$ for any $i\geq 0$.
By \cite[Lemma 1.2.1]{Fargues--Fontaine}, there exists a unique factorization
\begin{equation*}
\begin{aligned}
\xymatrix{
O\mathrm{-algebras}\ar[rr]^-{\mcF}\ar[dr]_-{W_O(-)}&&\mathrm{Sets}\\
&O\mathrm{-algebras}\ar[ur]&
}
\end{aligned}
\end{equation*}
such that the natural transformation 
\begin{equation*}
W_O(R):W_O(R)\to R^{\mbN},\quad
[r_i]_{i\geq 0}\mapsto (\bfw_n(r_0,\cdots,r_n))_{n\geq 0}
\end{equation*}
is an $O$-algebra morphism. We call $W_O(-)$ the $O$-Witt vectors (which depends on the choice of the uniformizer $\varpi$).

For any $O$-algebra $R$, the $O$-Witt vector $W_O(R)$ is equipped with a Frobenius endomorphism $\sigma$ and a Verschiebung endomorphism, such that $\sigma\circ V=V\circ \sigma=\varpi$. 
We collect some properties of $O$-Witt vectors that will be used in the remaining part of the section.
\begin{proposition}\label{prop:ramified Witt vector}
\begin{altenumerate}
\item For any $n\geq 1$, $V^nW_O(R)$ is an ideal of $W_O(R)$ that is independent of the choice of uniformizer $\varpi$.
The ring of $O$-Witt vectors is separated and complete with respect to the $VW_O(R)$-adic topology: let $W_O(R)_n:=W_O(R)/V^nW_O(R)$ be the truncated $O$-Witt vectors, we have an isomorphism
$$
W_O(R)\overset{\sim}{\to} \varprojlim_n W_O(R)_n.
$$

\item For any $n\geq 1$, we have the following inclusions of ideals in $W_O(R)$:
$$\varpi^nW_O(R)\subseteq (VW_O(R))^n\subseteq \varpi^{n-1}W_O(R).$$
\item Let $R=\kappa$ be a perfect field over $\mbF_q$. The $O$-Witt vector $W_O(\kappa)$ is a complete discrete valuation ring with maximal ideal $(\varpi)$.

\item Let $R=k$ be a field over $\mbF_q$. The $O$-Witt vector $W_O(k)$ is a local integral domain with maximal ideal $VW_O(k)$.
\end{altenumerate}
\end{proposition}
\begin{proof}
Parts (1) and (3) follow from \cite[\S 1.2]{Fargues--Fontaine}. For part (2), the left inclusion is clear from $\varpi W_O(R)\subseteq VW_O(R)$.
For the right-hand side inclusion, recall that for any $x,y\in W_O(R)$, we have $V(F(x)y)=xV(y)\in W_O(R)$, see for example, \cite[\S 1.2]{Fargues--Fontaine}. Therefore, we have
$$
V(x)V(y)=V(FV(x)\cdot y)=\varpi V(x\cdot y).
$$
This implies that $(VW_O(R))^2\subset \varpi VW_O(R)$. Applying this repeatedly, we obtain the right inclusion.

For part (4), the natural map $k\to\ov{k}$ induces an embedding $W_O(\kappa)\to W_O(\ov{k})$, making $W_O(\kappa)$ an integral domain as a subring of $W_O(\ov{k})$. It remains to show that it is a local ring. 
First, for any $x\in VW_O(k)$, the element $1-x$ is a unit with inverse
$$
(1-x)^{-1}=1+x+x^2+\cdots
$$
This is well-defined by part (1) of the Proposition. 
Next, for any $u\in W_O(k)\setminus VW_O(k)$, we can find $v\in W_O(k)$ and $x\in VW_O(k)$ such that $uv=1-x$. Hence $u$ is a unit, and therefore $W_O(k)$ is a local ring with maximal ideal $VW_O(k)$.
\end{proof}

\subsubsection{$O$-frame and $O$-windows}
We recall the theory of $O$-frame and $O$-windows in \cite{ACZ}, which generalizes the work of Lau \cite{Lau-frmae}.

\begin{definition}
An $O$-frame is a quintuple $\mcF=(S,I,R,\sigma,\dot{\sigma})$, where $S$ is an $O$-algebra, $I\subseteq S$ is an ideal, and $R=S/I$, together with an $O$-algebra morphism $\sigma:S\to S$ and a $\sigma$-linear morphism of $S$-modules $\dot{\sigma}:I\to S$, which satisfy the following conditions:
\begin{enumerate}
    \item $I+\varpi S\subseteq \mathrm{Rad}(S)$,
    \item $\sigma(a)\equiv a^q\mod \varpi S$ for all $a\in S$, and
    \item $\dot{\sigma}(I)$ generates $S$ as an $S$-module.
\end{enumerate}
\end{definition}
Let $\mcF_i=(S_i,I_i,R_i,\sigma_i,\dot{\sigma}_i)$ ($i\in \{1,2\}$) be two $O$-frames. A (strict) morphism of $O$-frames $\alpha:\mcF_1\to \mcF_2$ is an $O$-algebra morphism $\alpha:S_1\to S_2$ compatible with the additional structures.

\begin{example}
A special example is the Witt $O$-frame:
\begin{equation*}
\mcW_O(R)=(W_O(R),I_{O,R},R,\sigma,\dot{\sigma}),
\end{equation*}
where $I(R)=I_O(R):=\ker(W_O(R)\to R)$ is the kernel of the projection $W_O(R)\to R$, $\sigma$ is the Frobenius and $\dot{\sigma}=V^{-1}:I(R)\to W_O(R)$ is the inverse of the Verschiebung. Since $V$ is injective with image $I(R)$, this is a well-defined $\sigma$-linear surjection.
See \cite[\S 1.2.1]{ACZ} or \cite[\S 4.1]{LMZ} for a precise definition.
\end{example}

\begin{definition}
Let $\mcF=(S,I,R,\sigma,\dot{\sigma})$ be an $O$-frame. An \emph{$O$-window over $\mcF$}, or an \emph{$\mcF$-window}, is a quadruple $\mcP=(P,Q,\bF,\dot{\bF})$, where 
\begin{itemize}
\item $P$ is a finitely generated projective $S$-module;
\item $Q\subseteq P$ is a submodule with $IP\subseteq Q$ and such that $P/Q$ is a projective $R$-module;
\item $\bF:P\to P$ and $\dot{\bF}:Q\to P$ are $\sigma$-linear maps of $S$-modules.
\end{itemize}
We require that it satisfy the following conditions:
\begin{enumerate}
\item There is a normal decomposition $P=T\oplus L$ with $Q=IT\oplus L$,
\item $\dot{\bF}(ax)=\dot{\sigma}(a)\bF(x)$ for all $a\in I$ and $x\in P$.
\item $\dot{\bF}(Q)$ generates $P$ as an $S$-module.
\end{enumerate}
We refer the reader to \cite[Def. 3.3]{ACZ} for the definition of nilpotent $\mcF$-windows.
\end{definition}

\begin{example}\label{exm:recover dieudonne theory}
\begin{altenumerate}
\item For any $O$-algebra $R$, we call a $\mcW_O(R)$-windows an \emph{$O$-display} over $R$. 
\item When $R$ is a perfect algebra, the Frobenius map $\sigma$ is an isomorphism, we conclude from \cite[Lem. 2.2]{ACZ} that nilpotent $O$-display is equivalent to the relative Dieudonn\'e module $M(X):=P(X)$ over $W_O(R)$, by passing along the isomorphism $M(X)\cong \sigma^*M(X)$.  In this case, $I(R)\subset W_O(R)$ is generated by $\varpi\in O\subseteq W_O(R)$, and $Q(X)$ is identical to $VM(X)$, hence the Hodge filtration $\Fil(X)=Q(X)/I(R)P(X)\cong VM(X)/\varpi M(X)$.
\end{altenumerate}
\end{example}

\begin{definition}\label{def:pd-structure}
Let $R$ be an $\varpi$-adic $O$-algebra and let $I\subseteq R$ be an ideal. An \emph{$O$-pd-structure} on $I$ is a map $\gamma:I\to I$, such that
\begin{itemize}
    \item $\varpi\gamma(x)=x^q$,
    \item $\gamma(r\cdot x)=r^q\cdot \gamma(x)$, and
    \item $\gamma(x+y)=\gamma(x)+\gamma(y)+\sum_{0<i<q}(\binom{q}{i}/\varpi)\cdot x^i\cdot y^{q-i}$
\end{itemize}
hold for all $r\in R$ and $x,y\in I$. An $O$-pd-ring is a triple $(R,I,\gamma)$, where $R$ is an $O$-algebra, $I\subseteq R$ is an ideal, and $\gamma$ is an $O$-pd-structure. We call $(R,I,\gamma)$ an $O$-pd-ring.
\end{definition}
This recovers the usual definition of pd-structure, see \cite[Remarque B.5.2]{Fargues08}. The structure ring $(O,(\varpi),\gamma)$ with $\gamma(x):=\frac{x^{q}}{\varpi}$ forms an $O$-pd-structure. It is straightforward to verify the following:
\begin{proposition}\label{prop:O-pd structure extension}
Let $(A,I,\gamma)$ be an $O$-pd-ring. Let $A\to B$ be a ring map. If $\gamma$ extends to $(B,IB,\ov{\gamma})$ then the extension is unique. When $I$ is principal, the extension exists.
\end{proposition}
\begin{proof}
This is a direct generalization of \cite[\href{https://stacks.math.columbia.edu/tag/07H1}{07H1}]{stacks-project}.
\end{proof}

Let $\mcP$ be a nilpotent $O$-display over $R$, the base change of the frame defines a sheaf over the category of $O$-pd-thickenings, which is called the \emph{Witt crystal} and denoted by $\mcK_\mcP$. We define the \emph{Dieudonn\'e crystal} by 
$$
\mbD_P(S)=\mcK_\mcP(S)/I_S\mcK_\mcP(S).
$$

We refer the reader to \S \ref{sec:strict-module} for the definition of strict $O$-modules.
The following result is proved by Zink and Lau, and is generalized by Ahsendorf--Chen--Zink.
\begin{theorem}[\protect{\cite{ACZ}}]\label{thm:ACZ-main}
Let $S$ be a formal scheme over $\Spf O$.
There is an equivalence of categories  over $S$:
\begin{equation}\label{eq:equiv_display}
\mathrm{BT}:\left\{\text{nilpotent $O$-displays}\right\} \overset{\sim}{\longrightarrow} \left\{\text{\begin{varwidth}{\textwidth} \centering strict formal $O$-modules \end{varwidth}}\right\},
\end{equation}
which is compatible with the Faltings dual (in the biformal case) and base change. In particular, this functor induces an $O$-crystal $\bD(X)$ valued in the category of $O$-pd-thickenings. We denote the quasi-inverse of this equivalence by 
$$X\mapsto \cP(X)=(P(X),Q(X),\bF(X),\dF(X)).$$ 
In particular, the relative Grothendieck--Messing theorem holds for strict formal $O$-modules, i.e., deformations of $X$ along $O$-pd-thickenings are in canonical bijection with liftings of the Hodge filtration. See \cite[Lem. 3.12]{ACZ}, or \cite[Prop. B.8.2]{Fargues08}.\qed
\end{theorem}

\subsection{Cohen subring of ring of $O$-Witt vectors}
We keep the notation from the previous subsection. 
We construct the Cohen subring for $O$-Witt vectors ring, generalizing the construction in \cite{Schneider-note}. Since the original text is written in German, we provide full details here.

Let $k$ be a field over $\F_q$, where $q=p^r$ for some $r>0$.
The subset $k^{p^n}$ is a subfield of $k$ for any $n\geq 1$. Recall that a family of elements $(x_i)_{i\in I}$ in $k$ is called a \emph{$p$-basis} if the following map is bijective:
$$
k^{p}[\{X_i\}]/(X_i^p-x_i^p)\to k,
\quad 
X_i\to x_i.
$$
The following lemma is standard, see for example, \cite[\href{https://stacks.math.columbia.edu/tag/07P0}{07P0}]{stacks-project}. 
\begin{lemma}\label{lem:p-basis}
For every $p$-basis $(x_i)_{i\in I}$ of $k$, the following holds:
\begin{altenumerate}
\item $k=k^{p^n}(x_i)_{i\in I}$ for all $n\geq 1$.
\item The elements $\prod_{i\in I}x_i^{\mu_i}$ with $0\leq \mu_i< p^n$ and $\mu_i\neq 0$ for at most finitely many $i$, form a basis of $k$ as a $k^{p^n}$-vector space.\qed
\end{altenumerate}
\end{lemma}

\begin{definition}
A subring $C\subseteq W_O(k)$ is called an \emph{$O$-Cohen subring} if:
\begin{itemize}
    \item $C$ is a complete discrete valuation ring with maximal ideal generated by $\varpi$.
    \item $W_O(k)=VW_O(k)+C$.
\end{itemize}
In particular, since 
$
C/(VW_O(k)\cap C)\cong W_O(k)/VW_O(k)\cong k,
$
we have $VW_O(k)\cap C=\varpi C$, and $k$ is the residue field of $C$.
\end{definition}
\begin{theorem}\label{thm:ramified cohen subring}
Let $(a_i)_{i\in I}$ be a family of elements in $W_O(k)$ such that the family $(x_i)_{i\in I}$ defined by $x_i:=\bfw_0(a_i)$ is a $p$-basis of $k$. Then there exists a unique $O$-Cohen subring $C\subseteq W_O(k)$ containing all $a_i$.
\end{theorem}
In the remaining part of the subsection, we prove Theorem \ref{thm:ramified cohen subring}.
We adopt the following notation:
$$
A:=W_O(k), \mfm:=VA, \pr:=\bfw_0, \quad\text{and}\quad S:=\{a_i:i\in I\}\subset A.
$$

Fix an integer $m\geq 1$. For every $n\geq m-1$, let
\begin{equation}\label{def:cnm}
C_{n,m}:=\text{the subring of $A$ generated by }S\cup \bfw_n(W_O(A))\cup \mfm^m.
\end{equation}

\begin{lemma}\label{lem:count order}
Let $R$ be an $O$-algebra such that $p$ is not a zero divisor. Let $n\geq 0$ be an integer. For any $a,b\in R$,
$$
\text{if }a\equiv b\mod \varpi,\text{ then }a^{q^n}\equiv b^{q^n}\mod \varpi^{1+n}.
$$
\end{lemma}
\begin{proof}
Write $a=b+\varpi c$ for some $c\in R$. Then we have
$$
a^{q^n}-b^{q^n}=\sum_{i=1}^{q^n-1}(\binom{q^n}{i}\varpi^{i}) b^{q^n-i}c^i.
$$
Since $v_{\varpi}(q^n)=nv_{\varpi}(q)\geq n$, the assertion follows from Kummer's theorem.
\end{proof}

\begin{lemma}\label{lem:construct Cnm}
$C_{n,m}$ is the smallest subring of $A$ that satisfies $C_{n,m}+\mfm=A$ that contains $S\cup \mfm^m$.
In particular, $C_{n,m}$ is unique and is independent of the choice of $n$. We denote it by $C_m$.
\end{lemma}
\begin{proof}
Since 
$$
\bfw_n(W_O(A))=\{a_0^{q^n}+\varpi a_1^{q^{n-1}}+\cdots+\varpi^na_n:
a_0,\cdots,a_n\in A
\},
$$
and $\varpi A\subset \mfm$, it follows that $\pr(\bfw_n(W_O(A)))=k^{q^n}$ and  $\pr(C_{n,m})=k^{q^n}(\pr (S))$. By Lemma \ref{lem:p-basis}, we have $\pr(C_{n,m})=k$, which implies $C_{n,m}+\mfm=A$.

Now let $A'\subseteq A$ be a subring such that $A'+\mfm=A$ and $S\cup \mfm^m\subseteq A'$. We show that $\bfw_n(W_O(A))\subseteq A'$.
For any $a_0,\cdots,a_n\in A$, since $A'+\mfm=A$, there exist $a_0',\cdots,a_n'\in A'$ such that $a_i\equiv a_i'\mod \mfm$ for all $0\leq i\leq n$. Since $n\geq m-1$, we see that $\bfw_n(a_0,\cdots,a_n)\equiv \bfw_n(a_0',\cdots,a_n') \mod \mfm^m$ by 
Proposition \ref{prop:ramified Witt vector}(2) and Lemma \ref{lem:count order}. We conclude that $\bfw_n(a_0,\cdots,a_n)\in A'$.
\end{proof}

\begin{lemma}\label{lem:C_m maximal ideal}
$C_m\cap \mfm=\varpi C_m+\mfm^m$.
\end{lemma}
\begin{proof}
Since $\varpi\in\mfm$, it follows immediately that $\varpi C_m+\mfm^m\subseteq C_m\cap \mfm$. To show the opposite inclusion, let $\Lambda_m$ be the set of all tuples $\mu=(\mu_i)_{i\in I}$ of integers $0\leq \mu_i< q^m$ such that $\mu_i= 0$ for all but finitely many $i\in I$. For any $\mu\in\Lambda_m$, define 
$$
Z_\mu:=\prod_{i\in I}a_i^{\mu_i}.
$$
Since $S^{q^m}=\{
\bfw_m(a_i,0,\cdots,0):i\in I
\}\subseteq \bfw_m(W_O(A))$, the subring $C_m=C_{m,m}$ is generated as a module over the subring $\bfw_m(W_O(A))+\mfm^m$ by the elements $Z_\mu$. Using the identity 
$$
\bfw_m(a_0,\cdots,a_m)=a_0^{q^m}+\varpi \bfw_{m-1}(a_1,\cdots,a_m),
$$
we see that
$$
\bfw_m(W_O(A))\subseteq A^{q^m}+\varpi C_{m-1,m}=A^{q^m}+\varpi C_m.
$$
Therefore, every element $c\in C_m$ can be written in the form 
$$
c=\sum_{\mu\in \Lambda_m}c_\mu^{q^m}Z_\mu+\varpi c'+c'',\quad\text{with}\quad c_\mu\in A,c'\in C_m,c''\in\mfm^m.
$$
Now assume $c\in C_m\cap \mfm$. Then 
$$
0=\pr(c)=\sum_{\mu\in\Lambda_m}\pr(c_\mu)^{q^m}\pr(Z_\mu).
$$
By Lemma \ref{lem:p-basis}, the elements $\pr(Z_\mu)$ form a $k^{q^m}$-basis of $k$. Hence $\pr(c_\mu)=0$, i.e. $c_\mu\in \mfm$ and therefore $c_\mu^m\in\mfm^m$. We conclude that $c\in \varpi C_m+\mfm^m$.
\end{proof}

By the minimality of $C_m$, we have 
$$
C_m=C_{m+1}+\mfm^m\quad\text{for all }m\geq 1.
$$
We define
$$
C:=\bigcap_{m\geq 1}C_m.
$$
It is clear that $S\subseteq C$. 
We have the inclusion of the inverse system $(C_m/\mfm^m)_m\subset (A/\mfm^m)_m$, it induces a commutative diagram:
\begin{equation*}
\begin{aligned}
\xymatrix{
A\ar[r]^-{\sim}&\varprojlim_m A/\mfm^m\\
C\ar[r]\ar@{^(->}[u]&\varprojlim_{m} C_m/\mfm^m\ar@{^(->}[u]
}
\end{aligned}
\end{equation*}
The top horizontal map is an isomorphism by definition, see for example, \cite[Definition 1.2.2]{Fargues--Fontaine}.
The bottom horizontal map is induced by the morphism of inverse systems $(C_m)_m\to (C_m/\mfm^m)_m$, and is therefore surjective.
Injectivity follows from the commutative diagram.
Therefore, the bottom horizontal map is also an isomorphism, which induces surjections
$
C\to C_m/\mfm^m.
$
In particular, 
\begin{equation}\label{eq:C/C cap m}
C/C\cap \mfm\cong C_1/\mfm\cong A/\mfm\cong k,    
\end{equation}
where the second isomorphism follows by definition. Thus $C\cap \mfm$ is a maximal ideal in $C$. 

We claim that $C$ is a local ring.
For any $a\in C\setminus C\cap \mfm$, by \eqref{eq:C/C cap m}, we can find $b\in C$ such that $ab=1-c$ for some $c\in \mfm$. The latter is invertible, with inverse 
$$
(1-c)^{-1}=1+c+c^2+\cdots,
$$
this is well-defined since $\sum_{i=0}^n c^i\in C\subset C_m$. 

Next, we show that $C\cap\mfm$ is generated by $\varpi\in C$. 
\begin{lemma}
\begin{altenumerate}
\item For any $m\geq 1$, we have
    $\bigcap_{j\geq m}(\varpi C_m+\mfm^j)=\varpi C_m$.
\item We have $C\cap\mfm=\varpi C$.
\end{altenumerate}
\end{lemma}
\begin{proof}
\begin{altenumerate}
\item 
First, it is clear that 
$$
\bigcap_{j\geq m}(\varpi C_m+\mfm^j)\subseteq \bigcap_{j\geq m}(\varpi A+\mfm^j)=\varpi A.
$$
For any $c\in A$ such that $\varpi c\in \bigcap_{j\geq m}(\varpi C_m+\mfm^j)$, we can write $\varpi c=\varpi c_j+m_j$ with $c_j\in C_m$ and $m_j\in\mfm^j$. Hence $\varpi(c-c_{m+2})=m_{m+2}\in \mfm^{m+2}\subset \varpi^{m+1}A$, where the last inclusion follows from Proposition \ref{prop:ramified Witt vector}(2).
Since $A$ is an integral domain, we conclude that $c-c_{m+2}\in \varpi^mA\subseteq\mfm^m$. Hence $c\in C_m+\mfm^m=C_m$. Hence we have 
$$
\varpi C_m\subseteq \bigcap_{j\geq m}(\varpi C_m+\mfm^j)=\varpi A\cap \bigcap_{j\geq m}(\varpi C_m+\mfm^j)\subseteq \varpi C_m.
$$
This proves the assertion.
\item By Lemma \ref{lem:C_m maximal ideal}, we have
$$
C\cap \mfm=\bigcap_m(C_m\cap \mfm)=\bigcap_m (\varpi C_m+\mfm^m)=\bigcap_{m\geq 1}\bigcap_{j\geq m}(\varpi C_m+\mfm^j).
$$
Using part (1), we see that $C\cap\mfm=\bigcap_{m\geq 1}\varpi C_m=\varpi (\bigcap_{m\geq 1} C_m)=\varpi C$. Here we use the fact that $A$ is an integral domain.\qedhere
\end{altenumerate}
\end{proof}

The following commutative algebra result is standard:
\begin{lemma}\label{lem:almost dvr}
Let $R$ be a ring with exactly one maximal ideal $\mfm$, which is a principal ideal $\mfm=(\varpi)$ and satisfies $\bigcap_{i\geq 0}\mfm^i=(0)$. Then every ideal is of the form $(\varpi^k)$ for some $k\geq 0$.\qed
\end{lemma}

\begin{corollary}
$C$ is a complete discrete valuation ring.
\end{corollary}
\begin{proof}
Since $C$ is a subring of $A$, it is an integral domain with $p\neq 0$. Furthermore, we have $\bigcap_{m\geq 1}\varpi^m C\subseteq \bigcap_{m\geq 1}\varpi^mA=\{0\}$. Hence $C$ is a discrete valuation ring by Lemma \ref{lem:almost dvr}. It remains to show that $C$ is complete, we have
$$
C\overset{\sim}{\to}\varprojlim_m C/C\cap \mfm^m\overset{\sim}{\to}\varprojlim_m C_m/\mfm^m.
$$
Since $C$ is a DVR, $C\cap \mfm^m=\varpi^{j(m)} C$ for some $j(m)\geq 0$. This shows that $C$ is complete.
\end{proof}

It remains to show the uniqueness. Let $C'$ be another $O$-Cohen subring with $S\subseteq C'$. Recall that we have $C'+\mfm=A$ and $C'\cap\mfm=\varpi C'$ is the maximal ideal of $C'$. By Lemma \ref{lem:construct Cnm}, we have $C'+\mfm^m\supset C_m$ for all $m\geq 1$. 
We have the following commutative diagram:
\begin{equation*}
\begin{aligned}
\xymatrix{
C'\ar[r]^-{\sim}&\varprojlim_{m}C'/C'\cap\mfm^m\ar[d]^-{\sim}\\
&\varprojlim_m C'+\mfm^m/\mfm^m\\
C\ar[r]^-{\sim}&\varprojlim_m C_m/\mfm^m\ar@{^(->}[u].
}
\end{aligned}
\end{equation*}
By composing all isomorphisms, we get $C\hookrightarrow C'$. Since both are discrete valuation rings with the same uniformizer $\varpi$ and residue field, this is an isomorphism by \cite[\href{https://stacks.math.columbia.edu/tag/09E5}{09E5}]{stacks-project}. This completes the proof of Theorem \ref{thm:ramified cohen subring}.\qed

\subsection{$O$-Cohen frame}\label{sec:ramified cohen frame}
Let $k$ be a field over $\mbF_q$ and let $\{x_i\}_{i\in I}$ be a $p$-basis of $k$. Let $\{a_i\}_{i\in I}=\{\iota(a_i)\}_{i\in I}$ be the \emph{Teichm\"uller lift} of $\{x_i\}_{i\in I}$. We will keep this assumption throughout the remaining part of the section. In particular, we have a unique $O$-Cohen subring $C_O(k)\subset W_O(k)$ containing $S$  (which still depends on the choice of a $p$-basis).

\begin{lemma}\label{lem:restriction of FV to Cohen ring}
We have 
$$
\sigma(C)\subseteq C\quad\text{and}\quad\varpi C\subseteq VC,
$$
where $\sigma:W_O(k)\to W_O(k)$ and $V:W_O(k)\to W_O(k)$ are the Frobenius and Verschiebung in the $O$-Witt vectors ring.
\end{lemma}
\begin{proof}
For any $i\in I$, we have
$$
\sigma(\iota(x_i))=\iota(\sigma(x_i))=\iota(x_i^q)=\iota(x_i)^q\in C.
$$
Therefore, $\sigma(S)\subset C\subset C_m$ for any $m\geq 1$. Therefore we have $F(C_m)\subset C_m$ by \eqref{def:cnm}. Hence, we have $F(C)\subseteq \bigcap_{m}\sigma(C_m)\subseteq \bigcap_m C_m=C$.
Applying the Verschiebung, we obtain the second inclusion.
\end{proof}

\begin{definition}
We define the $O$-Cohen frame as 
$\mcC_O(k):=(C,\varpi C,\sigma,\dot{\sigma})$, where $\sigma=F:C\to C$ and $\dot{\sigma}:=V^{-1}:\varpi C\to C$ is the inverse of the Verschiebung, which is well-defined by Lemma \ref{lem:restriction of FV to Cohen ring} and the fact that $V$ is injective. 
\end{definition}

We have a natural morphism from $O$-Cohen frame to the $O$-Witt frame: $i:\mcC_O(k)\to \mcW_O(k)$, which defines a functor from the category of $\mcW_O(k)$-windows to $\mcC_O(k)$-windows. We show that it defines an equivalence of the categories. The crucial observation is the following:

\begin{theorem}\label{thm:cohen-equiv}
\begin{altenumerate}
\item The ideal $\varpi C\subset C$ defines an $O$-pd-structure.
\item The functor
$$
i^*:
\{\mcC_O(k)-\text{windows}\}
\to
\{\mcW_O(k)-\text{windows}\}
$$
defines an equivalence of the categories, which remains an equivalence after restricting to the subcategories of nilpotent windows.
\item The category of nilpotent $\mcC_O(k)$-windows is equivalent to the category of formal strict $O$-modules.
\end{altenumerate}
\end{theorem}
\begin{proof}
Part (1) follows from Proposition \ref{prop:O-pd structure extension}. Part (2) follows from the same idea as the proof of \cite[Theorem 4]{Zink_windows}.
We construct the quasi-inverse of the functor $i^*$. For any $\mcW_O(k)$-window $\mcP$, i.e., an $O$-display, by Theorem \ref{thm:ACZ-main}, we obtain a Dieudonn\'e crystal $\mbD_\mcP$ over the category of $O$-pd-thickenings. We recover the $\mcC_O(k)$-windows by evaluating $\mbD_\mcP(C)$ by part (1). Part (3) is a consequence of Theorem \ref{thm:ACZ-main}.
\end{proof}

\bibliographystyle{alpha}
\bibliography{reference}
\end{document}